\documentclass[letterpaper,11pt]{article}
\usepackage[utf8]{inputenc}
\usepackage{amssymb,amsmath,amsthm,mathtools,amsbsy}
\usepackage{subcaption}
\usepackage{graphicx, color}
\usepackage{tikz}
\usetikzlibrary{decorations.text}
\usepackage[margin=1.15in]{geometry}
\mathtoolsset{showonlyrefs}
\def\bbh{\mathbb{H}}
\def\bbr{\mathbb{R}}
\def\bbs{\mathbb{S}}
\def\M{M}

\def\d{\mathrm{d}}

\newcommand{\dm}{n}

\newcommand{\dist}{d}
\newcommand{\calP}{\mathcal{P}}
\newcommand{\supp}{\mathrm{supp}}

\newcommand{\dt}{d_2}
\newcommand{\V}{\mathbf{v}}
\newcommand{\calA}{\mathcal{A}}

\newcommand{\ta}{a}
\newcommand{\sa}{b}

\newcommand{\trho}{\tilde{\rho}}

\newcommand{\Ac}{A_c}
\newcommand{\rhoc}{\rho_c}
\newcommand{\trhoc}{\tilde{\rho}_c}
\newcommand{\Rc}{R_c}

\newcommand{\Ah}{A_h}
\newcommand{\rhoh}{\rho_h}
\newcommand{\trhoh}{\tilde{\rho}_h}
\newcommand{\Rh}{R_h}

\newcommand{\Am}{A_m}

\newcommand{\trhom}{\tilde{\rho}_m}
\newcommand{\fb}{f}

\newcommand{\calK}{\mathcal{K}}
\newcommand{\D}{D}
\newcommand{\R}{R}
\newcommand{\HL}{H_\mathrm{L}}
\newcommand{\HR}{H_\mathrm{R}}

\newcommand{\calPc}{\mathcal{P}^{ac}}
\newcommand{\calPcz}{\mathcal{P}^{ac}_0}
\newcommand{\crho}{c_{\rho}}
\newcommand{\crhoh}{\hat{c}_{\rho}}

\newcommand{\nx}{\hat{n}_x}

\newcommand{\rightlong}{\shortmid \! \longrightarrow}

\def\*#1{\mathbf{#1}}

\newcommand{\norm}[1]{\left\lVert#1\right\rVert}

\theoremstyle{plain}
\newtheorem{thm}{Theorem}[section]
\newtheorem{prop}[thm]{Proposition}
\newtheorem{lem}[thm]{Lemma}
\newtheorem{crly}[thm]{Corollary}

\theoremstyle{remark}
\newtheorem{remark}[thm]{Remark}
\newtheorem{rmk}[thm]{Remark}
\theoremstyle{definition}
\newtheorem{defn}{Definition}[section]

\allowdisplaybreaks

\begin{document}

\title{Equilibria and energy minimizers for an interaction model on the hyperbolic space} 

\author{Razvan C. Fetecau \thanks{Department of Mathematics, Simon Fraser University, 8888 University Dr., Burnaby, BC V5A 1S6, Canada; van@math.sfu.ca}
\and Hansol Park \thanks{Department of Mathematics, Simon Fraser University, 8888 University Dr., Burnaby, BC V5A 1S6, Canada; hansol$\_$park@sfu.ca}}

\maketitle

\begin{abstract}
{We study an intrinsic model for collective behaviour on the hyperbolic space $\bbh^\dm$. We investigate the equilibria of the aggregation equation (or equivalently, the critical points of the associated interaction energy) for interaction potentials that include Newtonian repulsion. By using the method of moving planes, we establish the radial symmetry and the monotonicity of equilibria supported on geodesic balls of $\bbh^\dm$. We find several explicit forms of equilibria and show that one such equilibrium is a global energy minimizer.  We also consider more general potentials and utilize a technique used for $\bbr^\dm$ to establish the existence of compactly supported global minimizers. Numerical simulations are presented, suggesting that some of the equilibria studied here are global attractors. The key tool in our investigations is a family of isometries of $\bbh^\dm$ that we have developed for this purpose.}
\end{abstract}

\textbf{Keywords}: swarming on manifolds, hyperbolic space, energy minimizers, Newtonian potential

%\tableofcontents
\section{Introduction}
\label{sect:intro}
In recent years there has been extensive work on investigating minimizers of nonlocal interaction energies, as motivated by a wide range of applications in a variety of disciplines, such as biology, physics, economics and social sciences \cite{CaCaPa2015, ChFeTo2015, SiSlTo2015, Balague_etalARMA, CaFiPa2017, FrankLieb2018}. With very few exceptions, such works have studied the Euclidean case, where individual particles in $\bbr^\dm$ interact via an interaction potential that depends on the Euclidean distance between them. Since our goal in this paper is to consider nonlocal interaction energies set up on Riemannian manifolds (in particular, on the hyperbolic space), we will present the general setup from the start.

Consider an $n$-dimensional Riemannian manifold $M$, and denote by $\calP(M)$ the set of probability measures on $M$. Define the {\em interaction energy} $E : \calP(M) \to \bbr \cup \{ \infty \}$ by
\begin{equation}
\label{eqn:energy}
E[\rho] = \frac{1}{2} \iint_{M \times M}  K(x,y) \d \rho(x) \d \rho(y),
\end{equation}
where $K: M \times M \to \bbr$ is an {\em interaction potential} that models attractive and repulsive interactions. More precisely, the negative manifold gradient $-\nabla_M K(x,y)$ (the gradient with respect to $x$) of $K$ provides the interaction force that a point mass located at $x$ feels by interacting with a point mass located at $y$ \cite{CaCaPa2015, FeZh2019}.

The interaction energy \eqref{eqn:energy} is closely related to the following evolution equation for a population density $\rho$: 
\begin{equation}
\partial_t \rho-\nabla_M \cdot(\rho \nabla_\M K\ast\rho)=0, \label{eqn:model}
\end{equation}
%\begin{subequations} \label{eqn:model}
%	\begin{gather}
%	\partial_t \rho+\nabla_M \cdot(\rho v)=0, \label{eqn:pde} \\
%	v =-\nabla_\M K\ast\rho. \label{eqn:v}%
%	\end{gather}
%\end{subequations}
where $\nabla_\M \cdot$ denotes the manifold divergence on $\M$. Here, for a time-dependent measure $\rho_t$ on $\M$, the convolution $K \ast \rho_t$ is defined by
\begin{equation} \label{eqn:conv}
	K * \rho_t(x) = \int_M K(x,y) \d \rho_t(y).
	%\int_M K(x,y) \rho(y,t)  \d \mu (y),
\end{equation}
%When $\rho$ is absolutely continuous with respect to the canonical volume form $\mu$ of $\M$ one has $\d \rho = \rho \d \mu$.  

Equation \eqref{eqn:model} is in the form of an active transport equation for the density $\rho$, with a nonlocal velocity field given by
\begin{equation} \label{eqn:v-field}
	\V[\rho] (x,t) =  -\nabla_M K *\rho_t (x).
\end{equation}
Note that $\displaystyle{\V[\rho] = - \nabla_M \frac{\delta E[\rho]}{\delta \rho}}$, which sets (at least formally) equation \eqref{eqn:model} as the gradient flow of the energy \eqref{eqn:energy}. One can check by a direct calculation that indeed, formally, the energy $E$ decays in time along solutions of equation \eqref{eqn:model} \cite{FeZh2019}. Equilibria (steady states) of the aggregation model \eqref{eqn:model} are the critical points of the interaction energy, and from a dynamical point of view, we expect that solutions to model \eqref{eqn:model} approach asymptotically as $t \to \infty$, local minimizers of the energy. 

There has been extensive literature in recent years on interaction energies of type \eqref{eqn:energy}, with the vast majority of the works focusing on the case $M=\bbr^\dm$. One approach is to use the techniques developed in \cite{AGS2005} and formulate the gradient flow of the energy $E$ on the space of probability measures with finite second (or other order) moments, equipped with the $2$-Wasserstein metric  \cite{CaMcVi2006, Figalli_etal2011, WuSlepcev2015}. Another approach is to use direct methods of the calculus of variations and establish the existence of global minimizers (ground states) of the interaction energy \cite{ChFeTo2015, SiSlTo2015, FrankLieb2018}. There are also various studies of quantitative and qualitative properties of minimizers, such as their regularity \cite{CaDeMe2016} and compactness and dimensionality of their support \cite{CaCaPa2015, Balague_etalARMA, CaFiPa2017}. 

Model \eqref{eqn:model} can exhibit a very diverse range of swarming or self-organized behaviour. We refer to \cite{KoSuUmBe2011, Brecht_etal2011, FeHuKo11, BaCaLaRa2013, CaHuMa2014} for explicit calculations and numerical illustrations of equilibria for the model in $\bbr^\dm$, which include aggregations on disks, annuli, rings, and soccer balls. The model on manifolds is far less investigated, but nevertheless, a similarly diverse set of equilibria was demonstrated in \cite{FeZh2019} for the hyperbolic plane and also, interesting aggregation patterns were shown in \cite{FeHaPa2021} for the model set up on the rotation group $SO(3)$. Applications of model \eqref{eqn:model} are numerous, e.g., to flocking and swarming of biological organisms \cite{CarrilloVecil2010}, material science and granular media \cite{CaMcVi2006}, self-assembly of nanoparticles \cite{HoPu2005}, robotics \cite{Gazi:Passino,JiEgerstedt2007} and opinion formation \cite{MotschTadmor2014}. 

In our study we will follow the {\em intrinsic} approach pursued in various recent papers \cite{FeZh2019, FeHaPa2021, FePaPa2021} and assume that $K(x,y)$ only depends on the {\em geodesic} distance $d(x,y)$ between the two points. By an abuse of notation we write $K(x,y) = K(d(x,y))$. In contrast, when $M$ comes with a canonical embedding in a larger Euclidean space (e.g., a surface in $\bbr^3$), $K(x,y)$ can be assumed to depend on the {\em Euclidean} distance $|x -y|$ in the ambient space between $x$ and $y$; we refer to this case as the {\em extrinsic} approach. The two approaches yield very different models (unless $M=\bbr^\dm$, in which case the two models are the same). For extrinsic models we refer to \cite{WuSlepcev2015, PaSl2021} for studies on well-posedness by gradient flow techniques, and to various works on the emergent behaviours of the extrinsic models on manifolds, e.g., on sphere \cite{HaChCh2014},  unitary matrices \cite{Lohe2009, HaRy2016, HaKoRy2018}, hyperbolic space \cite{Ha-hyperboloid}, and Stiefel manifolds \cite{HaKaKi2022}.  Relevant to the present paper, we note the work in \cite{Ha-hyperboloid}, where the authors study synchronization behaviour on the hyperbolic space with an extrinsic quadratic potential (i.e., $K(x,y) = |x-y|^2$) for the discrete analogue of \eqref{eqn:model}. A related study is done in \cite{RiLoWi2018}. The emergence of self-synchronization has been widely studied in the literature, due to its numerous occurrences in biological, physical and chemical systems (e.g., flashing of fireflies, neuronal synchronization in the brain, quantum synchronization) -- see \cite{Lohe2009, Ha-hyperboloid} and references therein.

For intrinsic models, the sign of $K'$ determines whether the interactions are repulsive or attractive in nature. Indeed, provided $x$ and $y$ are not in the cut locus of each other, by chain rule one gets
\[
-\nabla_M K(\dist(x,y)) = K'(\dist(x,y)) \frac{\log_x y}{\dist(x,y)},
\]
where $\log_x y$ denotes the Riemannian logarithm map (i.e., the inverse of the Riemannian exponential map) on $M$. Therefore, by interacting with the point mass at $y$, the point mass at $x$ is driven by a force of magnitude proportional to $|K'(d(x,y))|$, to move either towards $y$ (if $K'(d(x,y)) >0$) or away from $y$  (if $K'(d(x,y)) < 0$). To obtain non-trivial swarming behaviours, repulsion and attraction must balance each other. Typically, such interaction potentials have short-range repulsion and long-range attraction, i.e., $K'(r)<0$ for small $r$ and $K'(r)>0$ for large $r$ (see Figure \ref{Fig:K}). 

The interaction energy for intrinsic models is invariant under isometries. Indeed, if $f:M\to M$ is an isometry of $M$ and $f_\# \rho$ is the push-forward by $f$ of the measure $\rho$, then it is immediate to show that 
\begin{equation}
\label{eqn:E-inv}
E[f_\# \rho] = E[\rho].
\end{equation}
Consequently, the energy minimizers can only be expected to be unique up to isometries.

In this paper we are exclusively concerned with studying the nonlocal aggregation model on the hyperbolic space $\bbh^\dm$, represented as a one-sheeted hyperboloid in $\bbr^{\dm+1}$ endowed with the Lorentzian inner product.  Based on the Lorentz transform, we construct a family of isometries of $\bbh^\dm$ and a binary operation on $\bbh^\dm \times \bbh^\dm$, along with a concept of coordinates on $\bbh^\dm$, that are the key tools in the results we establish in the present work. 

In Section \ref{sect:Newt-repulsion} we consider attractive-repulsive interaction potentials that include Newtonian repulsion, i.e., potentials for which the repulsion component is given by the Green's function of the negative Laplacian on $\bbh^\dm$. We note here that the Green's function on $\bbh^\dm$ has an explicit expression  \cite{CohlKalnins2012} -- see \eqref{eqn:Phi}. For the model on $\bbr^\dm$, potentials in this form have been given extensive attention in recent years \cite{BertozziLaurentLeger, FeHuKo11, FeHu13, CaHuMa2014, ShuTadmor2021}, and a potential with Newtonian repulsion was also considered on the hyperbolic plane \cite{FeZh2019}.  In this section we only consider equilibria that are absolutely continuous with respect to the canonical volume measure on $\bbh^\dm$. This choice is motivated by results in the Euclidean space $\bbr^\dm$ which have established  that for interaction potentials with Newtonian repulsion, the local minimizers of the interaction energy are absolutely continuous with respect to the Lebegue measure \cite{CaHuMa2014}. 

We study equilibria supported in geodesic balls of $\bbh^\dm$ and show that such equilibria must be radially symmetric and monotone with respect to the centre of the ball. To show this result (Theorem \ref{thm:mov-planes}) we use the method of moving planes, a well-known method for studying qualitative properties of positive solutions of elliptic equations \cite{GidasNiNirenberg}. More recently, the method found applications in integral equations \cite{ChenLiOu,FeHu13}. The challenge here is to adapt the method to the hyperboloid; to the best of our knowledge this the first application of the method of moving planes to integral equations on manifolds. We also find several explicit forms of radially symmetric equilibria, corresponding to particular choices of the attractive part of the potential. For one such attractive potential, we show that the equilibrium we computed is in fact the global energy minimizer. We also present various numerical simulations in Section \ref{sect:numerics}. 

In Section \ref{sect:existence} we investigate the existence of compactly supported global minimizers of the energy \eqref{eqn:energy} for $M=\bbh^\dm$. As we apply and follow closely the technique used by Ca{\~n}izo {\em et al.} \cite{CaCaPa2015} for the Euclidean case, we only highlight the main differences from their work. In other words, we are not aiming for a self-contained presentation in Section \ref{sect:existence}, as some of results from \cite{CaCaPa2015} would transfer with no modifications to $\bbh^\dm$. Nevertheless, we show how the family of isometries and the concept of coordinates on $\bbh^\dm$ that we developed, enable a similar approach as for the Euclidean case. In particular, minimizers cannot have too large gaps, as otherwise, an isometric displacement of one of the components would decrease the energy (Lemma \ref{lemma:separation}).

Finally, we provide some numerical results in Section \ref{sect:numerics}. Using various interaction potentials, we illustrate the dynamical evolution  to steady states of radially symmetric solutions to \eqref{eqn:model}. The equilibria we present are qualitatively different in terms in their monotonicity, in consistency with the analytical findings. Numerics suggest these equilibria are global attractors for the dynamics (at least for radial initial densities).

The summary of the paper is the following. Section \ref{sect:prelim} provides key background needed for the paper, on the variational formalism and on the geometrical properties of $\bbh^\dm$. Section \ref{sect:Newt-repulsion} contains analytical results for interaction potentials that contain Newtonian repulsion (characterization of monotonicity, explicit forms for equilibria, convexity of energy). In Section \ref{sect:existence} we establish the existence of compactly supported global minimizers, following the approach for $\bbr^\dm$ from \cite{CaCaPa2015}. Proof of various lemmas are deferred to the Appendix.

%%%%%%%%%%

\section{Preliminaries}
\label{sect:prelim}
In this section we provide some background on the variational formalism used in our paper, as well as on the various geometrical properties of $\bbh^\dm$, in particular on a family of isometries of the hyperbolic space that plays a key role in our studies.

%%%%%

\subsection{Variational approach: Euler-Lagrange equations}
\label{subbsect:E-L}
For the interaction energy set up on the Euclidean space $\bbr^\dm$, there exists a well-established approach to characterize its local minimizers (by Euler-Lagrange equations) in the topology of transport distances \cite{Balague_etalARMA}. As this approach extends immediately to arbitrary manifolds, we will simply list below the main notations and results.

Take $p$ to be a fixed point in $M$, and denote by $\calP_2(M)$ the set of probability measures that have finite moments of order $2$ (with respect to $p$), i.e.,
\[
\calP_2(M) = \left \{ \mu \in \calP(M): \int_M \dist^2(x,p) \d \mu(x) < \infty \right \}.
\]
The $2$-Wasserstein distance between two measures  $\mu,\sigma \in \calP_2(M)$ is given by:
\[
	\dt^2 (\mu,\sigma) = \inf_{\pi \in \Pi(\mu,\sigma)} \iint_{M \times M} \dist^2(x,y) \d\pi(x,y),
\]
where $\Pi(\mu,\sigma) \subset \calP (M \times M)$ is the set of transport plans between $\mu$ and $\sigma$, i.e., the set of probability measures on the product space $M\times M$ with first and second marginals $\mu$ and $\sigma$, respectively. 

A probability measure $\mu \in \calP_2(M)$ is a local minimizer of $E$ with respect to the $\dt$-distance provided there exists $\epsilon>0$ such that $E[\sigma] \geq E[\mu]$, for all $\sigma \in B(\mu,\epsilon)$, where $B(\mu,\epsilon)$ denotes the open ball (in $\dt$) centred at $\mu$ and of radius $\epsilon$. 

%\begin{lem}[\cite{Balague_etalARMA}, Theorem 4] Given an interaction potential $W$ satisfying (1) $W$ is bounded from below, (2)  $W$ is lower semicontinuous(l.s.c.). Let us consider $\mu\in\mathcal{P}_2(\bbr^N)$ a local minimizer of $E$ with respect to $\dt$ such that $E[\mu]<\infty$. Then,
%
%\noindent(i) $(W*\mu)(x)=2E[\mu]$ $\mu$-a.e.
%
%\noindent(ii) $(W*\mu)(x)\leq 2E[\mu]$ for all $x\in\mathrm{supp}(\mu)$.
%
%\noindent(iii) $(W*\mu)(x)\geq2E[\mu]$ for a.e. $x\in\bbr^N$.
%\end{lem}

We will pay particular attention in this paper to probability measures that are absolutely continuous with respect to the canonical volume measure $\d x$; we denote this space by $\calPc(M)$.  
%Also denote $\calPc_2(M) = \calPc(M) \cap \calP_2(M)$. 
Throughout the paper, we will refer to an absolutely continuous measure directly by its density $\rho$, and by abuse of notation write $\rho \in \calPc(M)$ to mean $\d \rho (x) = \rho(x) \d x \in \calPc(M)$. Finally, we denote by $\calPc_c(M)$ the space of absolutely continuous probability measures with compact support in $M$.

The following result, stated originally for $M=\bbr^\dm$, extends trivially to arbitrary Riemannian manifolds.
\begin{prop}[\cite{Balague_etalARMA}, Theorem 4 and Remark 4]\label{lem:loc-min}
Assume the interaction potential $K$ is bounded from below and lower semicontinuous. Let $\rho \in \calPc(M) \cap \calP_2(M)$ be a local minimizer of $E$ with respect to $\dt$. Then, 
\begin{subequations}
\label{eqn:min-cond}
\begin{alignat}{1}
\Lambda(x)&=\lambda,\qquad\text{ for a.e. } x \in \supp(\rho), \label{eqn:min-cond-1} \\
\Lambda(x) & \geq\lambda, \qquad\text{ for a.e. on } M \setminus \supp(\rho), \label{eqn:min-cond-2}
\end{alignat}
\end{subequations}
where
\begin{align}\label{eqn:Lambda}
\Lambda(x):=\int_{M }K(x, y) \rho(y) \d y,
\end{align}
and $\lambda = 2 E[\rho] $ is a constant.
\end{prop}

Condition \eqref{eqn:min-cond-1} is shown by taking perturbations of $\rho$ supported in $\supp(\rho)$. The condition is equivalent to $\rho$ being a critical point of $E$, at which the first variation of the energy vanishes. The constant $\lambda$, which plays the role of a Lagrange multiplier in the derivation of \eqref{eqn:min-cond-1}, has a physical interpretation \cite{BeTo2011}: it represents the energy per unit mass felt by a point mass at position $x$ due to interaction with all points in $\supp(\rho)$. 

We also note that critical points of the interaction energy represent equilibrium solutions (steady states) of the aggregation model \eqref{eqn:model}. This can be inferred from the fact that the velocity $\V(x) = - \nabla_M K \ast \rho (x) = -\nabla_M \Lambda(x)$ vanishes for $x\in \supp(\rho)$, given that $\Lambda$ is constant in $\supp(\rho)$. We will use these concepts (critical point of the energy, equilibrium solution, steady state) interchangeably in the sequel.

The necessary condition \eqref{eqn:min-cond-2} is found by taking perturbations of $\rho$ that can be supported anywhere in $M$, in particular outside $\supp(\rho)$. The interpretation of  \eqref{eqn:min-cond-2} is that transporting mass from the support of $\rho$ into its complement increases the total energy (hence $\rho$ is a local minimizer, being in a favourably energetic state).
%%%%%

\subsection{Geometric properties of the hyperbolic space $\bbh^\dm$}
\label{sec:2.2}
In this part, we present some geometric properties of the hyperbolic space $\bbh^\dm$ which will be used in the paper.

%\subsubsection{Differentiation and integration on $\bbh^\dm$}
\subsubsection{Hyperboloid model, polar coordinates and notations}
We consider the space $\bbr^{\dm+1}$ endowed with the (negative) Lorentzian inner product:
\begin{equation}
\label{eqn:iproduct-L}
\langle x,y \rangle = x_0 y_0 - x_1 y_1 - \cdots - x_\dm y_\dm,
\end{equation}
where $x = (x_0,x_1, \dots, x_\dm)$, $y=(y_0,y_1,\dots,y_\dm) \in \bbr^{\dm+1}$. The inner product \eqref{eqn:iproduct-L} induces the (complex-valued) norm defined by
\begin{equation}
\label{eqn:norm-L}
\| x \| = \langle x,x \rangle^{\frac{1}{2}}, \qquad \text{ for } x \in \bbr^{\dm+1}.
\end{equation}

We use the hyperboloid model of the hyperbolic space and take:
\[
\bbh^\dm:= \{ x\in\bbr^{\dm+1}: -x_0^2 + x_1^2 + \dots + x_\dm^2 = -1, x_0>0 \} \subset \bbr^{\dm+1},
\]
or, written differently,
\[
\bbh^\dm = \{x\in\bbr^{\dm+1}: x^\top \eta x=-1, x_0>0\},
\]
where $\eta=\mathrm{diag}(-1, \underbrace{1, 1, \cdots, 1}_{\dm-\text{times}})$ is a matrix of size $(\dm+1)\times(\dm+1)$. 

The geodesic distance $\dist(x, y)$ on $\bbh^\dm$ between two points $x$ and $y$, is given by:
\begin{equation}
\label{eqn:geod-d}
\dist(x, y)=\cosh^{-1}(\langle x,y \rangle) = \cosh^{-1}(x_0y_0-x_1y_1-\cdots-x_\dm y_\dm).
\end{equation} 
For $x \in \bbh^\dm$ and $R>0$, we denote by $B_R(x)$ the open geodesic ball in $\bbh^\dm$ centred at $x$ of radius $R$, i.e.,
\[
B_R(x) = \{ y \in \bbh^\dm: \dist(x,y) <R\}.
\]

We use the hyperbolic polar coordinates on $\bbh^\dm$, given by:
\begin{align}\label{a-a-13}
(\theta, \xi)\in [0, \infty)\times \bbs^{\dm-1}\mapsto (\cosh\theta, \sinh\theta \, \xi)\in\bbh^\dm\subset\bbr^{\dm+1},
\end{align}
where $\bbs^{\dm-1}$ denotes the $(\dm-1)$-dimensional unit sphere embedded in $\bbr^\dm$. We denote by $v$ the vertex of $\bbh^\dm$, corresponding to $\theta =0$, i.e., $v=(1,0,\dots,0) \in  \bbr^{\dm+1}$. The vertex $v$ plays the role of the origin in the Euclidean space.

In polar coordinates, the Laplace-Beltrami operator $\Delta_{\bbh^\dm}$ can be written as:
\begin{align}
\begin{aligned}\label{a-a-14}
\Delta_{\bbh^\dm}f(\theta, \xi)&=\sinh^{1-\dm}\theta \, \partial_\theta\left(\sinh^{\dm-1}\theta\partial_\theta f\right)+\sinh^{-2}\theta\Delta_\xi f\\
&=\partial_{\theta\theta}f+(\dm-1)\coth\theta \, \partial_\theta f+\sinh^{-2}\theta\Delta_\xi f,
\end{aligned}
\end{align}
where $\Delta_\xi$ is the Laplace--Beltrami operator on $\bbs^{\dm-1}$, with respect to the variable $\xi$. Also, the canonical volume form on $\bbh^\dm$ in these coordinates is given by
\begin{equation}
\label{eqn:V-form}
\d x=\sinh^{\dm-1}\theta \, \d\theta \d\sigma_{\xi}.
\end{equation}

We denote by $\alpha(\dm)$ the volume of the $\dm$-dimensional unit ball in $\bbr^\dm$; consequently, $\dm \alpha(\dm)$ is the surface area of $\bbs^{\dm-1}$ in $\bbr^\dm$. Finally, we denote the standard Euclidean norm in $\bbr^\dm$ by $| \cdot |$, i.e.,
\[
|w|= \left( w_1^2+\cdots+w_\dm^2 \right)^{\frac{1}{2}}, \qquad \text{ for } w=(w_1,\dots,w_\dm) \in \bbr^\dm.
\]

%%%

\subsubsection{Lorentz transform}
For a fixed vector $w=(w_1,\dots,w_\dm) \in\bbr^\dm$ with $|w|<1$, consider the following matrix $B(w)$ of size $(\dm+1)\times (\dm+1)$ with indices $0, 1, \dots, \dm$:
\begin{align}\label{LB}
B(w)_{00}=\gamma,\quad B(w)_{i0}=B(w)_{0i}=-\gamma w_i,\quad B(w)_{ij}=\delta_{ij}+(\gamma-1)\frac{w_iw_j}{|w|^2},\quad \forall i, j=1, 2, \cdots, \dm,
\end{align}
where $\gamma=(1-|w|^2)^{-1/2}$. Equivalently, the matrix expression of $B(w)$ is
\[
B(w)=\begin{bmatrix}
\gamma&-\gamma w^\top\\
-\gamma w& I_\dm +(\gamma-1)\frac{ww^\top}{|w|^2}
\end{bmatrix}.
\]
We denote this matrix by $B$, since the corresponding concept is the Lorentz boost. 

\begin{lem}\label{Leme}
For a fixed $w\in\bbr^\dm$ with $|w|<1$, $B(w)$ maps $\bbh^\dm$ into $\bbh^\dm$, i.e., 
\[
y=B(w) x\in\bbh^\dm,\qquad \text{ for all } x\in\bbh^\dm\subset\bbr^{\dm+1}. 
\]
Here, $y$ is the matrix product of $B(w)\in\bbr^{(\dm+1)\times(\dm+1)}$ and $x\in\bbr^{\dm+1}$.
\end{lem}
\begin{proof}
By direct calculations (see Appendix \ref{sec:app.a} for details) one can show:
\[
y_0^2-y_1^2-\cdots-y_\dm^2 = x_0^2-x_1^2-\cdots-x_\dm^2,
\]
which implies the desired result.
\end{proof}

From the proof of Lemma \ref{Leme}, for any fixed $w\in\bbr^\dm$ with $|w|<1$, it holds that:
\[
x^\top \eta x=(B(w)x)^\top \eta B(w)x=x^\top(B(w)^\top \eta B(w))x, \qquad \forall~x\in\bbr^{\dm+1}.
\]
This yields
\[
x^\top(\eta-B(w)^\top \eta B(w))x=0,\qquad\forall~x\in\bbr^{\dm+1},
\]
and hence,
\begin{align}\label{LB-1}
B(w)^\top \eta B(w)=\eta.
\end{align}
Identity \eqref{LB-1} holds for all $w\in\bbr^\dm$ with $|w|<1$.

\begin{lem}\label{Lem-iso1}
For a fixed $w\in\bbr^\dm$ with $|w|<1$, the map $F_w:\bbh^\dm\to\bbh^\dm$ defined by 
\begin{equation}
\label{eqn:Fw}
F_w(x):=B(w)x,
\end{equation}
is an isometry. In particular, $F_0$ is the identity map on $\bbh^\dm$.
\end{lem}
\begin{proof}
The fact that $F_0$ is the identity map follows immediately from $B(0) = \text{Id}$. The general result follows from a simple calculation that uses \eqref{LB-1}; see Appendix \ref{sec:app.a} for details.
\end{proof}

\begin{lem}\label{La.3}
Let $x\in\bbh^\dm$ and $w\in\bbr^\dm$ with $|w|<1$. Then, the following are equivalent: 
\begin{enumerate}
\item $F_w(1, 0, \cdots, 0)=(x_0, -x_1, \cdots, -x_\dm)$,
\item $F_w(x)=(1, 0, \cdots, 0)$, 
\item $w_i=\frac{x_i}{x_0}$, for all $1\leq i\leq \dm$.
\end{enumerate}
\end{lem}
\begin{proof}
The proof is provided in Appendix \ref{sec:app.a}.
\end{proof}

%%%

\subsubsection{Translation on $\bbh^\dm$}
In this part, we define the translation on the hyperbolic space using the isometry $F_w$ defined in \eqref{eqn:Fw}. For any $x=(x_0, x_1, \cdots, x_\dm)\in\bbh^\dm$, we define a $\dm$-dimensional vector $\hat{x}$ by
\[
\hat{x}:=\left(\frac{x_1}{x_0},\cdots, \frac{x_\dm}{x_0}\right).
\]
From the definition of $\hat{x}$ and Lemma \ref{La.3}, we have that
\[
F_{\hat{x}}(x)=v \qquad\text{ and }\qquad F_{\hat{x}}(v)=(x_0, -x_1, \cdots, -x_\dm).
\]

%Also, $(x_0, -x_1, \cdots, -x_\dm)$ replaces the $-x$(the inverse of $x$) of the Euclidean space. From these observations, we define the translation operator on $\bbh^\dm$ as follows.
\begin{defn}
\label{defn:+prime}
We define a binary operation $+':\bbh^\dm\times\bbh^\dm\to\bbh^\dm$ as follows:
\[
x+'y:=F_{-\hat{y}} (x) \in\bbh^\dm, \qquad \forall~x, y\in \bbh^\dm.
%x+'y:=B(-\hat{y})x\in\bbh^\dm, \qquad \forall~x, y\in \bbh^\dm.
\]
\end{defn}

Since $F_w$ is an isometry for all $w\in\bbr^\dm$ with $|w|<1$, we know that the map $x\mapsto x+'y$ is also an isometry for all $y\in\bbh^\dm$. The following lemma provides the coordinate expression of $x+'y$.

\begin{lem}\label{Lemexf}
Let $x, y\in\bbh^\dm$, then we have
\begin{align}
\begin{aligned}\label{a-a-10}
&(x+'y)_0=x_0y_0+x_1y_1+\cdots+x_\dm y_\dm,\\
&(x+'y)_j=x_0y_j+x_j+\frac{y_j}{y_0+1}(x_1y_1+\cdots+x_\dm y_\dm), \qquad\forall~1\leq j\leq \dm.
\end{aligned}
\end{align}
\end{lem}

\begin{proof}
The proof follows from direct calculations, see Appendix \ref{sec:app.a}.
\end{proof}
\begin{remark}
We can also express $(x+'y)_j$ for any $1\leq j\leq \dm$, as 
\begin{align}\label{a-a-11}
(x+'y)_j=x_0y_j+x_j+\frac{y_j}{y_0+1}((x+'y)_0-x_0y_0)=x_j+\frac{y_j}{y_0+1}((x+'y)_0+x_0).
\end{align}
\end{remark}

We investigate now the algebraic structure of the operator $+'$. 
\medskip

\noindent{\em (1) Identity of $+'$.} Recall that $v=(1, \underbrace{0, \cdots, 0}_{\dm-\text{times}})$ denotes the vertex of $\bbh^\dm$. Since $\hat{v}$ is the $\dm$-dimensional zero vector, we can easily find that
\[
x+'v=F_0(x)=x,
\]
since $F_0$ is the identity map on $\bbh^\dm$. Now, we calculate $v+'x$. From the definition of $+'$, we have
\[
(v+'x)_0=v_0x_0+v_1x_1+\cdots+v_\dm x_\dm=x_0
\]
and
\[
(v+'x)_j=v_0x_j+v_j+\frac{x_j}{x_0+1}(v_1x_1+\cdots+v_\dm x_\dm)=x_j, \qquad  \text{ for } j =1,\dots,\dm.
\]
This implies that
\[
v+'x=x.
\]
We conclude that $v$ is an identity element for $+'$. The uniqueness of identities can be shown easily, so $v$ is the unique identity element of $+'$.\\

\noindent{\em (2) Inverse of $+'$.} We want to find the inverse $-x$ of $x$, which satisfies
\[
x+'(-x)=(-x)+'x=v.
\]
For  $x=(x_0, x_1, \cdots, x_\dm)$, set:
\begin{equation}
\label{eqn:inverse}
-x:=(x_0, -x_1, -x_2,\cdots, -x_\dm).
\end{equation}
It is easy to check that $-x$ defined above is indeed the inverse of $x$. In polar coordinates, if $(\theta, \xi)\in [0, \infty)\times \bbs^{\dm-1}$ are the coordinates of $x$, then $-x$ has coordinates $(\theta, -\xi)$. So one can see that the inverse operator is quite natural. We will prove the uniqueness of the inverse later (see Corollary \ref{crly:inverse-u}).

Using the definition \eqref{eqn:inverse} of the inverse, we define the operation $-'$ by
\begin{equation}
\label{eqn:-prime}
x-'y:=x+'(-y), \qquad\forall~x, y\in\bbh^\dm.
\end{equation}
We have the following simple lemma which will be used in further calculations.
\begin{lem}\label{msl}
Let $x$ and $y$ be two points on $\bbh^\dm$. Then we have the following relation:
\[
-(x+'y)=(-x)+'(-y).
\]
\end{lem}
\begin{proof}
The proof is provided in Appendix \ref{sec:app.a}.
\end{proof}

\noindent{\em (3) Non-commutativity of $+'$.} Using \eqref{a-a-10}, we can show easily that $+'$ is not commutative.\\

\noindent{\em (4) Non-associativity of $+'$.} By direct (but tedious) calculations, one can check that in general, $((x+'y)+'z)_0\neq (x+'(y+'z))_0$. Hence, the operation $+'$ is not associative.\\

To prove the uniqueness of the inverse, we provide the following lemma.
\begin{lem}\label{La.8}
For any $x, y\in \bbh^\dm$, we have
\[
x=(x+'y)-'y.
\]
\end{lem}

\begin{proof}
The proof is provided in Appendix \ref{sec:app.a}.
\end{proof}
From the previous lemma, we have the uniqueness of the inverse for $+'$.
\begin{crly}
\label{crly:inverse-u}
Let $x, y, z\in\bbh^\dm$ that satisfy
\[
x+'y=v,\qquad\text{and}\qquad x+'z=v.
\]
Then, we have $y=z$.
\end{crly}
\begin{proof}
From Lemma \ref{La.8}, we have
\[
(x+'y)-'y=x,\qquad  (x+'z)-'z=x.
\]
If we use $x+'y=x+'z=v$, we get
\[
v-'y=x,\qquad v-'z=x.
\]
This relation directly yields $y=z$.
\end{proof}

Note that one can express the geodesic distance on $\bbh^\dm$ using $-'$. Let $x, y\in \bbh^\dm$, $x=(x_0, x_1, \cdots, x_\dm)$ and $y=(y_0, y_1, \cdots, y_\dm)$. Then we have
\begin{equation}
\label{eqn:xmy-z}
(x-'y)_0=x_0y_0-x_1y_1-\cdots-x_\dm y_\dm,
\end{equation}
and by \eqref{eqn:geod-d},
\begin{equation}
\label{eqn:zcomp-dist}
(x-'y)_0=\cosh\Big(\dist(x, y)\Big).
\end{equation}

%%%

\subsubsection{Coordinate grid on $\bbh^\dm$}
In this part, we define a coordinate grid on $\bbh^\dm$. In the Euclidean space, the coordinate grid is set by the hyperplanes $x_i=\text{const.}$, for all coordinates $x_i$. However, we should be careful when we define this concept on $\bbh^\dm$.

%First, we define a projection map $\pi_k:\bbh^\dm\to\bbr$ in the following way.
For a fixed $k\in\{1, 2, \cdots, \dm\}$, we define the hypersurface:
\[
P_k(0):=\{x\in\bbh^\dm: x_k=0\}.
\]
For $\ta \in\bbr$, define
\begin{equation}
\label{eqn:uk}
u_k(\ta):=\left(\cosh \ta, 0, \cdots,0, \sinh \ta,0, \cdots, 0\right)\in\bbh^\dm,
\end{equation}
where $\cosh \ta$ is the $0^{th}$ index and $\sinh \ta$ is the $k^{th}$ index. Then consider translates of $P_k(0)$ by $u_k(\ta)$:
\begin{equation}
\label{eqn:translation-k}
P_k(\ta):=P_k(0)+'u_k(\ta),\qquad \ta \in\bbr.
\end{equation}

We refer to Figure \ref{Fig0} for an illustration of $P_k(0)$ and some of its translates for the $2$-dimensional hyperbolic space. 
\begin{figure}[thb]
 \begin{center}
 \begin{tabular}{cc}
 \hspace{-1.8cm}  \includegraphics[width=0.68\textwidth]{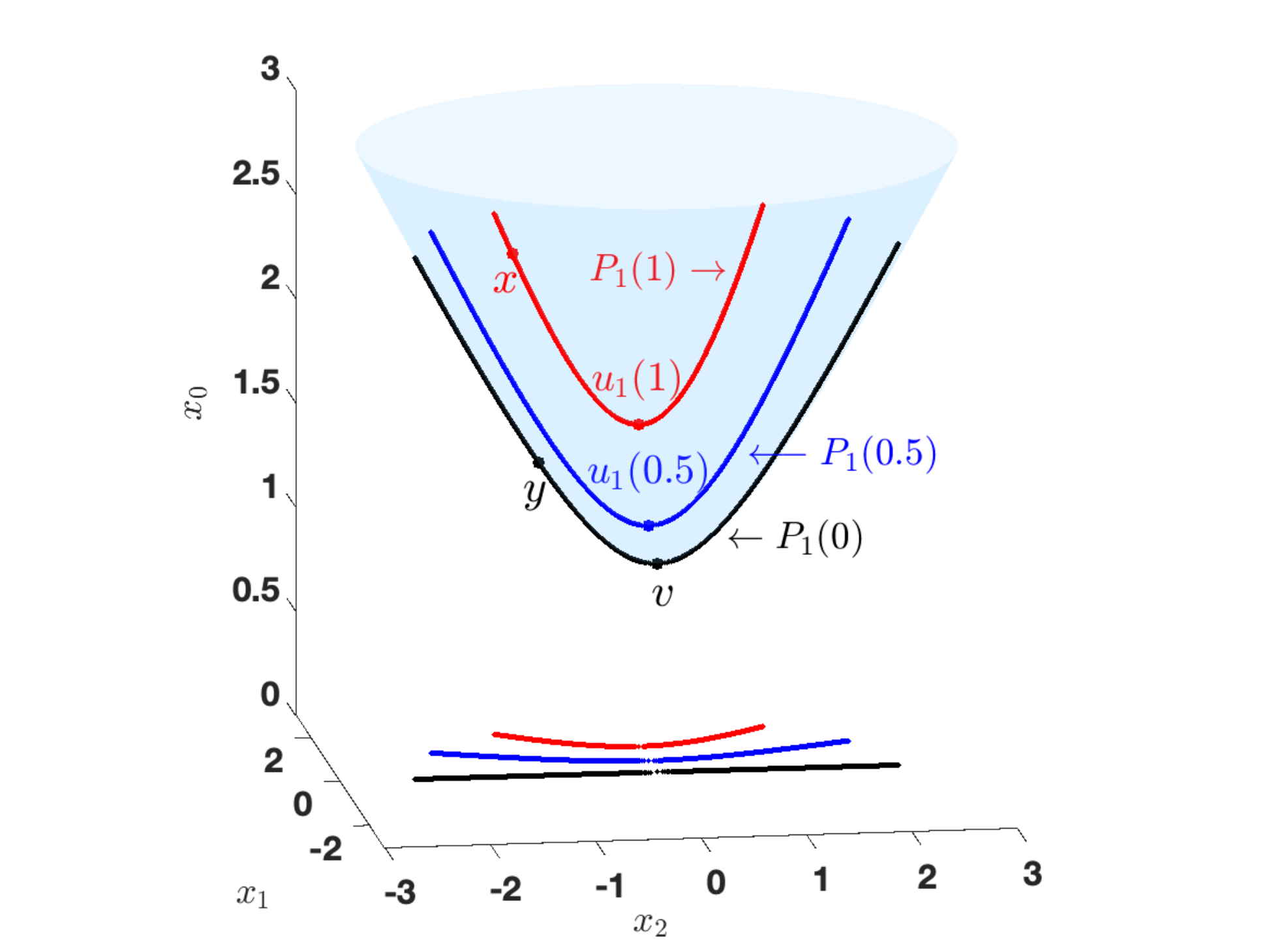}  & 
 \hspace{-1.8cm} \includegraphics[width=0.52\textwidth]{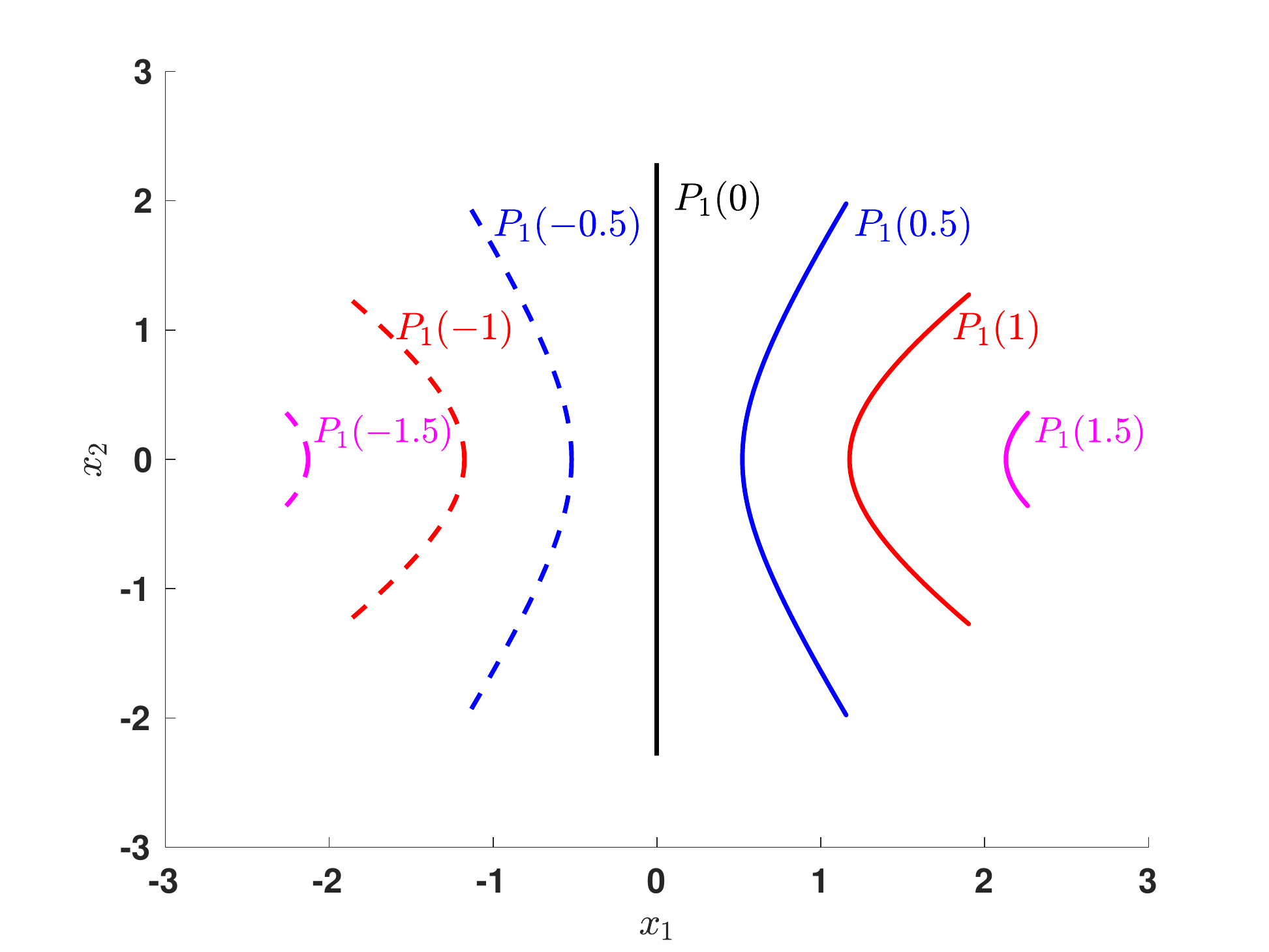}  \\
 (a) & (b) 
 \end{tabular}
 \begin{center}
 \end{center}
\caption{Illustration of the translation \eqref{eqn:translation-k} on the $2$-dimensional hyperbolic space $\bbh^2$. (a) The curve $P_1(0)$ and its translates $P_1(\ta)$ for $\ta=0.5$ and $\ta=1$. The points $u_1(\ta)$ are also indicated on the plot. (b) Projections of $P_1(\ta)$ on the $x_1x_2$-plane for various values of $\ta$.}
\label{Fig0}
\end{center}
\end{figure}

\begin{lem}\label{l3.2}
Let $x\in P_k(0)$ and $\ta, \sa \in\bbr$, then we have
\[
(x+'u_k(\ta))+'u_k(\sa)=x+'u_k(\ta+\sa).
\]
\end{lem}
\begin{proof}
The proof is provided in Appendix \ref{sec:app.a}.
\end{proof}

\begin{rmk}
(1) If we put $x=v$, we have
\begin{align*}
u_k(\ta)+'u_k(\sa)=u_k(\ta+\sa), \qquad\forall~ \ta, \sa \in\bbr.
\end{align*}
(2) Since $u_k(0)=v$, we also have
\[
u_k(\ta)=-u_k(-\ta), \qquad\forall~\ta \in\bbr.
\]
\end{rmk}

The following lemma is key for the concept of coordinate grid on $\bbh^\dm$. In particular, it shows that the family \eqref{eqn:translation-k} of translates of $P_k(0)$ covers the entire $\bbh^\dm$; this is similar to how translating a hyperplane $x_i = 0$ along the $x_i$-axis covers the Euclidean space.

\begin{lem}\label{l3.4}
For any $x\in\bbh^\dm$ there exists a unique pair $(\ta, y)\in\bbr\times P_k(0)$ such that
\begin{align}\label{decomp}
x=y+'u_k(\ta).
\end{align}
Furthermore, $\ta$ and $y$ can be expressed as
\begin{equation}
\label{eqn:ay}
\ta =\tanh^{-1}\left(\frac{x_k}{x_0}\right),\qquad y=x-'u_k(\ta).
\end{equation}
\end{lem}
\begin{proof}
The proof is provided in Appendix \ref{sec:app.a}.
\end{proof}

To illustrate the decomposition \eqref{decomp}, we show in Figure \ref{Fig0}(a) a point $x \in P_1(1) \subset \bbh^2$ and its corresponding $y \in P_1(0)$ such that $x = y+' u_1(1)$. In summary, any point $x\in \bbh^\dm$ lies on the hypersurface $P_k(\ta)$ with $\ta$ given by \eqref{eqn:ay}, where $P_k(\ta)$ is the translation of $P_k(0)$ by $u_k(\ta)$. One can think of $u_k(\ta)$ as the $k$-th coordinate of $x$. Given the uniqueness of the decomposition
\eqref{decomp}-\eqref{eqn:ay}, we will be using in the sequel the map $\pi_k :\bbh^\dm\to\bbr$ defined by 
\begin{equation}
\label{eqn:pik}
\pi_k(x)= \tanh^{-1}\left(\frac{x_k}{x_0}\right).
\end{equation}

\begin{crly}\label{c3.5}
Let $x\in\bbh^\dm$ with $x\in\pi_k^{-1}(\ta)$ for some $\ta \in\bbr$. Then we have
\[
x+'u_k(\sa)\in\pi_k^{-1}(\ta+\sa),\qquad\text{or equivalently,}\quad \pi_k(x+'u_k(\sa))=\ta+\sa.
\]
\end{crly}

\begin{proof}
From Lemma \ref{l3.4}, there exists $y\in P_k(0)$ such that
\[
x=y+'u_k(\ta).
\]
Then, by Lemma \ref{l3.2}, we get
\[
x+'u_k(\sa)=(y+'u_k(\ta))+'u_k(\sa)=y+'u_k(\ta+\sa).
\]
Therefore, $x+'u_k(\sa)\in\pi_k^{-1}(\ta+\sa)$.
\end{proof}

We conclude this section with a simple inequality which will be used later in the paper.
\begin{lem}\label{l3.7}
Let $x$ and $y$ be two points in $\bbh^\dm$. Then we have
\[
\dist(x, y)\geq |\pi_k(x)-\pi_k(y)|, \qquad\forall~k=1, 2, \cdots, \dm.
\]
\end{lem}
\begin{proof}
The proof is provided in Appendix \ref{sec:app.a}.
\end{proof}

%%%%%%%%%%

\section{Interaction potentials with Newtonian repulsion}
\label{sect:Newt-repulsion}
In this section, we consider interaction potentials that include Newtonian repulsion, i.e., potentials for which the repulsion component is given by the Green's function of the negative Laplacian on $\bbh^\dm$ \cite{FeZh2019}.

Define
\begin{align}\label{eqn:Phi}
\Phi(\theta):=\begin{cases}
\displaystyle-\frac{1}{2\pi}\log\left(\tanh\frac{\theta}{2}\right)& \text{ for } \dm=2,\\[10pt]
\displaystyle \frac{1}{\dm\alpha(\dm)}\int^\infty_{\theta}\frac{1}{\sinh^{\dm-1}\zeta}\d\zeta \qquad& \text{ for } \dm \geq3.
\end{cases}
\end{align}
Note that the case $\dm=2$ can in fact be recovered from the expression for general $\dm$, by calculating explicitly the integral. It was shown in \cite{CohlKalnins2012} that the Green's function for $-\Delta_{\bbh^\dm}$ is given by
\begin{equation}
\label{eqn:G}
G(x, y):=\Phi(\dist(x, y)),
\end{equation}
i.e., $G$ satisfies
\begin{equation}
\label{eqn:DeltaG}
\Delta_x G(x, y)=-\delta_y.
\end{equation}

In this section we consider interaction potentials in the form
\begin{equation}
\label{eqn:K-Na}
K(x,y) = G(x,y) + A(\dist(x,y)),
\end{equation}
where $G$ is given by \eqref{eqn:G} and $A:[0,\infty) \to \bbr $. As discussed in \cite{FeZh2019}, the Green's function component of $K$ (referred to in the sequel as Newtonian potential) models repulsive interactions. To balance the repulsion, we take the $A$ component to model attractive interactions. We refer to Figure \ref{Fig:K} for an illustration of several interaction potentials in the form \eqref{eqn:K-Na}; the attraction components are specified later -- see \eqref{eqn:Ac}, \eqref{eqn:Ah} and \eqref{lcpot}.
\begin{figure}[thb]
 \begin{center}
 \includegraphics[width=0.55\textwidth]{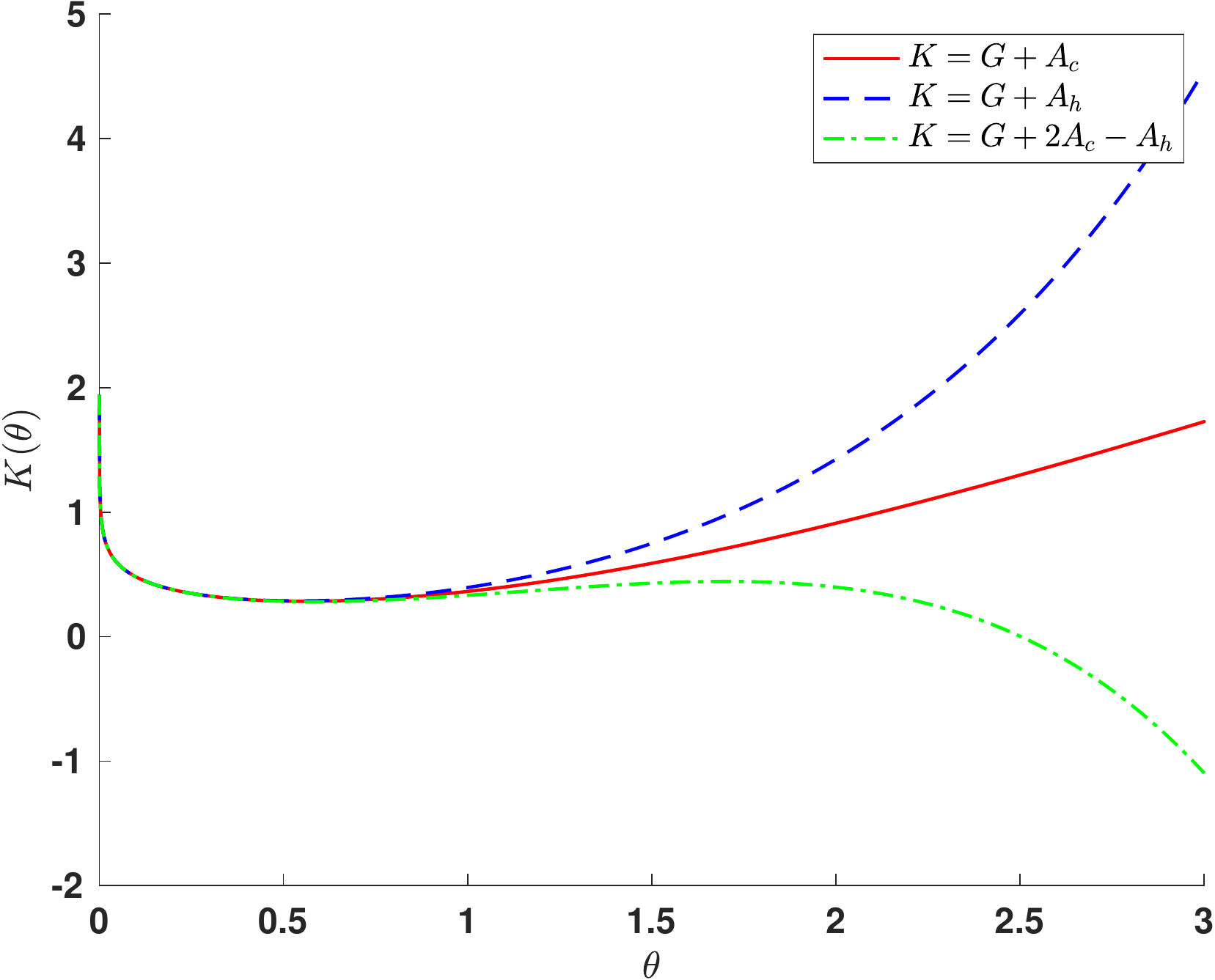} 
 \end{center}
\caption{Interaction potentials in the form \eqref{eqn:K-Na}, with attraction given by \eqref{eqn:Ac}, \eqref{eqn:Ah} and \eqref{lcpot}, respectively.}
\label{Fig:K}
\end{figure}

Our main goals here are to investigate the qualitative properties of critical points of the interaction energy \eqref{eqn:energy} with $K$ in the form \eqref{eqn:K-Na}, as well as to study explicit forms of minimizers for certain specific examples of attraction potentials $A$. All the equilibria investigated in this section are absolutely continuous probability measures $\rho(x) \d x$, with $\rho \in L^1(\d x)$.

\begin{rmk}
\label{rmk:H-1}
Similar to the Euclidean case, the energy \eqref{eqn:energy} corresponding to the Newtonian potential is the square of the $H^{-1}$ norm. Indeed, for $\rho \in H^{-1}(\bbh^\dm)$ with compact support, by definition of the $H^{-1}$ norm and the divergence theorem, we have:
\begin{align*}
\norm{\rho}_{H^{-1}(\bbh^\dm)}^2&=\int_{\bbh^\dm} \|\nabla u(x)\|_{\bbh^\dm}^2 \d x =-\int_{\bbh^\dm} (\Delta_{\bbh^\dm}u)(x) u(x)\d x,
\end{align*}
where $u$ satisfies (in the weak sense): $\displaystyle{-\Delta_{\bbh^\dm} u=\rho}$. Further, using the fundamental solution, write $\displaystyle{u(x) = \int_{\bbh^\dm} G(x,y) \rho(y) \d y}$, to reach
\begin{equation}
\label{eqn:H-1norm}
\norm{\rho}_{H^{-1}(\bbh^\dm)}^2 =\iint_{\bbh^\dm\times\bbh^\dm} G(x,y) \rho(x)\rho(y)\d x \d y.
\end{equation}
\end{rmk}

Regarding the Euler-Lagrange equations, for an interaction potential in the form \eqref{eqn:K-Na}, $\Lambda$ defined in \eqref{eqn:Lambda} reads:
\begin{equation}
\label{eqn:Lambda-GA}
\Lambda(x)=\int_{\bbh^\dm}\left(G(x,y)+A(\dist(x, y)\right)\rho(y) \d y.
\end{equation}
Then, by \eqref{eqn:DeltaG}, one finds:
\begin{equation}
\label{eqn:DeltaLambda}
\Delta\Lambda(x)=
\begin{cases}
\begin{aligned}
& -\rho(x)+\int_{\bbh^\dm}\Delta_x A(\dist(x, y))\rho(y) \d y \qquad  \text{ for } x \in \supp (\rho), \\
& \int_{\bbh^\dm}\Delta_x A(\dist(x, y))\rho(y) \d y\hspace{2.35cm}\text{ for } x \notin \supp (\rho).
\end{aligned}
\end{cases}
%\Delta\Lambda(x)=\int_{\bbh^\dm}(-\delta_x(y)+\Delta_x A(\dist(x, y))\rho(y)dy
\end{equation}
For critical points of the energy (or equivalently, equilibrium solutions of model \eqref{eqn:model}),  $\Lambda(x)$ is identically equal to a constant in the support of $\rho$ (see \eqref{eqn:min-cond-1}). Hence, $\Delta \Lambda = 0$ in $\supp (\rho)$, and the first equation in \eqref{eqn:DeltaLambda} leads to the following integral equation for equilibrium densities $\rho$:
\begin{equation}
\label{eqn:int-eq-gen0}
\rho(x)= \int_{\supp(\rho)}\Delta_x A(\dist(x, y))\rho(y) \d y ,\qquad  \text{ for } x \in \supp (\rho).
\end{equation}

%%%%%

\subsection{Equilibria supported on geodesic balls}
\label{subsect:radial-min}

We consider equilibria of the aggregation model \eqref{eqn:model} that are supported on geodesic balls centred at the vertex $v$, and study their qualitative properties. To prove the main result  (Theorem \ref{thm:mov-planes}) on the radial symmetry and the monotonicity of such states, we employ the method of moving planes by adapting it to the hyperbolic space. 

Before we present the main result, we introduce the reflection on the hyperbolic space $\bbh^\dm$. In the Euclidean space $\bbr^\dm$, the reflection by a hyperplane $x_1=\ta$ is given by the map:
\begin{align}\label{R-1}
(x_1, x_2, \cdots, x_\dm)\mapsto (2\ta-x_1, x_2, \cdots, x_\dm).
\end{align}
%However, there isn't a trivial way to define a reflection on the hyperbolic space by an arbitrary hyperplane. 
The reflection map \eqref{R-1} can be divided into three maps as follows:
\begin{align}\label{R-2}
(x_1, x_2, \cdots, x_\dm)\mapsto(x_1-\ta, x_2, \cdots, x_\dm)
\mapsto(\ta-x_1, x_2, \cdots, x_\dm)
\mapsto(2\ta-x_1, x_2, \cdots, x_\dm).
\end{align}
The first map is the translation in the $x_1$ coordinate by $-\ta$, the second map is the reflection by the hyperplane $x_1=0$, and the third map is the translation in the $x_1$ coordinate by $\ta$. 

We apply now a similar argument to define the reflection on $\bbh^\dm$. We present the reflection corresponding to coordinate $x_1$, the reflections corresponding to coordinates $x_2,\dots,x_\dm$ are similar. First, we define the reference hypersurface for the reflection. For the Euclidean space, this is a regular flat hyperplane, however, for the hyperbolic space we choose the hypersurface to be $P_1(\ta)$ defined in \eqref{eqn:translation-k} (see also Figure \ref{Fig0}). For the translation on $\bbh^\dm$ we use the $+'$ operation (see Definition \ref{defn:+prime}), and take the translation by $\pm\ta$ in the $x_1$ coordinate, to be $x\mapsto x\pm' u_1(\ta)$ (see \eqref{eqn:translation-k}). Finally, analogous to the reflection by the hyperplane $x_1=0$ in the Euclidean space, on $\bbh^\dm$ we have reflection by the hypersurface $P_1(0)=\{x\in\bbh^\dm:x_1=0\}$, which we define it to be:
\[
R_1:(x_0, x_1, x_2, \cdots, x_\dm)\in\bbh^\dm\rightlong (x_0, -x_1, x_2, \cdots, x_\dm)\in\bbh^\dm.
\]

Following the map sequence in \eqref{R-2} (translation by $-\ta$, reflection, translation by $\ta$), we express the reflection across an arbitrary hypersurface $P_1(\ta)$ on $\bbh^\dm$ as:
\begin{equation}
\label{R-1-hyp-seq}
x \rightlong x-'u_1(\ta) \rightlong R_1(x-'u_1(\ta))\rightlong R_1(x-'u_1(\ta))+'u_1(\ta).
\end{equation}
This sequence of three maps is illustrated schematically in Figure \ref{Fig3}. Written compactly, the reflection map above reads:
\begin{equation}
\label{R-1-hyp}
x \rightlong x^\ta := R_1(x-'u_1(\ta))+'u_1(\ta).
\end{equation}

%Denote by $x^\ta$ the reflection of $x$ across the hypersurface $P_1(\ta)$, i.e.,
%\begin{equation}
%\label{eqn:x-mu}
%x^\ta :=  R_1(x-'u_1(\ta))+'u_1(\ta).
%\end{equation}
%$x\mapsto R_1(x-'u_1(\ta))+'u_1(\ta)$
For a better understanding of the reflection map, let us express the reflected point $x^\ta$ in coordinates. From the definition of $-'$ and $u_1(\ta)$ (see \eqref{eqn:-prime} and \eqref{eqn:uk}), we have 
\[
\begin{cases}
(x-'u_1(\ta))_0=x_0\cosh\ta-x_1\sinh\ta,\\
(x-'u_1(\ta))_1=x_1\cosh\ta-x_0\sinh\ta,\\
(x-'u_1(\ta))_j=x_j,\qquad\forall~2\leq j\leq \dm.
\end{cases}
\]
Then, by the definition of the reflection $R_1$, we have
\[
\begin{cases}
(R_1(x-'u_1(\ta)))_0=x_0\cosh\ta-x_1\sinh\ta,\\
(R_1(x-'u_1(\ta)))_1=-x_1\cosh\ta+x_0\sinh\ta,\\
(R_1(x-'u_1(\ta)))_j=x_j,\quad\forall~2\leq j\leq \dm.
\end{cases}
\]
Finally, we use the definition of $+'$ and $u_1(\ta)$ to get
\[
\begin{cases}
(R_1(x-'u_1(\ta))+'u_1(\ta))_0&=(x_0\cosh\ta- x_1\sinh\ta)\cosh\ta+(-x_1\cosh\ta+x_0\sinh\ta)\sinh\ta\\
&=x_0\cosh 2\ta-x_1\sinh 2\ta,\\[2pt]
(R_1(x-'u_1(\ta))+'u_1(\ta))_1&=(-x_1\cosh\ta+x_0\sinh\ta)\cosh\ta-(x_0\cosh\ta-x_1\sinh\ta)\sinh\ta\\
&=-x_1\cosh2\ta+x_0\sinh2\ta,\\[2pt]
(R_1(x-'u_1(\ta))+'u_1(\ta))_j&=x_j,\quad\forall~2\leq j\leq \dm.
\end{cases}
\]
Therefore, we can express the reflection \eqref{R-1-hyp} in coordinates, by:
\[
x=(x_0, x_1, \cdots, x_\dm)\rightlong x^\ta=(x_0\cosh 2\ta-x_1\sinh 2\ta, -x_1\cosh 2\ta+x_0\sinh2\ta, x_2, \cdots, x_\dm).
\]
%See Figure \ref{Fig3} for a detail explanation on these three processes.

\begin{figure}[h]
\begin{center}
\begin{tikzpicture}
  \draw[->] (-2.2, 0) ->(14, 0) node[right] {$x_1$};
  \draw[->] (0, -1) ->(0, 4) node[right] {$x_2$};
  \draw[scale=1, domain=-1:4, smooth, variable=\x,very thick] plot({((e^(1.5)-e^(-1.5))/2)*(1+\x*\x)^(1/2)},{\x});
    \draw[scale=1, domain=-2.2:14, smooth, variable=\x] plot({\x},{3.62686});
    \draw (5, 1) node[left]{$P_1(\ta)$};
      \filldraw[black](13.645, 3.62686) circle (2pt) node[anchor=north east]{$A$};
            \filldraw[black](1.96046, 3.62686) circle (2pt) node[anchor=north ]{$B$};
\filldraw[black](-1.96046, 3.62686) circle (2pt) node[anchor=north east]{$C$};
\filldraw[black] (4.42134, 3.62686) circle (2pt) node[anchor=south ]{$D$};
\draw[-latex]  (13.645, 3.62686) to [bend right] (2, 3.7);
\node[draw] at (7.8,4.8) {$x\mapsto x-'u_1(\ta)$};

\draw[-latex]  (1.96046, 3.62686) to [bend right] (-1.9, 3.7);
\node[draw] at (0,4.8) {$(x_0, x_1, x_2)\mapsto (x_0, -x_1, x_2)$};

\draw[-latex] (-1.96046, 3.62686)  to [bend right] (4.3, 3.5);
\node[draw] at (1.5,2.2) {$x\mapsto x+'u_1(\ta)$};
\end{tikzpicture}
\end{center}
\caption{Schematic illustration (in the $x_1x_2$-plane) of the reflection on $\bbh^2$ across the hypersurface $P_1(\ta)$ -- see \eqref{R-1-hyp}. Point $A$ reflects into point $D$, as the result of three maps $A \mapsto B$ (translation by $-u_1(\ta)$), $B \mapsto C$ (reflection across $P_1(0)=x_0x_2-$plane), and $C \mapsto D$ (translation by $u_1(\ta)$.
%For any point $A$ on the hyperbolic space $\bbh^2\subset\bbr^3$, we can express it as $A=(\cosh \alpha\cosh\beta, \sinh\alpha\cosh\beta, \sinh\beta)$. From a translation $x\mapsto x-'u_1(\ta)$, we can obtain $B=(\cosh(\alpha-\ta)\cosh\beta, \sinh(\alpha-\ta)\cosh\beta, \sinh\beta)$. We reflect point $B$ by the $x_0x_2$-plane to obtain a point(i.e. the corresponding map is $(x_0, x_1, x_2)\mapsto(x_0, -x_1, x_2)$), we obtain $C=(\cosh(\ta-\alpha)\cosh\beta, \sinh(\ta-\alpha)\cosh\beta, \sinh\beta)$. Finally, we use a translation $x\to\mapsto x+'u_1(\ta)$ to obtain  $D=(\cosh(2\ta-\alpha)\cosh\beta, \sinh(2\ta-\alpha)\cosh\beta, \sinh\beta)$.
%Let a point $A=(\cosh 2 \cosh 2, \sinh 2\cosh 2, \sinh2)$ be given. From a translation, we have $A-'u_1(1.5):=B=(\cosh 0.5 \cosh 2, \sinh 0.5\cosh 2, \sinh2)$. We reflect point $B$ by the $x_0x_2$-plane to obtain a point $C=(\cosh 0.5 \cosh 2, -\sinh 0.5\cosh 2, \sinh2)$. Finally, we use a translation to obtain $C+'u_1(1.5):=D=(\cosh 1 \cosh 2, \sinh 1\cosh 2, \sinh2)$.
%Reflection by $P_1(1)$. The left blue circle is $x^2+y^2=1$, and its image is the ellipse $\left((\mathrm{sech}2) x_1-\sqrt{2}\tanh2\right)^2+x_2^2=1$. The left red line is a part of the hyperbola $\left(\frac{x_1}{\sinh(1/2)}\right)^2-x_2^2=1$, and its image is the right red line which is a part of the hyperbola $\left(\frac{x_1}{\sinh(3/2)}\right)^2-x_2^2=1$
}\label{Fig3}
\end{figure}
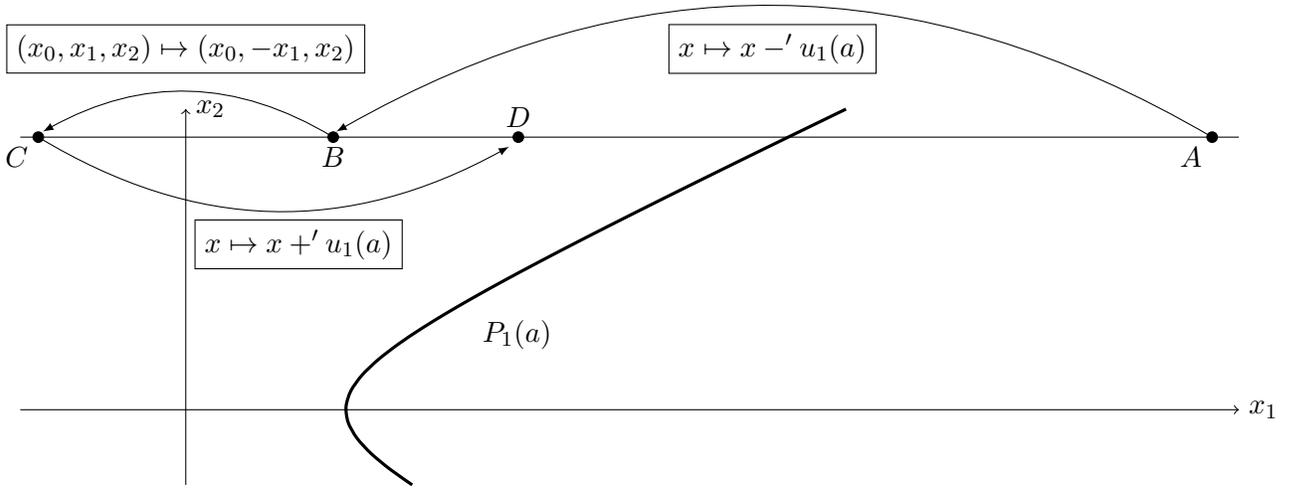

\begin{rmk} The following properties can be checked easily:
\begin{enumerate}
\item $x=(x^\ta)^\ta$, for all $x \in \bbh^\dm$.
\item The set of fixed points for the reflection \eqref{R-1-hyp} is $P_1(\ta)$. 
\item The reflection map \eqref{R-1-hyp} is an isometry on $\bbh^\dm$.
\end{enumerate}
\end{rmk}

We present now the main result of this subsection.
%The result generalizes Theorem 3.2 from \cite{FeHu13} for the aggregation equation in the Euclidean space
%. The flow of the proof is similar, but the various parts are replaced. So, we provide a complete proof of the theorem.

\begin{thm}
\label{thm:mov-planes}
Consider the aggregation model \eqref{eqn:model} with an interaction potential in the form \eqref{eqn:K-Na}. Let $\rho$ be a bounded equilibrium solution of model \eqref{eqn:model} that is supported in the geodesic ball $B_R(v)$, for some $R>0$. Let $F(\dist(x, y)):=\Delta_x A(\dist(x, y))$ and assume $F(\dist(x, \cdot))$ is locally integrable for all $x$. Then, provided $F$ is a monotone function, $\rho$ is radially symmetric and monotone about the vertex $v$. Specifically, (i) if $F$ is decreasing, then $\rho$ is radially symmetric and monotonically decreasing about the vertex $v$, and (ii) if $F$ is increasing, then $\rho$ is radially symmetric and monotonically increasing about $v$.
\end{thm}
\begin{proof}
Since $\rho$ is an equilibrium solution, we use \eqref{eqn:int-eq-gen0} to write:
\begin{equation}
\label{eqn:rho-ss}
\rho(x)=\begin{cases}
\displaystyle\int_{B_R(v)} F(\dist(x, y))\rho(y)\d y \qquad& \text{ for } x\in B_R(v),\\[10pt]
0&\text{otherwise}.
\end{cases}
\end{equation}
We define the reflection $\rho_\ta$ of $\rho$ across the plane $P_1(\ta)$ by:
\[
\rho_\ta(x)=\rho(x^\ta).
\]
From \eqref{eqn:rho-ss} we have
\[
\rho_\ta(x)=\begin{cases}
\displaystyle\int_{B_R(v)}F(\dist(x^\ta, y))\rho(y) \d y \qquad& x\in B_{R}(v^\ta),\\[10pt]
0&\text{otherwise}.
\end{cases}
\]
Also, we set
\[
\Sigma_\ta=\{x\in B_R(v): \pi_1(x)\geq\ta\}.
\]

We apply the method of moving planes and compare $\rho(x)$ and $\rho_\ta(x)$ on $\Sigma_\ta$. If $x\not\in B_{R}(v^\ta)$, then we know that $\rho_\ta(x)=0$. Since $\rho(x)\geq0$, we can conclude that $\rho(x)-\rho_\ta(x)\geq0$ for all $x\in \Sigma_\ta\setminus  B_{R}(v^\ta)$. 

Consider the case $x\in \Sigma_\ta\cap B_{R}(v^\ta)$. 
%Define $\Sigma_\ta^C:=B_R(v)\backslash \Sigma_\ta$. 
By \eqref{eqn:rho-ss} we write $\rho(x)$ as
\begin{equation}
\label{eqn:rho-decomp}
\rho(x)=\int_{B_R(v) \cap \Sigma_\ta}F(\dist(x, y))\rho(y)\d y+\int_{B_R(v) \setminus \Sigma_\ta}F(\dist(x, y))\rho(y)\d y.
\end{equation}
We use that the reflection map \eqref{R-1-hyp} is an isometry to rewrite the second integral above in reflected form (reflect the domain of integration, as well as $x$ and $y$). The image of $B_R(v) \setminus \Sigma_\ta$ under the reflection map \eqref{R-1-hyp} is $B_{R}(v^\ta) \cap \Sigma_\ta$. Also use $\rho(y) = \rho_\ta(y^\ta)$ to get from \eqref{eqn:rho-decomp}:
\begin{align*}
\rho(x)&=\int_{B_R(v) \cap \Sigma_\ta}F(\dist(x, y))\rho(y)\d y+\int_{B_{R}(v^\ta) \cap \Sigma_\ta}F(\dist(x^\ta, y^\ta))\rho_\ta(y^\ta)\d y^\ta\\
&=\int_{B_R(v) \cap \Sigma_\ta}F(\dist(x, y))\rho(y)\d y+\int_{B_{R}(v^\ta) \cap \Sigma_\ta}F(\dist(x^\ta, y))\rho_\ta(y)\d y,
\end{align*}
where for the second equal sign we simply renamed the variable $y^\ta$ to $y$. 

Since $\rho$ and $\rho_\ta$ are zero outside $B_R(v)$ and $B_R(v^\ta)$, respectively, we can also write the expression above as:
\begin{equation}
\label{eqn:rho-decomp-f}
\rho(x) = \int_{\Sigma_\ta}F(\dist(x, y))\rho(y)\d y+\int_{\Sigma_\ta}F(\dist(x^\ta, y))\rho_\ta(y)\d y.
\end{equation}
By replacing $x$ with $x^\ta$ in \eqref{eqn:rho-decomp-f}, we can obtain the following expression for $\rho_\ta$:
\begin{equation}
\label{eqn:rhomu-decomp}
\rho_\ta(x)=\int_{\Sigma_\ta}F(\dist(x^\ta, y))\rho(y)dy+\int_{\Sigma_\ta }F(\dist(x, y))\rho_\ta(y)\d y.
\end{equation}
From \eqref{eqn:rho-decomp-f} and \eqref{eqn:rhomu-decomp} we calculate the difference $\rho(x)-\rho_\ta(x)$ for $x\in \Sigma_\ta\cap B_{R}(v^\ta)$ as:
\begin{equation}
\label{eqn:diff-rho}
\rho(x)-\rho_\ta(x)=\int_{\Sigma_\ta} (F(\dist(x, y))-F(\dist(x^\ta, y)))(\rho(y)-\rho_\ta(y))\d y.
\end{equation}

We can show that
\begin{equation}
\label{eqn:ineq-dist}
\dist(x, y)\leq \dist(x^\ta, y), \qquad \text{ for all } x, y\in \Sigma_\ta.
\end{equation}
Indeed, without loss of generality, it is enough to proof the case $\ta=0$. Let $x=(x_0, x_1, \cdots, x_\dm)$ and $y=(y_0, y_1, \cdots, y_\dm)$. Since $x, y\in \Sigma_0$, we have $x_1, y_1\geq0$. This yields that
\[
\cosh\left(\dist(x, y)\right)=x_0y_0-x_1y_1-\cdots-x_\dm y_\dm\leq  x_0y_0+x_1y_1-\cdots-x_\dm y_\dm=\cosh\left(\dist(x^0, y)\right),
\]
which implies the desired result. 

Consider now the following two cases.
\medskip

\noindent {\em Case (i): $F$ is decreasing.} Since $F$ is decreasing and \eqref{eqn:ineq-dist} holds, we have 
\begin{equation}
\label{eqn:diff-F}
F(\dist(x, y))-F(\dist(x^\ta, y))\geq 0, \qquad \text{ for all }x\in\Sigma_\ta\cap B_{R}(v^\ta)\text{ and }y\in\Sigma_\ta.
\end{equation}
We will show that 
\[
\rho_{\ta} \leq \rho, \qquad \text{ on } \Sigma_{\ta}, 
\]
for all $ \ta \in (-R,0)$, which implies that $\rho$ is monotonically decreasing about the vertex, in the $x_1$ direction.

Assume first that there exists $\ta_0\in(-R, 0)$ such that 
\begin{equation}
\label{eqn:ineq-mu0}
\rho_{\ta_0}(x) \leq \rho(x), \qquad \text{ for all } x\in\Sigma_{\ta_0},
\end{equation}
and show that $\ta_0$ can be extended to the origin. Indeed, suppose by contradiction that \eqref{eqn:ineq-mu0} holds, but $\ta_0$ cannot be extended, and take $\ta \in (\ta_0,\ta_0+\epsilon)$, with $\epsilon>0$. Define
\[
\Sigma_\ta^- = \{ x \in \Sigma_\ta: \rho(x) < \rho_\ta(x)\}.
\]
Then, from \eqref{eqn:diff-rho} and \eqref{eqn:diff-F}, for any $x \in \Sigma_\ta^-$, we have:
\[
0 < \rho_\ta(x) - \rho(x) \leq \int_{\Sigma_\ta^-}  (F(\dist(x, y))-F(\dist(x^\ta, y)))(\rho_\ta(y)-\rho(y)) \d y.
\]
Therefore, 
\begin{equation}
\label{eqn:norm-diff}
\| \rho_\ta- \rho\|_{L^\infty(\Sigma_\ta^-)} \leq \| \rho_\ta- \rho\|_{L^\infty(\Sigma_\ta^-)} \sup_{x \in \Sigma_\ta^-} \int_{\Sigma_\ta^-}  (F(\dist(x, y))-F(\dist(x^\ta, y))) \d y.
\end{equation}

The function under the integral on the right-hand-side is locally integrable. Also, by \eqref{eqn:diff-rho}, 
it holds that $\rho(x) > \rho_{\ta_0}(x)$ in the interior of $\Sigma_{\ta_0}$, implying that the
 closure of $\Sigma_{\ta_0}^-$ has measure $0$. As $\lim_{\ta \to \ta_0} 
\Sigma_{\ta}^- \subset \overline{\Sigma_{\ta_0}^-}$, we infer that the measure of $\Sigma_{\ta}^-$ approaches 
$0$ as $\ta \to  \ta_0$. Therefore, we can choose $\epsilon$ small enough such that
\[
\sup_{x \in \Sigma_\ta^-} \int_{\Sigma_\ta^-}  (F(\dist(x, y))-F(\dist(x^\ta, y))) \d y  \leq \frac{1}{2}.
\]
From \eqref{eqn:norm-diff} we then infer that $ \| \rho_\ta - \rho \|_{{L^\infty}(\Sigma_{\ta}^-)} $ =\, 0, 
which implies that $\Sigma_{\ta}^-$ is empty. Consequently, $\ta_0$ is not maximal and we reached a contradiction.

To show that there exists $\ta_0 \in (-R,0)$ with the property~\eqref{eqn:ineq-mu0}, note 
that~\eqref{eqn:ineq-mu0} holds for $\ta_0 = -R$. Then, by an argument entirely similar to the one made in the previous step, one can infer that the plane  $\ta_0 = -R$ can be moved further to the right, while~\eqref{eqn:ineq-mu0} still holds. 

A similar argument can be made using hypersufaces with $\ta>0$. Finally, since the direction $x_1$ can be chosen arbitrarily, we conclude that $\rho$ is radially symmetric and decreasing about the vertex $v$. This concludes the proof for case (i).
\medskip

\noindent {\em Case (ii): $F$ is increasing.}  From a similar argument as for case (i), one can show that $\rho$ is radially symmetric and increasing about the vertex $v$.
\end{proof}

We now turn the attention on radially symmetric equilibria that are local minimizers of the energy (see Lemma \ref{lem:loc-min}).
\begin{prop}\label{P3.5} 
\label{prop:local-min}
Let $K$ be an interaction potential in the form \eqref{eqn:K-Na}, and let $\rho$ be a bounded, radially symmetric equilibrium solution of model \eqref{eqn:model} that is supported in the geodesic ball $B_R(v)$, for some $R>0$.
Assume that $A:[0,\infty) \to \bbr$ satisfies $\Delta_x A(\dist(x, y))\geq0$ for all $x, y\in\bbh^\dm$. Then $\rho$ satisfies the necessary condition \eqref{eqn:min-cond} for being a local energy minimizer.
\end{prop}
\begin{proof}
Since $\rho$ is an equilibrium solution, it satisfies the first equation of \eqref{eqn:min-cond}. It remains to show \eqref{eqn:min-cond-2}. By using \eqref{eqn:DeltaLambda}, we write:
\begin{align}\label{lambdaeq}
\Delta \Lambda(x)=
\begin{cases}
0\qquad&\text{if }x\in B_R(v),\\[3pt]
\displaystyle\int_{\bbh^\dm}\Delta_xA(\dist(x, y))\rho(y)\d y,\qquad&\text{if }x\not \in  B_R(v).
\end{cases}
\end{align}

As $K$ and $\rho$ are radially symmetric, $\Lambda$ has radial symmetry as well. Hence, we can define $\tilde{\Lambda}:[0,\infty) \to \R$ by:
\begin{equation}
\label{eqn:Lambda-t}
\tilde{\Lambda}(r):=\Lambda(x),
\end{equation}
for some $x\in\partial B_r(v)$. In particular, by \eqref{eqn:min-cond-1}, $\tilde{\Lambda}(r)=\lambda$ for $0\leq r \leq R$.

For $r>R$, use the divergence theorem in $B_r(v)$ to obtain that
\begin{equation}
\label{eqn:div-th}
\int_{B_r(v)}\Delta \Lambda(x) dx=\int_{\partial B_r(v)}\nabla \Lambda(x)\cdot \d \vec{n}=|\partial B_r(v)| \tilde{\Lambda}'(r).
\end{equation}
From $\Delta_xA(\dist(x, y))\geq0$ and \eqref{lambdaeq}, it follows that $\Delta\Lambda(x)\geq 0$ for all $x$. Consequently, using the equation above, we infer that $\tilde{\Lambda}$ is a non-decreasing function, and conclude that
\[
\begin{cases}
\tilde{\Lambda}(r)=\lambda \qquad\text{if }0\leq r\leq R,\\[3pt]
\tilde{\Lambda}(r)\geq \lambda\qquad\text{if }r>R.
\end{cases}
\]
This shows that the necessary condition \eqref{eqn:min-cond-2} for a local minimizer is also satisfied.
%From the definition of $\tilde{\Lambda}$, we can obtain that
%\[
%\begin{cases}
%\Lambda(x)\geq \lambda\quad \text{a.e.\quad if }\rho(x)=0,\\
%\Lambda(x)=\lambda \quad \text{a.e.\quad if }\rho(x)>0.
%\end{cases}
%\]
%Finally, we can obtain the necessary condition of local minimizer introduced in Lemma \ref{lem:loc-min}. 
\end{proof}

The following proposition shows that for potentials with $\Delta A>0$, radially symmetric local minimizers must have a geodesic ball as their support.
\begin{prop}
Let $K$ be an interaction potential in the form \eqref{eqn:K-Na}, where $A:[0,\infty) \to \bbr$ satisfies $\Delta_x A(\dist(x, y))>0$ for all $x, y\in\bbh^\dm$. If $\rho$ is a bounded, radially symmetric local minimizer of the interaction energy, then there exists $R>0$ such that $\supp(\rho)=B_R(v)$.
\end{prop}
\begin{proof}
As $\rho$ is bounded and radially symmetric, its support can be expressed as 
\[
\supp(\rho)=\bigcup_{r\in C}\partial B_r(v),
\]
where $C\subset[0, \infty)$ is a compact set. 

First note that $\Lambda$ is radially symmetric (see \eqref{eqn:Lambda-t}) and $\tilde{\Lambda}(r)=\lambda $ for $r \in C$. Also, from 
\[
\Delta \Lambda(x)=
\begin{cases}
0\qquad&\text{if }x\in \supp(\rho),\\[3pt]
\displaystyle\int_{\bbh^\dm}\Delta_xA(\dist(x, y))\rho(y)\d y >0 \qquad&\text{if }x\not \in  \supp(\rho),
\end{cases}
\]
and \eqref{eqn:div-th}, we find that:
\begin{align}\label{incr}
\tilde{\Lambda}'(r)>0\qquad\text{if } r\not\in C.
\end{align}

Since $C$ is a compact set in $[0, \infty)$, we can express $C$ as a union of disjoint closed intervals as follows:
\[
C=[a_1, b_1]\cup[a_2, b_2]\cup\cdots \cup[a_k, b_k],
\]
where $0\leq a_1<b_1<a_2<b_2<\cdots <a_k<b_k$. If $k\geq2$, then
\[
\tilde{\Lambda}(b_1)=\tilde{\Lambda}(a_2) \; (=\lambda).
\]
However, this contradicts \eqref{incr}, so we conclude that $k=1$. Now assume that $0<a_1$. From \eqref{incr}, we have
\[
\tilde{\Lambda}(0)<\tilde{\Lambda}(a_1),
\]
but this cannot occur for a local minimizer (see \eqref{eqn:min-cond-2}). 

We conclude that $a_1=0$, and $C=[0, b_1]$. By setting $R=b_1$, we then have
\[
\supp(\rho)=\bigcup_{r\in [0, R]}\partial B_r(v)=B_R(v).
\]
\end{proof}

%%%%%

\subsection{Explicit forms of radially symmetric equilibria}\label{subsect:explicit}
We first find some explicit expressions of equilibria for certain interaction potentials in the form \eqref{eqn:K-Na}. These equilibria are radially symmetric and with $\supp (\rho) = D_R(v)$, $R>0$.

Denote 
\begin{equation}
\label{eqn:rot-sym}
\trho(\dist(v, x)):=\rho(x).
%\qquad \text{ and } \qquad \tLambda(\dist(x, v)):=\Lambda(x).
\end{equation}
Then, the integral equation \eqref{eqn:int-eq-gen0} reads:
\begin{equation}
\label{eqn:int-eq-gen}
\trho(\theta_x)=\int_0^R\int_{\bbs^\dm} \left(A''(\theta_{xy})+(\dm-1)\coth(\theta_{xy})A'(\theta_{xy})\right) \trho(\theta_y) \sinh^{\dm-1}\theta_y \, \d\sigma_y \d\theta_y, 
\end{equation}
that holds for all $x \in D_R(v)$, where $\theta_x:=\dist(v,x)$ and $\theta_{xy}:=\dist(x,y)$. Note that we used \eqref{a-a-14} to write:
\begin{align}\label{d-2-1}
\Delta_xA(\dist(x, y))=A''(\dist(x, y))+(\dm-1)\coth(\dist(x, y)) A'(\dist(x, y)).
\end{align}

%Before we study a general attractive potential $A$, we consider some specific attractive potentials. 
The two main specific examples of attractive potentials investigated here are:
\medskip

\noindent(i) Constant Laplacian attraction:
\begin{align}\label{eqn:Ac}
\Ac(\theta)=\int_0^\theta\frac{\int_0^{\theta_*}\sinh^{\dm-1}\xi \d\xi}{\sinh^{\dm-1}\theta_*}\,  \d \theta_*.
\end{align}
This is a very important example for applications; $\Ac$ has the property that 
\begin{equation}
\label{eqn:DeltaAc}
\Delta_{x} \Ac(\dist(x,y))=1,
\end{equation}
and leads to equilibria that are constant within their support. \\

\noindent(ii) Hyperbolic-cosine attraction:
\begin{align}\label{eqn:Ah}
\Ah(\theta)
=\frac{\cosh\theta-1}{\dm}.
\end{align}
This potential satisfies (use \eqref{d-2-1}):
\[
\Delta_x\Ah(\dist(x, y))=\cosh(\dist(x, y)).
\] 
We will show that the energy functional in this case is strictly convex, and we will find explicitly the global minimizer.
%\medskip

%\noindent(iii) Quadratic attraction(\textcolor{red}{Actually, we have not studied anything on this potential yet. Maybe we will use this potential in the numeric section.}):
%\begin{align}\label{d-3}
%\Aq(\theta)=\frac{\theta^2}{2d}.
%\end{align}
% $\Delta_xA(\dist(x, y))=\frac{1}{d}\left(1+(\dm-1)\dist(x, y)\coth(\dist(x, y))\right)$.

\paragraph{Constant Laplacian attraction.}Consider the interaction potential \eqref{eqn:K-Na} with attraction given by \eqref{eqn:Ac}.  Since $\Ac$ satisfies \eqref{eqn:DeltaAc}, it is immediate to check that the following density distribution is a solution of \eqref{eqn:int-eq-gen}:
\begin{align}\label{d-15}
\rhoc(x)=
\begin{cases}
1\qquad \text{ if }\dist(v, x)\leq \Rc,\\[2pt]
0\qquad \text{ if }\dist(v, x)>\Rc,
\end{cases}
\end{align}
where $\Rc$ satisfies
\begin{align}\label{d-16}
\mathrm{Vol}(D_{\Rc}(v))=\dm \alpha(\dm)\int_0^{\Rc}\sinh^{\dm-1}\theta \d \theta=1.
\end{align}
Hence, $\rhoc$ is an equilibrium solution, and by Proposition \ref{prop:local-min}, it satisfies the necessary conditions for a local minimizer. This configuration was studied in \cite{FeZh2019} as an attractor for the aggregation equation \eqref{eqn:model}.

\paragraph{Hyperbolic-cosine attraction.}  We consider now the interaction potential \eqref{eqn:K-Na} with attraction given by \eqref{eqn:Ah}. We will show that the following density distribution is a global minimizer of the interaction energy:
\begin{align}\label{d-8}
\rhoh(x)=\begin{cases}
\displaystyle\frac{\cosh (\dist(v, x))}{\alpha(\dm)\sinh^\dm\Rh} &\qquad\text{ if }\dist(v, x)\leq \Rh,\\[10pt]
0 &\qquad\text{ if }\dist(v, x)>\Rh,
\end{cases}
\end{align}
where $\Rh>0$ satisfies
\begin{align}\label{d-8-1}
\dm \alpha(\dm)\int_0^{\Rh}\sinh^{\dm-1}\theta \cosh^2\theta \, \d \theta=1.
\end{align}

We will show that \eqref{d-8} is a solution of \eqref{eqn:int-eq-gen}. By rotational symmetry, use notation:
\[
\trhoh(\dist(v, x)):=\rhoh(x).
%\qquad\text{and}\qquad\tLambdah(\dist(x, v)):=\Lambdah(x).
\]
The hyperbolic law of cosines on $\bbh^\dm$ states:
\begin{align}\label{d-10}
\cosh (\dist(x, y))=\cosh(\dist(x, v))\cosh(\dist(y, v))-\sinh(\dist(x, v))\sinh(\dist(y, v))\cos \angle(xvy).
\end{align}

From \eqref{d-10} we compute, using polar coordinates for $y$ and the rotational symmetry of $\rhoh$:
\begin{align*}
& \int_{\bbh^\dm}\cosh(\dist(x, y))\rhoh(y)\d y \\
%& \quad =\int_{\bbh^\dm}\left(\cosh(\dist(x, v))\cosh(\dist(y, v))-\sinh(\dist(x, v))\sinh(\dist(y, v))\cos\angle(xvy)\right)\rhoh(y) \d y\\
& \quad = \int_0^{\Rh}\int_{\bbs^{\dm-1}}\left(\cosh(\dist(x, v))\cosh\theta-\sinh(\dist(x, v))\sinh\theta \cos \angle(xvy)\right)\trhoh(\theta)\sinh^{\dm-1}\theta \, \d \sigma_y \d \theta,
\end{align*} 
where we also used that $\supp(\rhoh)=D_{\Rh}(v)$. Then, using 
\begin{align*}
\int_{\bbs^{\dm-1}}\cos\angle(xvy)\d \sigma_y&=\frac{1}{2}\int_{\bbs^\dm}(\cos \angle (xvy) +\cos \angle (xv(-y)) \d \sigma_y\\
&=\frac{1}{2}\int_{\bbs^\dm}(\cos \angle (xvy)+\cos(\pi+\angle(xvy)) \d \sigma_y \\
&=0,
\end{align*}
in the equation above, we have
\begin{equation}
\label{eqn:int-cosh}
 \int_{\bbh^\dm}\cosh(\dist(x, y))\rhoh(y)\d y = \dm\alpha(\dm) \cosh (\dist(v, x))\int_0^{\Rh}\cosh\theta \trhoh(\theta) \sinh^{\dm-1}\theta \, \d \theta.
\end{equation}
Consequently, to show that $\rhoh$ is a solution of \eqref{eqn:int-eq-gen}, we need to show:
\[
\trhoh(\dist(v,x)) = \dm\alpha(\dm) \cosh (\dist(v, x))\int_0^{\Rh}\cosh\theta \trhoh(\theta) \sinh^{\dm-1}\theta \, \d \theta.
\]
Indeed, by substituting \eqref{d-8} into the right-hand-side above, one computes:
\begin{align*}
& \dm\alpha(\dm)\cosh(\dist(v, x))\int_0^{\Rh}\frac{\cosh^2\theta \sinh^{\dm-1}\theta }{\alpha(\dm)\sinh^\dm  \Rh}d\theta\\[2pt]
&=\frac{\cosh(\dist(v, x))}{\alpha(\dm)\sinh^\dm  \Rh} \, \underbrace{\dm\alpha(\dm)\int_0^{\Rh}\cosh^2\theta\sinh^{\dm-1}\theta d\theta}_{=1 \text{ by \eqref{d-8-1}}} \\[1pt]
&=\trhoh(\dist(v,x)).
\end{align*}
This shows that $\rhoh$ is an equilibrium solution of model \eqref{eqn:model}.  By Proposition \ref{prop:local-min}, it also satisfies the necessary conditions to be a local energy minimizer. In Section \ref{sec:3.3} we will prove that the energy functional corresponding to this potential is strictly convex and hence, $\rhoh$ is a global energy minimizer.

%\paragraph{General convex attraction.}  We know that
%\[
%\int_{\bbs^\dm} \left(A''(\theta_{xy})+(\dm-1)\coth(\theta_{xy})A'(\theta_{xy})\right)\d \sigma_y
%\]
%only depends on $\theta_x$ and $\theta_y$, since
%\[
%\cosh\theta_{xy}=\cosh\theta_x\cosh\theta_y-\sinh\theta_x\sinh\theta_y\cos\angle(xvy)
%\]
%and we integrate it along $y\in\bbs^\dm$. Define
%\begin{align}\label{d-29}
%H(\theta_x, \theta_y):=\int_{\bbs^\dm} \left(A''(\theta_{xy})+(\dm-1)\coth(\theta_{xy})A'(\theta_{xy})\right)d\sigma_y.
%\end{align}
%Here, we can easily check the symmetry: $H(\theta_x, \theta_y)=H(\theta_y, \theta_x)$.
%For example,
%\[
% \Hc(\theta_x, \theta_y)=d\alpha(\dm), \qquad \Hh(\theta_x, \theta_y)=d\alpha(\dm)\cosh\theta_x\cosh\theta_y,
%\]
%for the constant Laplacian attraction and the hyperbolic-cosine attraction, respectively. 
%
%Using the definition of $H$, \eqref{eqn:int-eq-gen} can be written as:
%\begin{align}\label{d-30}
%\trho(\theta_x)=\int_0^R H(\theta_x, \theta_y)\sinh^{\dm-1}\theta_y \, \trho(\theta_y) \d\theta_y, \qquad \text { for } 0 \leq \theta_x < R.
%\end{align}
%The integral equation \eqref{d-30} can be solved explicitly for other attractive potentials as well. 

%%%%%

\subsection{Hyperbolic cosine attraction: convexity of the energy}\label{sec:3.3}
In this subsection we consider the potential \eqref{eqn:K-Na} with hyperbolic cosine attraction given by \eqref{eqn:Ah}, and show that the equilibrium solution $\rhoh$ from \eqref{d-8} is a global energy minimizer in $\calA:= \calPc_c(\bbh^\dm) \cap H^{-1}(\bbh^\dm)$.  

Consider the energy \eqref{eqn:energy} defined on $\calA$. Then use Remark \ref{rmk:H-1} (see \eqref{eqn:H-1norm}) to write the energy as
\begin{align}
E[\rho]&=\frac{1}{2}\iint_{\bbh^\dm\times \bbh^\dm}(G(x, y)+\Ah(\dist(x, y))) \rho(x) \rho(y) \d x \d y \nonumber \\ 
&=\frac{1}{2}\|\rho\|^2_{H^{-1}(\bbh^\dm)}+ \frac{1}{2 \dm}\calK[\rho] - \frac{1}{2\dm}, \label{eqn:Eh}
\end{align}
where the functional $\calK[\rho]$ is defined by
\begin{equation}
\label{eqn:calK}
\calK[\rho]=\iint_{\bbh^\dm\times \bbh^\dm} \cosh(\dist(x, y))\rho(x)\rho(y) \d x \d y.
\end{equation}

First note that the repulsion component of the energy is given in terms of the square of the $H^{-1}$ norm, which is strictly convex. Since the last term in \eqref{eqn:Eh} is simply a constant, it remains to investigate the convexity of the functional $\calK$.

By \eqref{eqn:zcomp-dist}, we write
\begin{align*}
\calK[\rho]=
%&\iint \cosh(\dist(x, y))\rho(x)\rho(y)dxdy\\
&=\iint_{\bbh^\dm\times \bbh^\dm} (x_0y_0-x_1y_1-\cdots-x_\dm y_\dm)\rho(x)\rho(y)  \d x \d y\\
&=\left\langle\int_{\bbh^\dm}(x_0, x_1, \cdots, x_\dm)\rho(x)\d x,\int_{\bbh^\dm}(y_0, y_1, \cdots, y_\dm)\rho(y) \d y\right\rangle\\
&=\left\|\int_{\bbh^\dm}(x_0, x_1, \cdots, x_\dm)\rho(x)\d x\right\|^2,
\end{align*}
where we also used \eqref{eqn:iproduct-L} and \eqref{eqn:norm-L}.

Denote 
\begin{equation}
\label{eqn:c-rho}
\crho:=\int_{\bbh^\dm}(x_0, x_1, \cdots, x_\dm)\rho(x)\d x,
\end{equation}
and hence, write
\begin{equation}
\label{eqn:calK-c}
\calK[\rho]= \left\| \crho \right\|^2.
\end{equation}

By a direct calculation, we then have
\begin{align} 
&t\calK[\rho_1]+(1-t)\calK[\rho_2]-\calK[t\rho_1+(1-t)\rho_2] = t\|c_{\rho_1}\|^2+(1-t)\|c_{\rho_2}\|^2-\| c_{t\rho_1+(1-t)\rho_2} \|^2   \nonumber \\
&\hspace{2cm} = t\|c_{\rho_1}\|^2+(1-t)\|c_{\rho_2}\|^2-\| t c_{\rho_1} + (1-t) c_{\rho_2}\|^2 \nonumber \\
&\hspace{2cm} = t(1-t)\|c_{\rho_1}\|^2+t(1-t)\|c_{\rho_2}\|^2-2t(1-t) \left \langle c_{\rho_1}, c_{\rho_2}\right \rangle \nonumber \\
&\hspace{2cm} = t(1-t)\|c_{\rho_1}-c_{\rho_2}\|^2, \label{eqn:conv-comb}
\end{align}
for any density functions $\rho_1$ and $\rho_2$. Note that the right-hand-side above is not necessarily non-negative, as the inner product  \eqref{eqn:iproduct-L} is not positive definite. So the positivity of this term needs to be shown, which is what we will focus on next.

Consider the following set:
\[
\calPcz(\bbh^\dm):=\left\{\rho\in\calPc(\bbh^\dm): \int_{\bbh^\dm} x_j  \rho(x)\d x=0, \text{ for all } 1 \leq j \leq \dm \right\}.
\]
\begin{lem}\label{transdensity}
For any $\rho\in\calPc(\bbh^\dm)$, define $\crhoh$ to be the normalized $\crho$, i.e., 
\[
\crhoh= \frac{\crho}{\| \crho\|}.
\]
Then,
\[
f_{\#}\rho\in\calPcz(\bbh^\dm),\qquad \text{ where } f(x)=x-' \crhoh.
\]
\end{lem}
\begin{proof}
Write $\crho=(c_0,c_1,\cdots,c_\dm )$ using coordinates, i.e., 
\begin{equation}
\label{eqn:c-alpha}
c_k=\int_{\bbh^\dm} x_k \, \rho(x) \d x,\qquad\text{ for all } 0\leq k \leq \dm.
\end{equation}
First note that $\langle \crho, \crho \rangle >0$, so $\| \crho\| $ is real-valued. Indeed, it is easy to show that $ \crho \in\mathrm{conv}(\bbh^\dm)$, where 
$\mathrm{conv}(\bbh^\dm)$ denotes the convex hull of $\bbh^\dm$, defined as:
\[
 \mathrm{conv}(\bbh^\dm):=\{(x_0, x_1, \cdots, x_\dm) \in \bbr^{\dm+1} : x_0^2-x_1^2-\cdots-x_\dm^2\geq1\}.
\]
Consequently, 
\[
\langle \crho, \crho \rangle = c_0^2-c_1^2-\cdots-c_\dm ^2\geq1>0.
\] 

By a change of coordinates, one gets:
\begin{equation}
\label{eqn:ch-coord}
\int_{\bbh^\dm} x_jf_\# \rho(x)\d x=\int_{\bbh^\dm} (x-'\crhoh)_j\rho(x)\d x, \qquad \text{ for all } 1\leq j \leq \dm.
\end{equation}
To have $f\#\rho\in\calPcz(\bbh^\dm)$, the expression above needs to be $0$ for all $1\leq j \leq \dm$.

By \eqref{eqn:-prime} and \eqref{a-a-10}, we compute:
\[
\int_{\bbh^\dm} (x-'\crhoh)_j\rho(x)dx=\int_{\bbh^\dm} \left(x_j-x_0(\crhoh)_j+\frac{(\crhoh)_j}{(\crhoh)_0+1}(x_1(\crhoh)_1+\cdots+x_\dm(\crhoh)_\dm)\right)\rho(x)dx.
\]
From here, by the definition of $\crhoh$ and \eqref{eqn:c-alpha}, we get:
\begin{align}
\int_{\bbh^\dm} (x-'\crhoh)_j\rho(x)dx&=c_j-c_0(\crhoh)_j+\frac{(\crhoh)_j}{(\crhoh)_0+1}(c_1(\crhoh)_1+\cdots+c_\dm (\crhoh)_\dm) \nonumber \\
&= \| \crho \| (\crhoh)_j - \| \crho\| (\crhoh)_0(\crhoh)_j + \frac{ \| \crho \| (\crhoh)_j}{(\crhoh)_0+1}((\crhoh)_1^2+\cdots+(\crhoh)_\dm^2) \nonumber \\
&=\| \crho \| \Big((\crhoh)_j-(\crhoh)_0(\crhoh)_j+(\crhoh)_j((\crhoh)_0-1) \Big) \nonumber \\
& =0, \label{eqn:rhs-zero}
\end{align}
where for the third equal sign we used that $\crhoh$ has norm $1$, which implies $(\crhoh)_1^2+\cdots+(\crhoh)_\dm^2 = (\crhoh)_0^2 - 1$.

The conclusion now follows from \eqref{eqn:ch-coord} and \eqref{eqn:rhs-zero}.
\end{proof}

From \eqref{eqn:calK} and the fact that the operation $-'$ is an isometry, we get:
\begin{align*}
\calK[f_\#\rho]&=\iint_{\bbh^\dm \times \bbh^\dm}\cosh(\dist(x, y))f_\#\rho(x)f_\#\rho(y)\d x\d y\\
&=\iint_{\bbh^\dm \times \bbh^\dm}\cosh(\dist(x-'\crhoh, y-'\crhoh))\rho(x)\rho(y) \d x \d y\\
&=\iint_{\bbh^\dm \times \bbh^\dm}\cosh(\dist(x, y))\rho(x)\rho(y)\d x \d y\\
&=\calK[\rho].
\end{align*} 
Since $\calK[\rho]=\calK[f_\#\rho]$, we can restrict $\calK$ to  $\calPcz(\bbh^\dm) \subset \calPc(\bbh^\dm)$ for the purpose of investigating its minimizers. 

\begin{lem}\label{lem:Kconvex} The functional $\calK$ defined in \eqref{eqn:calK} is convex on $\calPcz(\bbh^\dm)$.
\end{lem}
\begin{proof}
Take $\rho_1,\rho_2\in\calPcz(\bbh^\dm)$; also note here that $\calPcz(\bbh^\dm)$ is a convex set. Then, $c_{\rho_1}=\alpha_1(1, 0, \cdots, 0)$ and $c_{\rho_2}=\alpha_2(1, 0, \cdots, 0)$, for some positive constants $\alpha_1,\alpha_2>0$. Therefore,
\[
\|c_{\rho_1}-c_{\rho_2} \|^2=(\alpha_1-\alpha_2)^2\geq0,
\]
and by using \eqref{eqn:conv-comb}, we find that:
\[
t\calK[\rho_1]+(1-t)\calK[\rho_2]\geq \calK[t\rho_1+(1-t)\rho_2], \qquad \text{ for all } 0<t<1.
\]
This shows the claim.
\end{proof}

Next, we state and prove the main result of this subsection.
\begin{thm}
\label{thm:convexity}
Consider the interaction potential \eqref{eqn:K-Na} with hyperbolic cosine attraction given by \eqref{eqn:Ah}. Then, the equilibrium solution $\rhoh$ from \eqref{d-8} is a global minimizer of the interaction energy on $\calA$.  
\end{thm}
\begin{proof}

As noted above, the square of the $H^{-1}$ norm, representing the repulsion component of the energy, is strictly convex on $\calA$. Since the attraction component $\calK$ is convex on $\calPcz(\bbh^\dm)$, the sum of the two functionals (and hence $E$) is strictly convex on $\calPcz(\bbh^\dm) \cap \calA$. We conclude from here that the critical point $\rho_{h}$ given by \eqref{d-8} is the unique global minimizer of $E$ in $\calPcz(\bbh^\dm) \cap \calA$. 

By Lemma \ref{transdensity}, any density can be translated by an isometry into $\calPcz(\bbh^\dm)$. Since the interaction energy is invariant under isometries (see \eqref{eqn:E-inv}), we infer that the global minimizer $\rho_h$ of $E[\cdot]$ on $\calPcz(\bbh^\dm)  \cap \calA$ is also a global minimizer of $E[\cdot]$ on $\calA$ itself.
\end{proof}

\begin{rmk}
\label{rmk:quadratic}
A similar result as in Theorem \ref{thm:convexity} holds in the standard Euclidean space $\bbr^\dm$-- see Theorem 2.4 in \cite{ChFeTo2015}. There, it is shown that for an interaction potential that consists of Newtonian repulsion and quadratic attraction (i.e., $A(x) = \frac{1}{2} |x|^2$), the interaction energy is strictly convex and the global energy minimizer consists in a uniform density supported on a ball. However, for the model on $\bbh^\dm$ we could not show the convexity of the energy for the attractive potential $\Ac$ (that leads to constant density equilibria), but we showed it instead for $\Ah$.
\end{rmk}

%\begin{rmk}
%\label{rmk:convexity-K}
%We also point out that the argument in Lemma \ref{transdensity} works for any $\rho \in \calP(\bbh^\dm)$. Specifically, any $\rho \in \calP(\bbh^\dm)$ can be translated by an isometry $f$ to $f_\# \rho \in \calPz(\bbh^\dm)$, where 
%\[
%\calPz(\bbh^\dm):=\left\{\mu\in\calP(\bbh^\dm): \\ \int_{\bbh^\dm} x_j  \d \mu (x)=0, \text{ for all } 1 \leq j \leq \dm \right\}.
%\]
%Also,  $\calK[\rho]=\calK[f_\#\rho]$. Therefore, by an argument identical to that used in the proof of Theorem \ref{thm:convexity}, the functional $\calK$ is convex on the entire space $\calP_0(\bbh^\dm)$ (not just on $\calPc_0(\bbh^\dm)$). The reason we restricted the convexity argument in Theorem \ref{thm:convexity} to absolutely continuous measures was due to the repulsion component of the energy.
%\end{rmk}

%%%%%%%%%%%

\section{Existence of compactly supported minimizers}
\label{sect:existence}

In this section we study the existence of compactly supported global minimizers of the energy \eqref{eqn:energy}. The arguments we use here follow the layout and the general technique used by Ca{\~n}izo {\em et al.} \cite{CaCaPa2015} for the Euclidean case. For this reason, we use similar notations, and we only highlight the main differences. 

%\begin{thm}[\cite{CaCaPa2015}, Theorem 1.4]\label{Rtheorem}
%Assume that $W:\bbr^\dm\to\bbr\cup\{+\infty\}$ satisfies the following Hypotheses 1-5:
%
%\noindent$(\mathcal{H}1)$: $W$ is bounded from below by a finite constant $W_{min}<0$.
%
%\noindent$(\mathcal{H}2)$: $W$ is locally integrable.(i.e. $\int_B|W|<\infty$ for any open ball $B$ in the Euclidean space.)
%
%
%\noindent$(\mathcal{H}3)$: $W$ is symmetric (that is, $W(x)=W(-x)$ for all $x\in\bbr^\dm$.
%
%\noindent$(\mathcal{H}4)$: The limit $W_\infty:=\lim_{|x|\to\infty}W(x)$ exists and $W$ is unstable. Here, $W$ is stable if $E(\rho)\geq\frac{1}{2}W_\infty$ for all $\rho\in \mathcal{P}(\bbr^\dm)$.
%
%\noindent$(\mathcal{H}5)$: $W$ is lower semi-continuous.
%
%\noindent$(\mathcal{H}6)$: There is $R_6>0$ such that $W$ is strictly increasing on $\bbr^{k-1}\times [R_6,\infty)\times \bbr^{d-k}$ as a function of its $k$-th variable, for all $k\in\{1, \dots, d\}$. 
%
%Then, there exists a global minimizer $\rho\in\mathcal{P}(\bbr^\dm)$ of the energy
%\[
%E(\rho):=\frac{1}{2}\int_{\bbr^\dm\times\bbr^\dm}W(x-y)d\rho(x)d\rho(y).
%\]
%In addition, there exists $K>0$ (depending only on $W$ and the dimension $d$) such that every minimizer of $E$ has compact support with diameter at most $K$.
%\end{thm}

Consider the interaction energy \eqref{eqn:energy} defined on $\calP(\bbh^\dm)$. Retaining the numbering from \cite{CaCaPa2015}, we make the following assumptions on the interaction potential $K$: \vspace{0.2cm}

\noindent$(\mathcal{H}1)$: $K$ is bounded from below by a finite constant $K_{min}<0$. \vspace{0.2cm}

\noindent$(\mathcal{H}2)$: $K$ is locally integrable, i.e., $\displaystyle{\int_{\dist(x, v)<R}|K(\dist(x, v))|dx<\infty}$, for any $R>0$. \vspace{0.2cm}

\noindent$(\mathcal{H}3)$: Interactions are symmetric, i.e., $\displaystyle{K(x,y) = K(\dist(x,y))}$ (by an abuse of notation), for all $x,y \in \bbh^\dm$. \vspace{0.2cm}

\noindent$(\mathcal{H}4)$: The limit $\displaystyle{K_\infty=\lim_{\theta\to\infty}K(\theta)}$ exists and $K$ is unstable. Here, $K$ is unstable if there exists $\rho \in \calP(\bbh^\dm)$ such that $E[\rho] < \frac{1}{2}K_\infty$. \vspace{0.2cm}

\noindent$(\mathcal{H}5)$: $K$ is lower semi-continuous. \vspace{0.2cm}

\noindent$(\mathcal{H}6)$: There exists $R_6>0$ such that $K$ is strictly increasing on $[R_6, \infty)$. \vspace{0.2cm}

Note that hypothesis $(\mathcal{H}3)$ has been used throughout the entire paper, but we list it here separately for consistency with the numbering in \cite{CaCaPa2015}. In the Euclidean case, this hypothesis shows in a weaker form in fact, specifically $K(x) = K(-x)$ for all $x \in \bbr^\dm$.

Assumptions $(\mathcal{H}1)-(\mathcal{H}6)$ were considered in \cite{CaCaPa2015} for the Euclidean space $\bbr^\dm$. Under such assumptions, the main result in \cite{CaCaPa2015} (Theorem 1.4) states that there exists a global minimizer of the energy in $\calP(\bbr^\dm)$. Moreover, there exists a uniform bound for the supports of all minimizers of $E$. A very similar result holds for $\bbh^\dm$. We state the result below and then we sketch its proof. 

\begin{thm}\label{Htheorem}
Assume that $K:[0, \infty)\to\bbr\cup\{\infty\}$ satisfies $(\mathcal{H}1)-(\mathcal{H}6)$. Then, there exists a global minimizer $\rho\in\mathcal{P}(\bbh^\dm)$ of the energy $E[\rho]$ defined in \eqref{eqn:energy}. In addition, there exists $\D>0$ (depending only on $K$ and the dimension $\dm$) such that every minimizer of $E$ has compact support with diameter at most $\D$.
\end{thm}

The layout of the proof of the analogous theorem for $\bbr^\dm$ from \cite{CaCaPa2015} is the following (we write it for $\bbh^\dm$ instead):
\medskip

{\em Part 1.} Fix $\R\geq 0$ and set $\calP_\R(\bbh^\dm) = \{ \rho \in \calP(\bbh^\dm) : \supp(\rho) \subset \overline{B_R(v))} \}$. Show the existence of global minimizers on $\calP_\R(\bbh^\dm)$. This can be done as follows \cite[Lemma 2.1]{CaCaPa2015}. Consider a minimizing sequence $(\rho_k)_{k \geq 1}$ of $E$ restricted to $\calP_\R(\bbh^\dm)$. As $\calP_\R(\bbh^\dm)$ is tight and weakly-$*$ closed in $\calP(\bbh^\dm)$, by Prokhorov's Theorem, there exists a subsequence $(\rho_{k_l})_{l \geq 1}$ that converges weakly-$*$ to $\rho_\R \in \calP_\R(\bbh^\dm)$. By the weakly-$*$ lower semicontinuity of $E$, $\rho_\R$ is a global minimizer of $E$ on $\calP_\R(\bbh^\dm)$.
\medskip

{\em Part 2.} Show that there exists $\D >0$ (depending only on $K$ and $\dm$) such that for all $\R \geq 0$ and any global minimizer $\rho_\R$ of $E$ on $\calP_\R(\bbh^\dm)$, the diameter of the support of $\rho_\R$ is bounded by $\D$. This is the core of the proof and follows from several key lemmas which we will discuss after presenting the layout.
\medskip

{\em Part 3.} Consider $\rho'$ a global minimizer on $\calP_\D(\bbh^\dm)$, where $\D$ is the upper bound on the diameters of minimizers from Part 2. Show that $\rho'$ is in fact a global minimizer on $\calP(\bbh^\dm)$. This can be argued as follows \cite[Lemma 2.10]{CaCaPa2015}. First note that for any $\rho \in \calP(\bbh^\dm)$ with compact support, $E[\rho] \geq E[\rho']$. Indeed, for $\R \geq 0$ fixed, such that $\rho \in \calP_\R(\bbh^\dm)$, take $\rho_\R$ a global minimizer of $E$ on $\calP_\R(\bbh^\dm)$. Then, $E[\rho] \geq E[\rho_\R] \geq E[\rho']$, where the second inequality holds by the translation invariance of $E$ (as $\rho_\R$ can be translated to a measure in $\calP_\D(\bbh^\dm))$. Second, for any $\rho \in \calP(\bbh^\dm)$, approximate $\rho$ by compactly supported measures $\rho_k$ with $E[\rho_k] \to E[\rho]$ as $k \to \infty$. Since $E[\rho_k] \geq E[\rho']$, for all $k$, infer $E[\rho] \geq E[\rho']$ and conclude $\rho'$ global energy minimizer on $\calP(\bbh^\dm)$.
\smallskip

As noted above, Part 2 is the core of the proof. It is based on two fundamental lemmas. The first lemma states that for any point in the support of a minimizer $\rho_\R$ on $\calP_\R(\bbh^\dm)$, there has to be some mass not far from it. Specifically, the following result holds.
\begin{lem}\cite[Lemma 2.6]{CaCaPa2015}
\label{lemma:mass}
Suppose the potential $K$ satisfies hypotheses $(\mathcal{H}1)-(\mathcal{H}5)$. Then there exist constants $r,m>0$ (depending only on $K$) such that for all sufficiently large $R$, and all global minimizers $\rho_\R$ of the energy on $\calP_\R(\bbh^\dm)$ we have
\[
\int_{B_r(x_0)} \d \rho_\R(x) \geq m, \qquad \text{ for all } x_0 \in \supp(\rho_\R).
\]
\end{lem}
Note that the constants $m$ and $r$ are universal (they do not depend on $\R$). The proof of the lemma can be immediately adapted from that of \cite[Lemma 2.6]{CaCaPa2015}. We only point out here that for its proof, Hypotheses $(\mathcal{H}1)$ and $(\mathcal{H}4)$ are used in a fundamental way, to bound the interaction potential $K$ at short and long distances, respectively.

The second fundamental lemma states that minimizers $\rho_\R$ cannot have large gaps. The general idea of its proof is the same as for \cite[Lemma 2.7]{CaCaPa2015}, but the details are more involved. In place of the projection on the $k$-th coordinate (as used in $\bbr^\dm$), we will use the map $\pi_k$ defined in \eqref{eqn:pik}. This lemma is the only place where hypothesis $(\mathcal{H}6)$ is being used.

\begin{lem}(Generalization of \cite[Lemma 2.7]{CaCaPa2015})
\label{lemma:separation}
Assume that the potential $K$ satisfies $(\mathcal{H}1)-(\mathcal{H}6)$. Let $\R\geq0$ and suppose that $\rho_\R$ is a global minimizer of the energy \eqref{eqn:energy} on $\calP_\R(\bbh^\dm)$. Then the support of $\rho_\R$ cannot have ``gaps" larger than $2R_6$ in each coordinate  ($R_6$ is the constant from $(\mathcal{H}6)$): if $k\in\{1, \dots, \dm \}$ and $a_k\in\bbr$ is such that $\displaystyle{\pi_k^{-1}([a_k-R_6, a_k+R_6])\subseteq \bbh^\dm\backslash \supp(\rho_\R)}$, then either $\displaystyle{\pi_k^{-1}((-\infty, a_k-R_6])\subseteq \bbh^\dm\backslash \supp(\rho_\R)}$ or 
$\displaystyle{\pi_k^{-1}([a_k+R_6, \infty))\subseteq \bbh^\dm\backslash \supp(\rho_\R)}$.
\end{lem}
\begin{proof}
Suppose by contradiction that there exists an index $k\in\{1, 2, \cdots, \dm\}$ and $a_k \in \bbr$ with $\displaystyle{\pi_k^{-1}([a_k-R_6, a_k+R_6])\subseteq \bbh^\dm\backslash \supp(\rho_\R)}$, such that the support of $\rho_\R$ has nontrivial intersections with both the ``left" and ``right" parts, defined as
\[
\HL:=\pi_k^{-1}((-\infty, a_k-R_6])\quad\text{and}\quad \HR:=\pi_k^{-1}([a_k+R_6, \infty)),
\]
respectively. Also, fix $0<\epsilon<R_6$ and consider
\[
\bar{\rho}_\R:=\rho_\R\vert_{_{\HL}}+{T_\epsilon}_{\#} \bigl(\rho_\R\vert_{_{\HR}} \bigr),
\]
where $T_\epsilon$ is the isometry defined as $T_\epsilon(x)=x-'u_k(\epsilon)$ (recall \eqref{eqn:uk}). 

It follows immediately that
\begin{align*}
\supp(\rho_\R\vert_{_{\HL}})\subset \HL \quad \textrm{ and } \quad \supp\bigl ({T_\epsilon}_{\#} \bigl(\rho_\R\vert_{_{\HR}} \bigr) \bigr)\subset \HR-'u_k(\epsilon).
\end{align*}
From Corollary \ref{c3.5} we have
\[
\HR-'u_k(\epsilon)=\pi_k^{-1}([a_k+R_6, \infty))-'u_k(\epsilon)=\pi_k^{-1}([a_k+R_6-\epsilon, \infty)).
\]
Hence, we get
\[
\supp(\rho_\R\vert_{_{\HL}})\subset \pi_k^{-1}((-\infty, a_k-R_6]) \quad \textrm{ and } \quad \supp\bigl ({T_\epsilon}_{\#} \bigl(\rho_\R\vert_{_{\HR}} \bigr) \bigr) \subset \pi_k^{-1}([a_k+R_6-\epsilon, \infty)).
\]

From this fact and $0<\epsilon<R_6$, we infer that $\supp(\rho_\R\vert_{_{\HL}})$ and $\supp\bigl ({T_\epsilon}_{\#} \bigl(\rho_\R\vert_{_{\HR}} \bigr) \Bigr)$ are disjoint. This yields the following calculation:
\begin{align}
E(\bar{\rho}_\R)&=E(\rho_\R\vert_{_{\HL}})+E({T_\epsilon}_{\#} \bigl(\rho_\R\vert_{_{\HR}} \bigr) )
+\iint_{\bbh^\dm\times\bbh^\dm}K(\dist(x,y)) \,\d\rho_\R\vert_{_{\HL}} (x) \, \d {T_\epsilon}_{\#} \bigl(\rho_\R\vert_{_{\HR}} \bigr) (y) \nonumber \\
&=E(\rho_\R\vert_{_{\HL}})+E(\rho_\R\vert_{_{\HR}})
+\iint_{\bbh^\dm\times\bbh^\dm} K(\dist(x,y-'u_k(\epsilon))) \, \d\rho_\R\vert_{_{\HL}} (x) \, \d\rho_\R\vert_{_{\HR}}(y), \label{c-10}
\end{align}
where for the second equal sign we used the invariance under isometries of the energy (see \eqref{eqn:E-inv}) and the change of variable $z = y+'u_k(\epsilon) $ in the integral.

Now take $x\in \HL$ and $y\in \HR$, and recall \eqref{eqn:zcomp-dist}, by which we write
\[
\cosh(\dist(x,y-'u_k(\epsilon))) = (x-'(y-'u_k(\epsilon)))_0.
\]
Next, we will compare $(x-'(y-'u_k(\epsilon)))_0$ with $(x-'y)_0$.

Using \eqref{a-a-10}, \eqref{eqn:inverse} and \eqref{eqn:-prime}, we compute 
\begin{align*}
(x-'(y-'u_k(\epsilon)))_0&=x_0(y-'u_k(\epsilon))_0-\sum_{l=1}^\dm x_l(y-'u_k(\epsilon))_l  \\
&=x_0 (y_0 \cosh\epsilon- y_k \sinh\epsilon)-\sum_{l=1}^\dm x_l (-y_0\delta_{kl}\sinh \epsilon + y_l + y_l\delta_{kl}(\cosh \epsilon-1)) \\
&=x_0(y_0 \cosh\epsilon - y_k \sinh\epsilon) + x_k y_0 \sinh \epsilon -\sum_{l=1}^\dm x_l y_l+(1-\cosh \epsilon)x_ky_k.
\end{align*}
Furthermore, add and subtract $x_0 y_0$ to the right-hand-side above and use \eqref{eqn:xmy-z} to get
\begin{align}
&(x-'(y-'u_k(\epsilon)))_0-(x-'y)_0 \nonumber \\
=&(\cosh\epsilon-1)(x_0y_0-x_ky_k)+\sinh\epsilon(x_ky_0-x_0y_k)\\
=&2\sinh^2\left(\frac{\epsilon}{2}\right)x_0y_0\left(1-\frac{x_ky_k}{x_0y_0}\right)+2\sinh\left(\frac{\epsilon}{2}\right)\cosh\left(\frac{\epsilon}{2}\right) x_0y_0\left(\frac{x_k}{x_0}-\frac{y_k}{y_0}\right) \nonumber \\
=&2\sinh\left(\frac{\epsilon}{2}\right)x_0y_0\left(\sinh\left(\frac{\epsilon}{2}\right)\left(1-\tanh(\pi_k(x))\tanh(\pi_k(y))\right) +\cosh\left(\frac{\epsilon}{2}\right)(\tanh(\pi_k(x))-\tanh(\pi_k(y)))\right) \nonumber \\
=&2\sinh\left(\frac{\epsilon}{2}\right)x_0y_0(1-\tanh(\pi_k(x)) \tanh(\pi_k(y)))\left(\sinh\left(\frac{\epsilon}{2}\right)+\cosh\left(\frac{\epsilon}{2}\right)\tanh(\pi_k(x)-\pi_k(y))\right)\nonumber \\
=&\left(\frac{2\sinh\left(\frac{\epsilon}{2}\right)x_0y_0(1-\tanh\pi_k(x)\tanh\pi_k(y))}{\cosh(\pi_k(x)-\pi_k(y))}\right)\sinh\left(\pi_k(x)-\pi_k(y)+\frac{\epsilon}{2}\right). \label{eqn:diff-z}
\end{align}

Since
\[
\frac{2\sinh\left(\frac{\epsilon}{2}\right)x_0y_0(1-\tanh\pi_k(x)\tanh\pi_k(y))}{\cosh(\pi_k(x)-\pi_k(y))}
\]
is always positive for any $\epsilon>0$, we only have to study a sign of $\sinh\left(\pi_k(x)-\pi_k(y)+\frac{\epsilon}{2}\right)$. As $x\in H_L$, $y\in H_R$, and $\epsilon<R_6$, we have
\[
\pi_k(x)-\pi_k(y)+\frac{\epsilon}{2}<(a_k-R_6)-(a_k+R_6)+\frac{R_6}{2} = - \frac{3}{2} R_6 < 0,
\]
which by \eqref{eqn:diff-z} it implies that
\begin{equation}
\label{eqn:equiv}
(x-'(y-'u_k(\epsilon)))_0<(x-'y)_0\quad\Longleftrightarrow\quad \dist(x, y-'u_k(\epsilon))<\dist(x, y).
\end{equation}

Now apply Lemma \ref{l3.7} for $x$ and $y-'u_k(x)$, and use $\pi_k(x)<a_k-R_6$ and $\pi_k(y-'u_k(\epsilon))>a_k+R_6-\epsilon$, to get
\begin{equation}
\label{eqn:dist-ineq}
\dist(x, y-'u_k(\epsilon))>|(a_k+R_6-\epsilon)-(a_k-R_6)|=2R_6-\epsilon >R_6.
\end{equation}
By combining  \eqref{eqn:equiv} and \eqref{eqn:dist-ineq} we now get
\[
R_6<\dist(x, y-'u_k(\epsilon))<\dist(x, y), 
\]
and furthermore,
\[
K(\dist(x,y-'u_k(\epsilon)))<K(\dist(x,y)),
\]
since $K$ is strictly increasing on $[R_6, \infty)$ by Hypothesis $(\mathcal{H}6)$. 

Finally, we substitute this result into \eqref{c-10} to get
\begin{align*}
E(\bar{\rho}_\R) &=E(\rho_\R\vert_{_{\HL}})+E(\rho_\R\vert_{_{\HR}})
+\iint_{\bbh^\dm\times\bbh^\dm} K(\dist(x,y-'u_k(\epsilon))) \, \d\rho_\R\vert_{_{\HL}} (x) \, \d\rho_\R\vert_{_{\HR}}(y) \\
&< E(\rho_\R\vert_{_{\HL}})+E(\rho_\R\vert_{_{\HR}})
+\iint_{\bbh^\dm\times\bbh^\dm} K(\dist(x,y)) \, \d\rho_\R\vert_{_{\HL}} (x) \, \d\rho_\R\vert_{_{\HR}}(y) \\
&=E(\rho_\R).
\end{align*}
However,  this is a contradiction to the fact that $\rho_\R$ is a global minimizer on $\calP_\R(\bbh^\dm)$.
\end{proof}

Lemmas \ref{lemma:mass} and \ref{lemma:separation} lead to the following last result, that concludes Part 2 of the proof of Theorem \ref{Htheorem}. The statement and its proof are identical to those of \cite[Lemma 2.9]{CaCaPa2015}.

\begin{lem} \cite[Lemma 2.9]{CaCaPa2015} Assume $K$ satisfies $(\mathcal{H}1)-(\mathcal{H}6)$. Then there exists $\D>0$ (depending only on $K$ and $\dm$) such that for all $\R \geq 0$ and any global minimizer $\rho_\R$ of the energy on $\calP_\R(\bbh^\dm)$, the diameter of the support of $\rho_\R$ is bounded by $\D$.
\label{lemma:support}
\end{lem}
\begin{proof}
The proof of the lemma follows that of  \cite[Lemma 2.9]{CaCaPa2015}  by replacing the $k$-th coordinate of a point $x$ by $\pi_k(x)$ as defined in \eqref{eqn:pik}. We sketch it here. First, given that there is at least an amount of mass $m$ in any geodesic ball $B_r(x)$ with $x \in \supp(\rho_\R)$ (by Lemma \ref{lemma:mass}), the support of $\rho_\R$ can be covered by a finite number of geodesic balls $B_{2r}(x_i)$, $i=0,1,\dots,N$, with $N+1 \leq [1/m]$ ($[\cdot]$ denotes the ceiling of a number). Note that the points $x_i$ are chosen so that $x_i \notin B_{2r}(x_j)$, for any $i \neq j$, so the balls $B_r(x_i)$ are disjoint.

With $k$ fixed ($1\leq k \leq \dm$), order the points so that $\pi_k(x_0) < \dots < \pi_k(x_N)$. Then, 
\[
\pi_k(x_{i+1}) - \pi_k(x_i) \leq 4r + 2 R_6, \qquad i = 0,1,\dots N-1,
\]
as otherwise this would contradict Lemma \ref{lemma:separation} (gaps larger than $2R_6$). From here, get that
\begin{equation}
\label{eqn:pik-diff}
\pi_k(x_N) - \pi_k(x_0) \leq N (4r + 2 R_6) \leq ([1/m]-1)(4r + 2 R_6).
\end{equation}

Now set $x_l$ and $x_m$ such that 
\[
d(x_l,x_m) = \max_{0\leq i,j \leq N} d(x_i,x_j).
\]
Consider the geodesic curve between $x_l$ and $x_m$, and the midpoint between the two points on this geodesic. Take the isometry $T$ that maps the midpoint into the vertex $v$, and $x_l \mapsto z$, $x_m \mapsto -z$ (see \eqref{eqn:inverse}), where $z=(\cos \alpha, \sin \alpha, 0, \dots,0)$. Then,
\[
d(x_l,x_m) = d(z,-z)= 2\alpha,
\]
where the second equal sign comes from the fact that (see \eqref{eqn:geod-d})
\[
\dist(-z,z) = 2 \dist(v,z) = 2\cosh^{-1} (\cosh \alpha) = 2 \alpha.
\]
On the other hand, by \eqref{eqn:pik}, we get
\[
2 \alpha = \pi_1(z) - \pi_1(-z).
\]
Finally, by \eqref{eqn:pik-diff} (used for the image under the isometry $T$ of $\supp(\rho_R)$), one has
\[
\pi_1(z) - \pi_1(-z) \leq ([1/m]-1)(4r + 2 R_6).
\]

Combining these considerations, we find
\[
\text{diam}(\supp(\rho_R)) \leq  ([1/m]-1)(4r + 2 R_6) + 4r =:\D.
\] 

\end{proof}
%%%%%%%%%%%%%%%%%%%

\section{Numerical simulations}\label{sect:numerics}

In this section, we provide numerical simulations of the model \eqref{eqn:model} for some specific interaction potentials in the form \eqref{eqn:K-Na}. For this purpose, we write model \eqref{eqn:model} in characteristic form; this form of the equation is widely used for both analysis and numerical simulations \cite{BertozziCarilloLaurent, FeRa10, FeHu13, FeZh2019}.

\paragraph{Model \eqref{eqn:model} in characteristic form.} Consider the flow generated by the velocity field $\V[\rho]$ given by \eqref{eqn:v-field}. The characteristic paths $x(t)$ then satisfy
\begin{equation}
\label{eqn:flow}
\frac{ \d x}{\d t} = \V[\rho](x(t),t).
\end{equation}
To reach the characteristic form of model \eqref{eqn:model}, expand the divergence 
\[
\nabla \cdot (\rho \V[\rho]) = \nabla_{\V[\rho]} \rho + \rho \nabla \cdot \V[\rho],
\]
where  $\nabla_{\V[\rho]} \rho$ denotes the covariant derivative of $\rho$ along $\V[\rho]$. Then write \eqref{eqn:model} as
\begin{align}
\partial_t \rho +  \nabla_{\V[\rho]} \rho &= - \rho \nabla \cdot \V[\rho] \nonumber \\
&= \rho(t, x)\int_{\bbh^\dm} \Delta_x K(x, y) \rho(t, y)\d y, \label{eqn:ch-form}
\end{align}
where for the second equal sign we used \eqref{eqn:v-field}.  Note that the left-hand-side above represents the material derivative (along characteristic paths) of $\rho$.

For numerical simulations we consider only radially symmetric solutions. In such case, the characteristic paths $x(t)$ are geodesic curves passing through the vertex $v$. Denote the radial coordinate
\[
\theta_{x}:= \dist(x,v).
\]
By writing the characteristic equations \eqref{eqn:flow} in radial coordinates, we then find
\begin{align}\label{Nu2}
\frac{\d}{\d t}\theta_{x(t)}=-\nx\cdot \nabla_x \int_{\bbh^\dm} K(x, y)\rho(t, y)\d y,
\end{align}
where $\nx := \frac{x}{\theta_x}$.

The interaction potentials in the numerical simulations are in the form \eqref{eqn:K-Na}. Using \eqref{eqn:K-Na}, the radial symmetry of $\rho$, the divergence theorem and \eqref{eqn:DeltaG}, we compute starting from \eqref{Nu2}:
\begin{equation*}
%\label{Nu3}
\begin{aligned}
&\frac{\d}{\d t}\theta_{x(t)}= -\nx\cdot \nabla_x \int_{\bbh^\dm} ( G(x, y) + A(d(x,y)) \rho(t, y)\d y\\
&\quad =-\frac{1}{|\partial B_v(\theta_x)|}\int_{\partial B_v(\theta_x)}\nx\cdot \nabla_x \int_{\bbh^\dm}G(x, y)\rho(t, y)\d y \d\sigma_x-\nx\cdot \nabla_x \int_{\bbh^\dm} A(d(x, y))\rho(t, y)\d y \\
&\quad =-\frac{1}{|\partial B_v(\theta_x)|}\int_{B_v(\theta_x)}\Delta_x \int_{\bbh^\dm} G(x, y)\rho(t, y)\d y\d \sigma_x-\nx\cdot \nabla_x \int_{\bbh^\dm} A(d(x, y))\rho(t, y)\d y\\
&\quad =-\frac{1}{\dm\alpha(\dm)\sinh^{\dm-1}\theta_x}\int_{B_v(\theta_x)}\left(-\rho(t, x)\right)\d x-\nx\cdot \nabla_x \int_{\bbh^\dm} A(d(x, y))\rho(t, y)dy.
\end{aligned}
\end{equation*}
Furthermore, using the notation $\tilde{\rho}(t,\theta_x) = \rho(t,x)$ and the hyperbolic law of cosines, we write
\begin{align}
&\frac{\d}{\d t}\theta_{x(t)}
%=\frac{1}{\sinh^{\dm-1}\theta_x}\int_0^{\theta_x}\tilde{\rho}(t, \lambda)\sinh^{\dm-1}(\lambda)\d\lambda-\nx\cdot \nabla_x \int A(d(x, y))\rho(t, y)\d y \\
=\frac{1}{\sinh^{\dm-1}\theta_x}\int_0^{\theta_x}\tilde{\rho}(t, \lambda)\sinh^{\dm-1}\lambda \, \d\lambda \nonumber \\
&\quad  -\int_{\bbh^\dm} \partial_{\theta_x}A\left(\cosh^{-1}\left(\cosh\theta_x \cosh\theta_y-\sinh\theta_x \sinh\theta_y\cos\angle xvy\right)\right)\tilde{\rho}(t, \theta_y)\d y \nonumber \\
& = \frac{1}{\sinh^{\dm-1}\theta_x}\int_0^{\theta_x}\tilde{\rho}(t, \lambda)\sinh^{\dm-1}\lambda \, \d\lambda -  (\dm-1)\alpha(\dm-1) \nonumber \\
& \quad  \times \int_0^\pi\int_0^\infty \partial_{\theta_x}A\left(\cosh^{-1}\left(\cosh\theta_x \cosh\theta_y-\sinh\theta_x \sinh\theta_y\cos\beta\right)\right)\tilde{\rho}(t, \theta_y)\sinh^{\dm-1}\theta_y \sin^{\dm-2}\beta \, \d\theta_y\d\beta. \label{eqn:dthetax-dt}
\end{align}
%where the second integral in the right-hand-side can be written in radial coordinates as
%{\small
%\begin{equation}
%\begin{aligned}
%\label{eqn:rhs-s}
%&\int_{\bbh^\dm} \partial_{\theta_x}A\left(\cosh^{-1}\left(\cosh\theta_x \cosh\theta_y-\sinh\theta_x \sinh\theta_y\cos\angle xvy\right)\right)\tilde{\rho}(t, \theta_y)\d y  = (\dm-1)\alpha(\dm-1) \nonumber \\
%& \quad \times \int_0^\pi\int_0^\infty \partial_{\theta_x}A\left(\cosh^{-1}\left(\cosh\theta_x \cosh\theta_y-\sinh\theta_x \sinh\theta_y\cos\beta\right)\right)\tilde{\rho}(t, \theta_y)\sinh^{\dm-1}\theta_y \sin^{\dm-2}\beta \, \d\theta_y\d\beta. \nonumber
%\end{aligned}
%\end{equation}
%}

For the evolution of the density along characteristic paths, we calculate from \eqref{eqn:ch-form}:
\begin{align}
&\frac{\d}{\d t}\rho(t, x(t)) =\rho(t, x)\int_{\bbh^\dm} \Delta_x K(x, y)\rho(t, y)\d y\nonumber \\
&=\rho(t, x)\left(-\rho(t, x)+\int_{\bbh^\dm} \Delta_x A(d(x, y))\rho(t, y)\d y\right) \nonumber \\
&=-\rho(t, x)^2 +\rho(t, x)\int_{\bbh^\dm}\Delta A\left(\cosh^{-1}\left(\cosh\theta_x\cosh\theta_y-\sinh\theta_x\sinh\theta_y\cos\angle xvy\right)\right) \tilde{\rho}(t, \theta_y) \d y \nonumber \\
&=-\rho(t, x)^2+(\dm-1)\alpha(\dm-1)\rho(t, x) \nonumber \\
&\quad\times\int_0^\pi\int_0^\infty \Delta A\left(\cosh^{-1}\left(\cosh\theta_x\cosh\theta_y-\sinh\theta_x\sinh\theta_y\cos\beta\right)\right) \tilde{\rho}(t, \theta_y)\sinh^{\dm-1}\theta_y \sin^{\dm-2}\beta \, \d\theta_y\d\beta. \label{eqn:drho-dt}
\end{align}

Due to the radial symmetry assumption, the evolution equations \eqref{eqn:dthetax-dt} and \eqref{eqn:drho-dt} are one-dimensional. For simplicity, denote $\theta(t) = \theta_{x(t)}$ and use the notation $\tilde{\rho}$, to write \eqref{eqn:dthetax-dt} and \eqref{eqn:drho-dt} as a system of two integro-differential equations for $\theta(t)$ and $\tilde{\rho}(t,\theta(t))$:

\begin{align}\label{charsys-dm}
\begin{cases}
&\displaystyle \frac{\d}{\d t}\theta(t)=\frac{1}{\sinh^{\dm-1}\theta}\int_0^{\theta}\tilde{\rho}(t, \lambda)\sinh^{\dm-1}\lambda \, \d\lambda-(\dm-1)\alpha(\dm-1) \\[12pt]
&\quad \displaystyle \times\int_0^\pi \int_0^\infty \partial_{\theta}A\left(\cosh^{-1}\left(\cosh\theta \cosh\theta_y-\sinh\theta \sinh\theta_y\cos\beta\right)\right)\tilde{\rho}(t, \theta_y)\sinh^{\dm-1}\theta_y \sin^{\dm-2}\beta\,  \d\theta_y\d\beta \\[15pt]
&\displaystyle \frac{\d}{\d t}\tilde{\rho}(t, \theta(t))
=-\tilde{\rho}(t, \theta)^2+(\dm-1)\alpha(\dm-1)\tilde{\rho}(t, \theta)\\[12pt] 
&\quad \displaystyle \times\int_0^\pi\int_0^\infty \Delta A\left(\cosh^{-1}\left(\cosh\theta \cosh\theta_y-\sinh\theta\sinh\theta_y\cos\beta\right)\right) \tilde{\rho}(t, \theta_y)\sinh^{\dm-1}\theta_y \sin^{\dm-2}\beta \, \d\theta_y\d\beta,
\end{cases}
\end{align}
where $\Delta A(\theta)=A''(\theta)+(\dm-1)A'(\theta)\coth\theta $.

\paragraph{Numerical results.} All numerical results we present are for the hyperbolic plane ($\dm=2$), for which \eqref{charsys-dm} simplifies to
\begin{align}\label{charsys-2d}
\begin{cases}
&\displaystyle \frac{\d}{\d t}\theta(t)=\frac{1}{\sinh \theta}\int_0^{\theta}\tilde{\rho}(t, \lambda)\sinh \lambda \, \d\lambda \\[12pt]
&\qquad \displaystyle -2 \int_0^\pi \int_0^\infty \partial_{\theta}A\left(\cosh^{-1}\left(\cosh\theta \cosh\theta_y-\sinh\theta \sinh\theta_y\cos\beta\right)\right)\tilde{\rho}(t, \theta_y)\sinh \theta_y \,  \d\theta_y\d\beta \\[15pt]
&\displaystyle \frac{\d}{\d t}\tilde{\rho}(t, \theta)
=-\tilde{\rho}(t, \theta)^2 \\[12pt] 
&\qquad \displaystyle + 2 \tilde{\rho}(t, \theta)\int_0^\pi\int_0^\infty \Delta A\left(\cosh^{-1}\left(\cosh\theta \cosh\theta_y-\sinh\theta\sinh\theta_y\cos\beta\right)\right) \tilde{\rho}(t, \theta_y)\sinh \theta_y \, \d\theta_y\d\beta.
\end{cases}
\end{align}

For all simulations, the (radially symmetric) initial density $\tilde{\rho}$ is taken to be
\begin{align}\label{initialdataradial}
\tilde{\rho}_0(\theta)=e^{-5\theta}(0.01+\theta-\theta^2)/C,
\end{align}
where $C$ is a normalization constant which satisfies $\int_0^\infty \tilde{\rho}_0(\theta)\sinh\theta \d\theta=C$.  We use the classical fourth order Runge-Kutta method with $\Delta t=0.02$ to evolve \eqref{charsys-2d} in time. The initial particle positions are set at $\theta_i(0) = (i-1) \Delta x$ with $\Delta x=0.02$, $i=1,\dots, N$, with $N=51$.

We consider the interaction potentials (in the form \eqref{eqn:K-Na}) from Section \ref{subsect:explicit}, for which we have explicit expressions of the equilibria. First, take attraction given by \eqref{eqn:Ac} and \eqref{eqn:Ah}, for which we computed the explicit solutions \eqref{d-15} and \eqref{d-8}. In 
$\dm = 2$ dimensions, these solutions (written in the radial notation, using the tilde symbol) read
\begin{align}\label{exact:rhoc}
\begin{aligned}
&\trhoc(\theta)=\begin{cases}
\displaystyle1 \qquad\text{if }0\leq \theta<\cosh^{-1}\Bigl(1+\frac{1}{2\pi}\Bigr),\\[5pt]
0\qquad\text{otherwise},
\end{cases}
\end{aligned}
\end{align}
and
\begin{align}\label{exact:rhoh}
\begin{aligned}
&\trhoh(\theta)=\begin{cases}
\displaystyle\frac{\cosh\theta}{\pi\left(\left(1+\frac{3}{2\pi}\right)^{2/3}-1\right)}&\displaystyle\qquad\text{if }0\leq \theta<\cosh^{-1}\biggl( \Bigl(1+ \frac{3}{2\pi}\Bigr)^{1/3}\biggr),\\[10pt]
0&\qquad\text{otherwise}.
\end{cases}
\end{aligned}
\end{align}
Note the monotonicity of these equilibria, which is consistent with Theorem \ref{thm:mov-planes}.

The time evolution of the model with Newtonian repulsion and attraction \eqref{eqn:Ac} and \eqref{eqn:Ah}, is illustrated in Figure \ref{fig:AcAh}. In both cases, the solution approaches the equilibria in \eqref{exact:rhoc} and \eqref{exact:rhoh}, respectively. The simulation confirms the theoretical findings in Section \ref{sec:3.3} (see also Theorem \ref{thm:convexity}): the energy functional is strictly convex and solutions approach asymptotically the global energy minimizer. On the other hand, numerical simulations with a variety of other initial conditions suggest that the equilibrium \eqref{exact:rhoc} is also a global attractor; this was first conjectured in \cite{FeZh2019}. 
\begin{figure}
\begin{center}
\begin{tabular}{cc}
\includegraphics[scale=0.45]{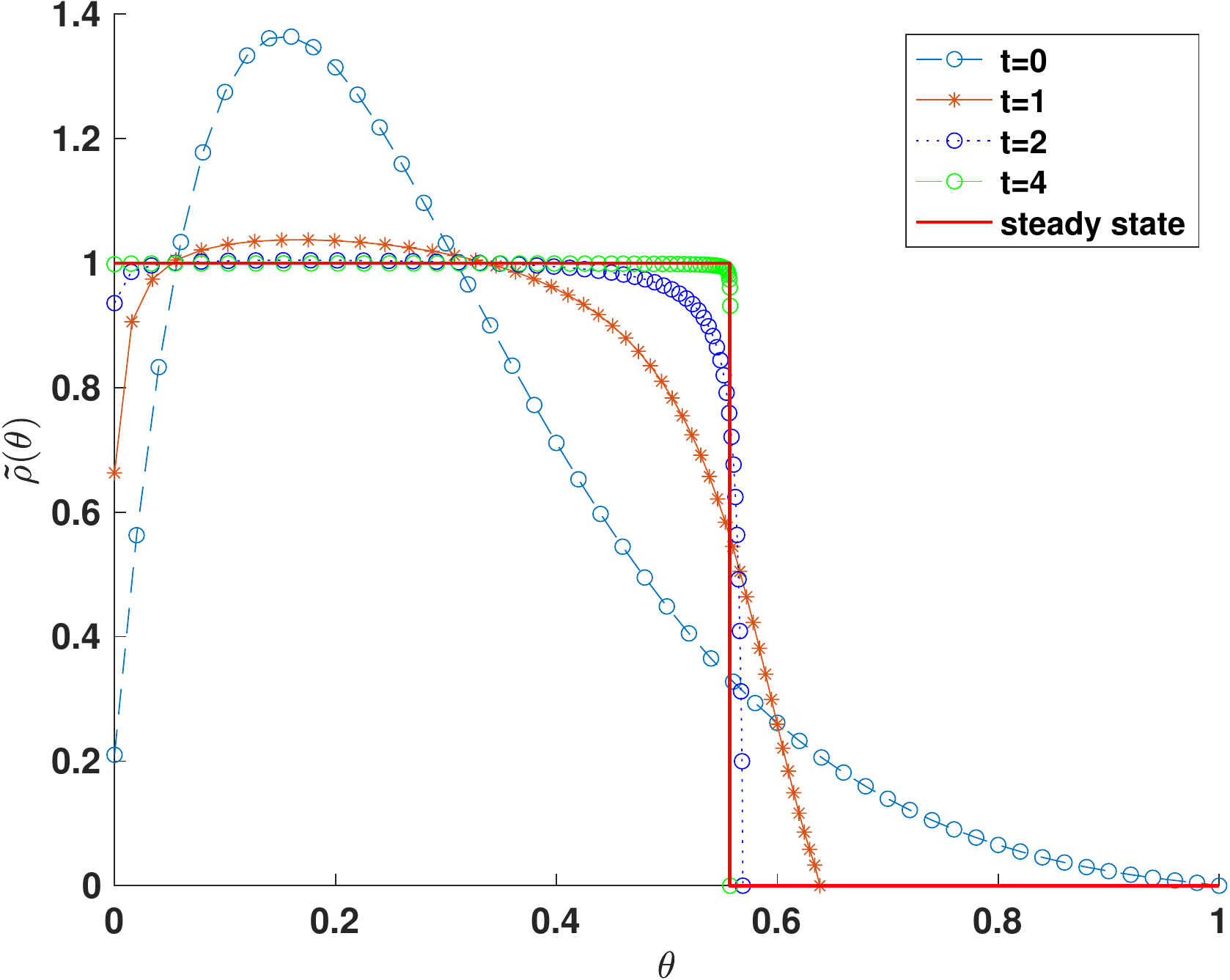}&
\includegraphics[scale=0.45]{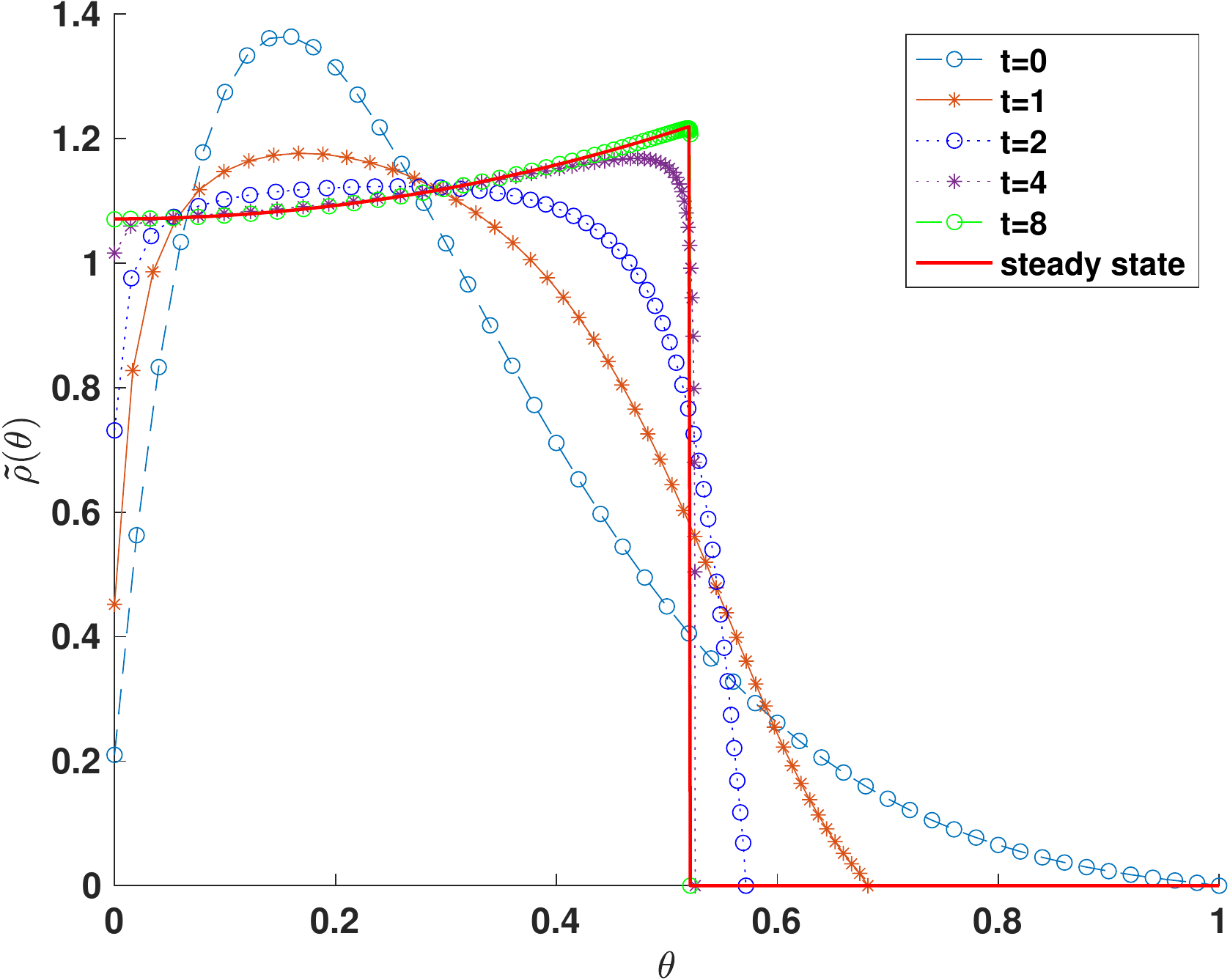} \\
 (a) & (b) 
\end{tabular}
\end{center}
\caption{Time evolution of a radially symmetric solution to the aggregation model \eqref{eqn:model} on $\bbh^2$, for $K$ in the form \eqref{eqn:K-Na}, with attraction given by (a) $A = \Ac$  and (b) $A=\Ah$ -- see \eqref{eqn:Ac} and \eqref{eqn:Ah}. The initial density (blue dashed-line marked with circles) is given in \eqref{initialdataradial}. The solutions approach asymptotically the equilibria \eqref{exact:rhoc} and \eqref{exact:rhoh}, respectively, indicated by red solid lines.}
%Numerics with a variety of other initial conditions suggests that these equilibria are global attractors for the dynamics of \eqref{eqn:model}
\label{fig:AcAh}
\end{figure}

For numerical simulations, we also considered linear combinations of the potentials $\Ac$ and $\Ah$, i.e., potentials in the form 
\begin{align}\label{lcpot}
\Am:=b_0 \Ac +b_1 \Ah,
\end{align}
with $b_0,b_1$ constants. Note that $\Delta \Am(\theta)=b_0+b_1\cosh\theta$. 

The following proposition provides an explicit expression for the equilibria corresponding to mixed attraction $\Am$.
\begin{prop}\label{P5.1}
Consider the aggregation model \eqref{eqn:model} on $\bbh^2$, with $K(x, y)=G(x, y)+ \Am(d(x, y))$, and $b_0,b_1\geq0$ such that $(b_0, b_1)\neq(0, 0)$. Then, there exists a unique radially symmetric equilibrium in the form
\[
\trhom(\theta)=\begin{cases}
\displaystyle  b_0+\left(\frac{1-2\pi b_0\alpha_{b_0b_1}}{\pi\alpha_{b_0b_1}(\alpha_{b_0b_1}+2)}\right)\cosh\theta \qquad&\text{ if }0\leq \theta\leq \cosh^{-1}(\alpha_{b_0b_1}+1),\\[10pt]
0 \qquad&\text{ if }\theta>\cosh^{-1}(\alpha_{b_0b_1}+1),
\end{cases}
\]
where $\alpha_{b_0b_1}$ is a solution of the polynomial
\[
1-2\pi (b_0+b_1) x-2\pi b_1x^2-\frac{2\pi b_1}{3}x^3+\frac{\pi^2 b_0b_1}{3}x^4=0.
\]
\end{prop}

\begin{proof}
See Appendix \ref{sec:app.b}.
\end{proof}

\begin{remark}
While we cannot prove Proposition \ref{P5.1} for general $(b_0, b_1)\in\bbr^2\backslash \{(0, 0)\}$, similar arguments can be used for some specific cases. For example, the result in Proposition \ref{P5.1} holds for $(b_0, b_1)=(2, -1)$, although $b_1<0$. We will present below numerical simulations for $(b_0, b_1)=(2, -1)$, as in this case the equilibrium solution $\trhom$ is qualitatively different from $\trhoc$ and $\trhoh$. Indeed, $\trhom$ is monotonically decreasing, following the monotonicity of $\Delta \Am$ (see Theorem \ref{thm:mov-planes}).
\end{remark}

For $(b_0, b_1)= (2, -1)$, the equilibrium in Proposition  \ref{P5.1} can be calculated as
\begin{align}\label{exactsolutions}
\begin{aligned}
%&\tilde{\rho}_{1/2, 1/2}(\theta)\simeq\begin{cases}
%\displaystyle\frac{1}{2}+0.5374\cosh\theta,&\displaystyle\quad\text{if }0\leq \theta<0.5371,\\
%0,&\quad\text{otherwise},
%\end{cases}\\
&\trhom(\theta)\simeq\begin{cases}
2-1.0956\cosh\theta &\qquad\text{ if }0\leq \theta<0.6227,\\
0&\qquad\text{ otherwise}.
\end{cases}
\end{aligned}
\end{align}
Figure \ref{fig:Am}(a) shows the time evolution of the initial density \eqref{initialdataradial}, for a potential with mixed attraction \eqref{lcpot}, and $(b_0, b_1)= (2, -1)$. The solution approach asymptotically the steady state $\trhom$. Simulations with a variety of other initial densities (including 
non-symmetric ones) suggest that $\trhom$ is in fact a global attractor. We clarify here that for non-symmetric densities we use a particle method and evolve in time the discrete analogue of \eqref{eqn:model} -- see \cite{FeZh2019} for the particle method formulation of the model \eqref{eqn:model} on $\bbh^\dm$. Indeed, starting from random distributions, particles approach a symmetric steady state supported on a geodesic ball (we do not present these simulations here).

Finally, we provide a simulation with an interaction potential in the form \eqref{eqn:K-Na} for which $\Delta A$ is not monotone (so Theorem \ref{thm:mov-planes} does not hold). Specifically, we take
\begin{equation}
\label{eqn:A-non-monotone}
A(\theta) =\frac{1}{6}\sinh^2\theta-\frac{3}{2}\cosh\theta+6\ln\left(\cosh\frac{\theta}{2}\right),
\end{equation}
for which $\Delta A(\theta)=\cosh^2\theta-3\cosh\theta+8/3$. Figure \ref{fig:Am}(b) shows this simulation, starting from the same initial density \eqref{initialdataradial}. Note that in this case, the steady state is not monotone, as in the previous examples. Also, we do not have any explicit form for  the steady state, or any result about its uniqueness; this is simply a numerically obtained equilibrium. Nevertheless, simulations with various other symmetric and non-symmetric initial data suggest a large basin of attraction for this equilibrium, possibly being a global attractor as well.
\begin{figure}
 \begin{center}
 \begin{tabular}{cc}
 \includegraphics[scale=0.45]{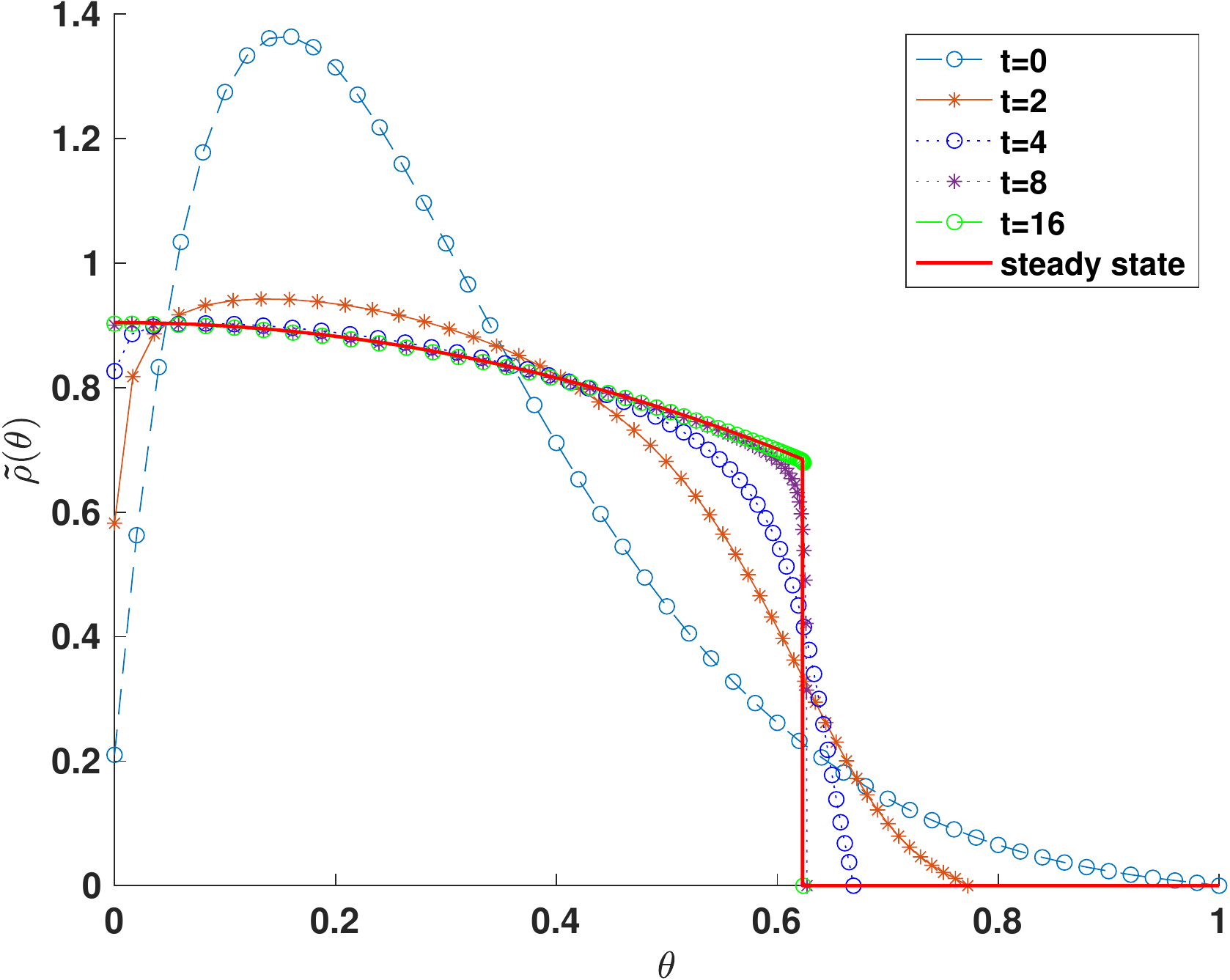}  &
 \includegraphics[scale=0.45]{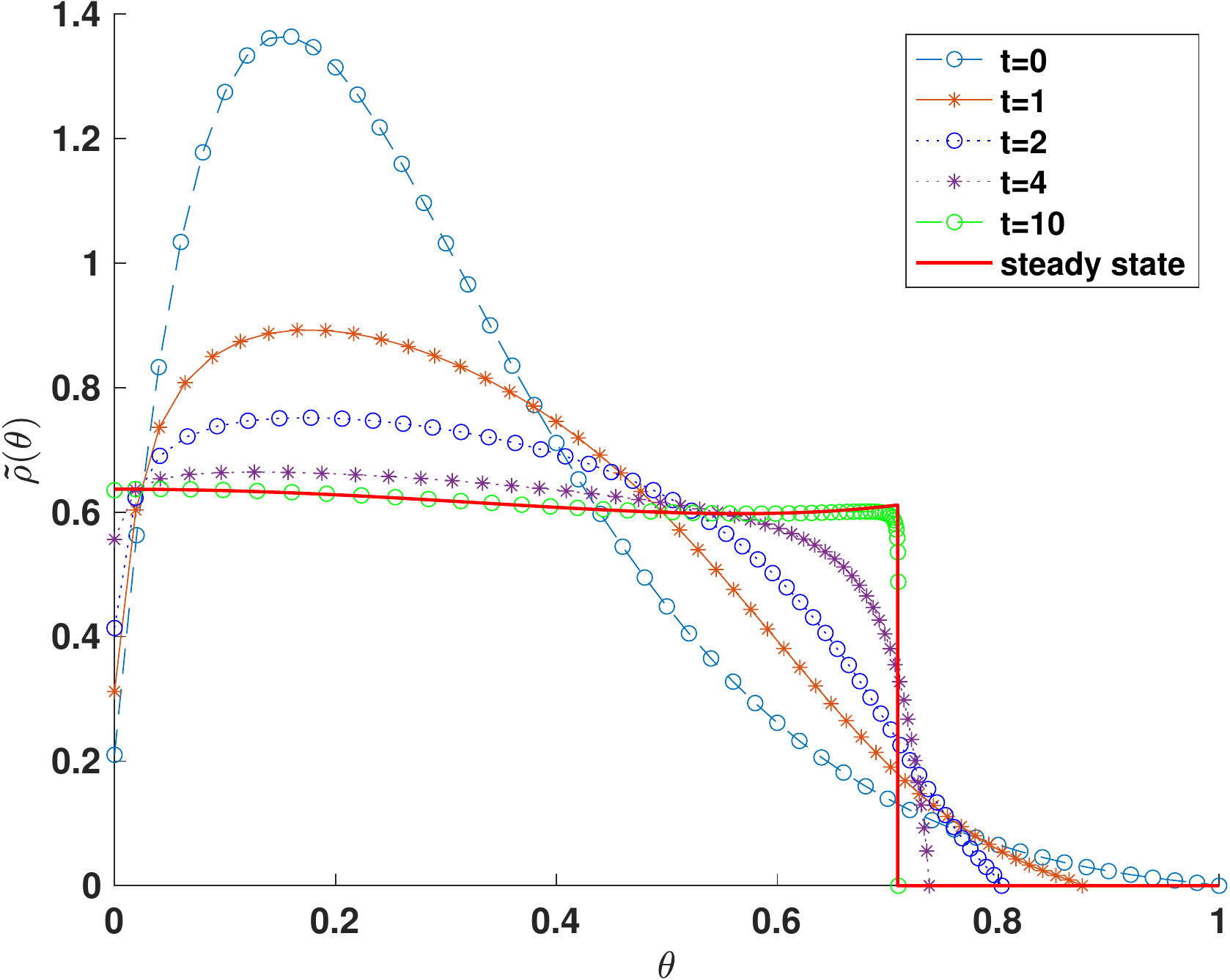}  
\\
 (a)  & (b) 
 \end{tabular}
 \begin{center}
 \end{center}
\caption{Time evolution of a radially symmetric solution to the aggregation model \eqref{eqn:model} on $\bbh^2$, for $K$ in the form \eqref{eqn:K-Na}, with attraction given by (a) $A= \Am$ with $(b_0, b_1)=(2, -1)$ -- see \eqref{lcpot}, and (b) $A$ given by \eqref{eqn:A-non-monotone}. The initial data is given by \eqref{initialdataradial}.  For simulation (a), the solution approaches asymptotically the equilibrium \eqref{exactsolutions}; note that the steady state is monotonically decreasing about the vertex, consistent with the result in Theorem \ref{thm:mov-planes}. The steady state for the simulation (b) is not monotone; this case is not covered by Theorem \ref{thm:mov-planes}.}
\label{fig:Am}
\end{center}
\end{figure}

%%%%%%%%%%%%%%%

\bibliographystyle{abbrv}
\def\url#1{}
\bibliography{lit.bib}

\appendix
\section{Proofs of several lemmas from Section \ref{sec:2.2}}\label{sec:app.a}

\begin{proof}[Proof of Lemma \ref{Leme}]
From the definition of the matrix product, we have
\[
y_{\alpha}=\sum_{\beta=0}^\dm B(w)_{\alpha\beta}x_\beta, \qquad\forall~\alpha=0, 1, \cdots, \dm.
\]
We can simplify it as follows:
\[
y_0=\sum_{\beta=0}^\dm B(w)_{0\beta}x_\beta=B(w)_{00}x_0+\sum_{i=1}^\dm B(w)_{0i}x_i=\gamma x_0-\gamma \sum_{i=1}^\dm w_i x_i,
\]
and for $j =1,\dots,\dm$,
\begin{align*}
y_j&=\sum_{\beta=0}^\dm B(w)_{j\beta}x_\beta=B(w)_{j0}x_0+\sum_{i=1}^\dm B(w)_{ji}x_i \\
& =-\gamma w_jx_0+\sum_{i=1}^\dm \left(\delta_{ij}+(\gamma-1)\frac{w_iw_j}{|w|^2}\right)x_i\\
&=-\gamma w_jx_0+x_j+(\gamma-1)\frac{w_j}{|w|^2}\sum_{i=1}^\dm x_iw_i.
\end{align*}
This yields
\begin{align*}
&y_0^2-y_1^2-\cdots-y_\dm^2\\
&=\gamma^2\left(x_0-\sum_{i=1}^\dm w_ix_i\right)^2-\sum_{j=1}^\dm \left(-\gamma w_jx_0+x_j+(\gamma-1)\frac{w_j}{|w|^2}\sum_{i=1}^\dm x_iw_i\right)^2\\
&=\gamma^2 x_0^2-2\gamma^2x_0\sum_{i=1}^\dm w_ix_i+\gamma^2\left(\sum_{i=1}^\dm w_ix_i\right)^2-\sum_{j=1}^\dm\left(
\gamma^2w_j^2x_0^2+x_j^2+(\gamma-1)^2\frac{w_j^2}{|w|^4}\left(\sum_{i=1}^\dm x_iw_i\right)^2
\right)\\
&\quad -\sum_{j=1}^\dm\left(
-2\gamma w_jx_jx_0-2\gamma(\gamma-1)\frac{w_j^2x_0}{|w|^2}\sum_{i=1}^\dm x_iw_i+2(\gamma-1)\frac{x_jw_j}{|w|^2}\sum_{i=1}^\dm x_iw_i
\right)\\
&=\gamma^2 x_0^2-2\gamma^2x_0\sum_{i=1}^\dm w_ix_i+\gamma^2\left(\sum_{i=1}^\dm w_ix_i\right)^2
-\gamma^2x_0^2|w|^2-\sum_{j=1}^\dm x_j^2-(\gamma-1)^2\frac{1}{|w|^2}\left(\sum_{i=1}^\dm x_iw_i\right)^2\\
&\quad +2 x_0\gamma\sum_{i=1}^\dm x_iw_i+2\gamma(\gamma-1)x_0\sum_{i=1}^\dm x_iw_i-2(\gamma-1)\frac{1}{|w|^2}\left(\sum_{i=1}^\dm x_iw_i\right)^2\\
&=\gamma^2 x_0^2+\gamma^2\left(\sum_{i=1}^\dm w_ix_i\right)^2
-\gamma^2x_0^2|w|^2-\sum_{j=1}^\dm x_j^2-(\gamma^2-1)\frac{1}{|w|^2}\left(\sum_{i=1}^\dm x_iw_i\right)^2\\
&=\gamma^2(1-|w|^2)x_0^2-\sum_{j=1}^\dm  x_j^2=x_0^2-x_1^2-\cdots-x_\dm^2.
\end{align*}
If we use $x\in\bbh^\dm$, then we have $x_0^2-x_1^2-\cdots-x_\dm^2=1$. This implies that
\[
y_0^2-y_1^2-\cdots-y_\dm^2=1,
\]
which is the desired result.
\end{proof}
\medskip

\begin{proof}[Proof of Lemma \ref{Lem-iso1}]
Let $a, b$ be tangent vectors at $x\in\bbh^\dm$. Since $\bbh^\dm$ is embedded in $\bbr^{\dm+1}$, tangent vectors $a$ and $b$ can also be considered as vectors in $\bbr^{\dm+1}$ which satisfy
\[
a_0x_0-a_1x_1-\cdots-a_\dm x_\dm=b_0x_0-b_1x_1-\cdots-b_\dm x_\dm=0.
\]
The induced metric defines the inner product between two tangent vectors as 
\[
a\cdot b=a_0b_0-a_1b_1-\cdots-a_\dm b_\dm =- a^\top \eta b.
\]
From the definition of the push forward of a tangent vector, we have
\[
dF_w(a)=B(w)a,\quad dF_w(b)=B(w)b.
\]
This yields
\[
dF_w(a)\cdot dF_w(b)=-(dF_w(a))^\top \eta\, dF_w(b)=-a^\top B(w)^\top \eta B(w) b=-a^\top\eta b=a\cdot b,
\]
where we used \eqref{LB-1} for the third equal sign. Finally, this implies that $F_w$ is an isometry on $\bbh^\dm$ for all $w$ such that $|w|<1$.
\end{proof}
\medskip

\begin{proof}[Proof of Lemma \ref{La.3}]
By a simple calculation, for any $w\in\bbr^\dm$ with $|w|<1$, we know the vertex of $\bbh^\dm$ maps to
\[
F_w(1, 0, \cdots, 0)=(\gamma, -\gamma w_1, \cdots, -\gamma w_\dm).
\]
The equation above shows that $1 \iff 3$.

Now, we investigate the relationship between $x\in\bbh^\dm$ and $w\in\bbr^\dm$ so that:
\[
F_w(x)=(1, 0, \cdots, 0).
\]

From the definition of $F_w$, we have
\begin{align}\label{a-a-1}
1=\gamma x_0-\gamma(w_1x_1+\cdots+w_\dm x_\dm),
\end{align}
and
\begin{align}\label{a-a-2}
0=-\gamma w_i x_0+x_i+(\gamma-1)\frac{w_i}{|w|^2}\sum_{j=1}^\dm w_jx_j,\quad \forall~i=1,2 , \cdots, \dm.
\end{align}
Equation \eqref{a-a-2} yields
\begin{align}\label{a-a-3}
w_i=\alpha x_i,\qquad\forall~i=1, 2, \cdots, \dm,
\end{align}
for some $\alpha\in\bbr$. We substitute \eqref{a-a-3} into \eqref{a-a-1} to get
\[
1=\gamma x_0-\gamma\alpha(x_1^2+\cdots+x_\dm^2)=\gamma x_0-\gamma\alpha(x_0^2-1),
\]
which gives
\begin{align}\label{a-a-4}
\gamma=\frac{1}{x_0-\alpha(x_0^2-1)}.
\end{align}

By \eqref{a-a-3}, we also have:
\[
|w|^2=w_1^2+w_2^2+\cdots+w_\dm^2=\alpha^2(x_1^2+\cdots+x_\dm^2)=\alpha^2(x_0^2-1),
\]
where for the last equality we used that $x\in\bbh^\dm$. From the definition of $\gamma$ and the above relation, we get
\begin{align}\label{a-a-5}
\gamma^2=\frac{1}{1-|w|^2}=\frac{1}{1-\alpha^2(x_0^2-1)}.
\end{align}
Finally, we compare \eqref{a-a-4} and \eqref{a-a-5} to get
\[
(x_0-\alpha(x_0^2-1))^2=1-\alpha^2(x_0^2-1),
\]
which is equivalent to
\begin{equation}
\label{eqn:x0-alpha}
(x_0^2-1)(1-\alpha x_0)^2=0.
\end{equation}

If $w=0$, then $F_w$ is the identity map. So, in this case, we have $x=(1, 0, \cdots, 0)$. Now assume that $|w|>0$.
Then, \eqref{a-a-3} implies $x_i\neq 0$ for some $1\leq i\leq \dm$, which yields
\[
x_0^2=1+x_1^2+\cdots+ x_\dm^2\geq1+x_i^2>1.
\]
Hence, from \eqref{eqn:x0-alpha}, we get the unique solution
\[
\alpha=\frac{1}{x_0},
\]
and this gives the relation
\[
w=\left(\frac{x_1}{x_0}, \cdots, \frac{x_\dm}{x_0}\right).
\]
Note that this relation also holds for $w=0$. We have thus shown that $2 \iff 3$, which concludes the proof.
\end{proof}
\medskip

\begin{proof}[Proof of Lemma \ref{Lemexf}]
First, we have
\begin{align*}
\gamma(-\hat{y})&=(1-\|\hat{y}\|^2)^{-1/2}=\left(1-\frac{y_1^2}{y_0^2}-\cdots-\frac{y_\dm^2}{y_0^2}\right)^{-1/2}\\
&=\frac{y_0}{y_0^2-y_1^2-\cdots-y_\dm^2}=y_0.
\end{align*}

From direct calculations, we have
\begin{align*}
(x+'y)_0&=(B(-\hat{y})x)_0=\sum_{\alpha=0}^\dm [B(-\hat{y})]_{0\alpha}x_\alpha=y_0x_0+\sum_{j=1}^\dm (-\gamma(-\hat{y}))(-\hat{y}_j)x_j\\
&=x_0y_0+x_1y_1+\cdots+x_\dm y_\dm.
\end{align*}
This is the first desired result. By a similar calculation, we get for $j =1,\dots,\dm$,
\begin{align*}
(x+'y)_j&=\sum_{\alpha=0}^\dm [B(-\hat{y})]_{j\alpha}x_\alpha=[B(-\hat{y})]_{j0}x_0+\sum_{k=1}^\dm [B(-\hat{y})]_{jk}x_k\\
&=\gamma(-\hat{y})\hat{y}_jx_0+\sum_{k=1}^\dm \left(\delta_{jk}+(\gamma(-\hat{y})-1)\frac{\hat{y}_j\hat{y}_k}{\|\hat{y}\|^2}\right)x_k
=x_0y_j+x_j+(y_0-1)\sum_{k=1}^\dm \frac{y_jy_kx_k}{y_0^2\|\hat{y}\|^2}\\
&=x_0y_j+x_j+(y_0-1)\sum_{k=1}^\dm \frac{y_jy_kx_k}{y_0^2-1}\\
&=x_0y_j+x_j+\frac{y_j}{y_0+1}(x_1y_1+\cdots+x_\dm y_\dm).
\end{align*}
This is the second part of the desired result.
\end{proof}
\medskip

\begin{proof}[Proof of Lemma \ref{msl}]
From direct calculations, we have
\begin{align*}
(-(x+'y))_0=(x+'y)_0=x_0y_0+x_1y_1+\cdots+x_\dm y_\dm,
\end{align*}
and
\begin{align*}
(-(x+'y))_j=-(x+'y)_j=-x_0y_j-x_j-\frac{y_j}{y_0+1}(x_1y_1+\cdots+x_\dm y_\dm).
\end{align*}
Also, we have
\[
((-x)+'(-y)')_0=x_0y_0+x_1y_1+\cdots+x_\dm y_\dm,
\]
and
\[
((-x)+'(-y)')_j=-x_0y_j-x_j-\frac{y_j}{y_0+1}(x_1y_1+\cdots+x_\dm y_\dm).
\]
Since we have $(-(x+'y))_0=((-x)+'(-y))_0$ and $(-(x+'y)_j)=((-x)+'(-y))_j$ for all $1\leq j\leq \dm$, we obtain the desired result.
\end{proof}
\medskip

\begin{proof}[Proof of Lemma \ref{La.8}]
First, we have the following calculation:
\begin{align}
\begin{aligned}\label{a-a-12}
((x+'y)+'z)_0&=(x+'y)_0z_0+\sum_{j=1}^\dm  (x+'y)_jz_j\\
&=x_0y_0z_0+z_0\sum_{j=1}^\dm x_jy_j+\sum_{j=1}^\dm z_j\left(x_0y_j+x_j+\frac{y_j}{y_0+1}\sum_{k=1}^\dm x_ky_k\right)\\
&=x_0y_0z_0+z_0\sum_{j=1}^\dm x_jy_j+x_0\sum_{j=1}^\dm y_jz_j+\sum_{j=1}^\dm  x_jz_j+\frac{1}{y_0+1}\sum_{j=1}^\dm  y_jz_j\sum_{k=1}^\dm x_ky_k.
\end{aligned}
\end{align}

We substitute $z=-y$ into \eqref{a-a-12} to get
\[
((x+'y)-'y)_0=x_0y_0^2+y_0\sum_{j=1}^\dm x_jy_j-x_0\sum_{j=1}^\dm y_j^2-\sum_{j=1}^\dm x_jy_j+\frac{1}{y_0+1}\sum_{j=1}^\dm (-y_j^2)\sum_{k=1}^\dm x_ky_k.
\]
If we use $y_0^2-y_1^2-\cdots-y_\dm^2=1$, then we can simplify the above relation to get
\[
((x+'y)-'y)_0=x_0.
\]
From the above result and \eqref{a-a-11}, we find for $j =1,\dots,\dm$,
\begin{align*}
((x+'y)-'y)_j&=(x+'y)_j-\frac{y_j}{y_0+1}(((x+'y)-'y)_0+(x+'y)_0)\\
&=x_j+\frac{y_j}{y_0+1}((x+'y)_0+x_0)-\frac{y_j}{y_0+1}(x_0+(x+'y)_0)\\
&=x_j.
\end{align*}
We infer that $(x+'y)-'y=x$.
\end{proof}
\medskip

\begin{proof}[Proof of Lemma \ref{l3.2}]
Set
\[
x=(x_0, x_1, \cdots, x_\dm)\in\bbh^\dm,\quad x_k=0.
\]
By a simple calculation, we have
\begin{align}
\begin{aligned}\label{c-2}
(x+'u_k(\ta))_0&=x_0\cosh \ta+x_k\sinh \ta,\\
(x+'u_k(\ta))_j&=x_0\sinh \ta\, \delta_{jk}+x_j+\frac{\sinh \ta\, \delta_{jk}}{\cosh \ta+1}(\sinh \ta \, x_k)\\
&=x_j(1+\delta_{jk}(\cosh \ta-1))+x_0\delta_{jk}\sinh \ta, \qquad\forall 1\leq j\leq \dm,
\end{aligned}
\end{align}
where $\delta$ is the Kronecker delta symbol. We substitute $\ta+\sa$ instead of $\ta$ in \eqref{c-2} to get
\begin{align}
\begin{aligned}\label{c-3}
(x+'u_k(\ta+\sa))_0&=x_0\cosh (\ta+\sa)+x_k\sinh (\ta+\sa),\\
(x+'u_k(\ta+\sa))_j&=x_j(1+\delta_{jk}(\cosh (\ta+\sa)-1))+x_0\delta_{jk}\sinh (\ta+\sa).
\end{aligned}
\end{align}
We substitute $x+'u_k(\ta)$ instead of $x$ and $\sa$ instead of $\ta$ into \eqref{c-2} to get
\begin{align}
\begin{aligned}\label{c-4}
((x+'u_k(\ta))+'u_k(\sa))_0&=(x+'u_k(\ta))_0\cosh \sa+(x+'u_k(\ta))_k\sinh \sa,\\
((x+'u_k(\ta))+'u_k(\sa))_j&=(x+'u_k(\ta))_j(1+\delta_{jk}(\cosh \sa-1))+(x+'u_k(\ta))_0\delta_{jk}\sinh \sa.
\end{aligned}
\end{align}

From the first equality of \eqref{c-4} and \eqref{c-2}, we have
\begin{align}
((x+'u_k(\ta))+'u_k(\sa))_0
%&=(x+'u_k(\ta))_0\cosh \sa+(x+'u_k(\ta))_k\sinh \sa\\
&=(x_0\cosh \ta+x_k\sinh \ta)\cosh \sa+(x_k\cosh \ta+x_0\sinh \ta)\sinh \sa \nonumber \\
&=x_0\cosh(\ta+\sa)+x_k\sinh(\ta+\sa). \label{c-5}
\end{align}
Also, from the second equality of \eqref{c-4} and \eqref{c-2}, we have
\begin{align}
((x+'u_k(\ta))+'u_k(\sa))_j
%&=(x+'u_k(\ta))_j(1+\delta_{jk}(\cosh \sa-1))+(x+'u_k(\ta))_0\delta_{jk}\sinh \sa\\
&=\big(x_j(1+\delta_{jk}(\cosh \ta-1))+x_0\delta_{jk}\sinh \ta\big)
(1+\delta_{jk}(\cosh \sa-1)) \nonumber \\
&\quad+(x_0\cosh \ta+x_k\sinh \ta)\delta_{jk}\sinh \sa \nonumber \\
&=x_j\big(
(1+\delta_{jk}(\cosh \ta-1))(1+\delta_{jk}(\cosh \sa-1))+\sinh \ta\delta_{jk}\sinh \sa
\big) \nonumber \\
&\quad +x_0\big(
\delta_{jk}\sinh \ta(1+\delta_{jk}(\cosh \sa-1))+\cosh \ta\delta_{jk}\sinh \sa
\big) \nonumber \\
&=x_j(1+\delta_{jk}(\cosh(\ta+\sa)-1)+x_0\delta_{jk} \sinh(\ta+\sa). \label{c-6}
\end{align}
Finally, we combine \eqref{c-3}, \eqref{c-5}, and \eqref{c-6} to conclude that
\[
(x+'u_k(\ta))+'u_k(\sa)=x+'u_k(\ta+\sa).
\]
\end{proof}
\medskip

\begin{proof}[Proof of Lemma \ref{l3.4}]
$\diamond$ (Existence of the decomposition) First, we will show the existence of the decomposition \eqref{decomp}. Set $x=(x_0, x_1, \cdots, x_\dm)$, and define
\[
\ta=\tanh^{-1}\left(\frac{x_k}{x_0}\right).
\]
Then we can check that
\begin{align*}
(x-'u_k(\ta))_k&=((x_0, x_1, \cdots, x_\dm)-'(\cosh \ta, 0, \cdots, 0, \sinh \ta, 0, \cdots, 0))_k\\
&=((x_0, x_1, \cdots, x_\dm)+'(\cosh \ta, 0, \cdots, 0, -\sinh \ta, 0, \cdots, 0))_k\\
&=-x_0\sinh \ta+x_k+\frac{-\sinh \ta}{\cosh \ta+1}(- x_k \sinh \ta )\\
&=x_k \cosh \ta - x_0\sinh \ta =0.
\end{align*}
Therefore, 
\[
x-'u_k(\ta)\in P_k(0).
\]
If we set $y=x-'u_k(\ta)$, then we have
\[
y+'u_k(\ta)=(x-'u_k(\ta))+'u_k(\ta)=x,
\]
by using Lemma \ref{La.8}. \\

\noindent$\diamond$ (Uniqueness of the decomposition) Now, we show the uniqueness of a pair $(\ta, y)$ introduced in the existence part. If there exists another pair $(\sa, z)\in \bbr\times P_k(0)$ which satisfies
\[
x=z+'u_k(\sa),
\]
then we have
\[
y+'u_k(\ta)=z+'u_k(\sa).
\]
From Lemma \ref{l3.2}, we get
\[
(y+'u_k(\ta))-'u_k(\ta)=(z+'u_k(\sa))-'u_k(\ta)\quad\Rightarrow\quad y=z+'u_k(\sa-\ta).
\]
This implies that 
\begin{align*}
y_k=z_0\sinh(\sa-\ta)+z_k\cosh(\sa-\ta).
\end{align*}
Since $y_k=z_k=0$ and $z_0\geq1$, we have $\ta=\sa$. This yields $y=z$, showing the uniqueness of the decomposition \eqref{decomp}.
\end{proof}
\medskip

\begin{proof}[Proof of Lemma \ref{l3.7}]
Without loss of generality, we assume that $\pi_k(x)\leq\pi_k(y)$. Then, there exists $R$ such that
\[
\pi_k(x)=a_k-R, \qquad\text{and}\qquad\pi_k(y)=a_k+R.
\]
If we define
\[
\tilde{x}=x-'u_k(a_k),\quad \tilde{y}=y-'u_k(a_k),
\]
we have
\begin{align*}
\begin{cases}
\pi_k(\tilde{x})=\pi_k(x-'u_k(a_k))=\pi_k(x)-a_k=-R,\\
\pi_k(\tilde{y})=\pi_k(y-'u_k(a_k))=\pi_k(y)-a_k=R,
\end{cases}
\end{align*}
by Corollary \ref{c3.5}. On the other hand,
\[
\pi_k(\tilde{x})=\tanh^{-1}\left(\frac{\tilde{x}_k}{\tilde{x}_0}\right)\quad\text{and}\quad \pi_k(\tilde{y})=\tanh^{-1}\left(\frac{\tilde{y}_k}{\tilde{y}_0}\right)
\]
which implies
\[
\frac{\tilde{x}_k}{\tilde{x}_0}=-\tanh R\quad\text{and}\quad\frac{\tilde{y}_k}{\tilde{y}_0}=\tanh R.
\]
We use the above relation and the following facts:
\[
\tilde{x}_0\geq\sqrt{1+\tilde{x}_k^2}\quad\text{and}\quad\tilde{y}_0\geq\sqrt{1+\tilde{y}_k^2},
\]
to get
\[
\tanh (\sinh^{-1}\tilde{x}_k)=\frac{\tilde{x}_k}{\sqrt{1+\tilde{x}_k^2}}\leq -\tanh R\quad\text{and}\quad\tanh (\sinh^{-1}\tilde{y}_k)=\frac{\tilde{y}_k}{\sqrt{1+\tilde{y}_k^2}}\geq \tanh R.
\]
Since $\tanh$ and $\sinh$ are increasing function, we get
\[
\tilde{x}_k\leq -\sinh R\quad\text{and}\quad \tilde{y}_k\geq \sinh R. 
\]
We substitute the above result into the following calculation:
\begin{align*}
(\tilde{x}-'\tilde{y})_0=&=\tilde{x}_0\tilde{y}_0-\tilde{x}_1\tilde{y}_1-\cdots-\tilde{x}_\dm \tilde{y}_\dm\\
&=\sqrt{(1+\tilde x_1^2+\cdots+\tilde x_\dm^2)(1+\tilde y_1^2+\cdots+\tilde{y}_\dm^2)}-\tilde x_1\tilde y_1-\cdots-\tilde x_\dm\tilde y_\dm\\
&\geq 1+\sum_{l=1}^\dm |\tilde x_l\tilde y_l|-\sum_{l=1}^\dm\tilde x_l\tilde y_l\geq 1-2\tilde x_k\tilde y_k,
\end{align*}
to get
\[
(\tilde{x}-'\tilde{y})_0\geq 1+2\sinh^2R=\cosh(2R).
\]
Since $z \mapsto z-'u_k(a_k)$ is an isometry on $\bbh^\dm$, we finally get
\[
\cosh(\dist(x, y))=(x-'y)_0=(\tilde{x}-'\tilde{y})_0\geq\cosh(2R),
\]
which yields
\[
\dist(x, y)\geq 2R=|\pi_k(x)-\pi_k(y)|.
\]
\end{proof}

%%%%%%%%%%

\section{Proof of Proposition \ref{P5.1}}\label{sec:app.b}
We look for a solution of (see \eqref{eqn:int-eq-gen0}):
\begin{equation}
\label{eqn:int-radial}
\tilde{\rho}(\theta_x)=\int_{B_v(R)}\Delta_xA(\theta_{xy})\tilde{\rho}(\theta_y)\d y, \qquad \text{ for } 0 \leq \theta_x \leq R,
\end{equation}
where $\theta_x=\mathrm{dist}(x, v)$ and $\theta_{xy}=\mathrm{dist}(x, y)$. Note that the radius $R$ is also unknown.

We first show that if $\Delta A(\theta)$ is in the form of a sum of powers of $\cosh \theta$, then necessarily, $\trho(\theta)$ is in the same form. Therefore, assume
\[
\Delta A(\theta):=\sum_{m=0}^N b_m \cosh^m\theta.
\]

Recall the hyperbolic cosine law:
\[
\cosh \theta_{xy}=\cosh\theta_x\cosh\theta_y-\sinh\theta_x\sinh\theta_y\cos\angle(xvy).
\]
This yields that
\begin{align}
\begin{aligned}\label{npowercosh}
\cosh^m\theta_{xy}&=\left(\cosh\theta_x\cosh\theta_y-\sinh\theta_x\sinh\theta_y\cos\angle(xvy)\right)^m\\
&=\sum_{k=0}^m(-1)^k{m \choose k}\cosh^{m-k}\theta_x\cosh^{m-k}\theta_y\sinh^k \theta_x\sinh^k \theta_y\cos^k \angle(xvy).
\end{aligned}
\end{align}

From a simple observation, $\cos\angle(xvy)=\cos(\pi-\angle(xv(-y)))=-\cos\angle(xv(-y))$. Hence, if $k$ is an odd number, then $\displaystyle{\int_{\mathbb{S}^{\dm-1}}\cos^k\angle(xvy) \d\sigma_y=0}$, and by integrating \eqref{npowercosh} over $\bbs^{\dm-1}$, we get
\[
\int_{\mathbb{S}^{\dm-1}}\cosh^m\theta_{xy}\d\sigma_y=\sum_{l=0}^{[m/2]}{m \choose 2l}\mathcal{A}_{l}\cosh^{m-2l}\theta_x\sinh^{2l}\theta_x\cosh^{m-2l}\theta_y\sinh^{2l}\theta_y,
\]
where
\[
\mathcal{A}_l:=\int_{\mathbb{S}^{\dm-1}}\cos^{2l}\angle(xvy) \d\sigma_y.
\]

Then, we have from \eqref{eqn:int-radial}:
\begin{align*}
&\tilde{\rho}(\theta_x)=\int_{B_v(R)}\sum_{m=0}^N b_m\cosh^m\theta_{xy}\tilde{\rho}(\theta_y) \d y\\
&=\sum_{m=0}^N b_m \int_0^R \left(\int_{\mathbb{S}^{\dm-1}}\cosh^m\theta_{xy}\d\sigma_y \right)\sinh^{\dm-1}\theta_y\tilde{\rho}(\theta_y) \d\theta_y\\
&=\sum_{m=0}^N b_m \int_0^R \left(\sum_{l=0}^{[m/2]}{m \choose 2l}\mathcal{A}_{l}\cosh^{m-2l}\theta_x\sinh^{2l}\theta_x\cosh^{m-2l}\theta_y\sinh^{2l}\theta_y \right)\sinh^{\dm-1}\theta_y\tilde{\rho}(\theta_y) \d\theta_y\\
&=\sum_{m=0}^N \sum_{l=0}^{[m/2]} b_m{m \choose 2l}\mathcal{A}_{l} \cosh^{m-2l}\theta_x\sinh^{2l}\theta_x\int_0^R \cosh^{m-2l}\theta_y\sinh^{2l+\dm-1}\theta_y \tilde{\rho}(\theta_y) \d\theta_y\\
&=\sum_{m=0}^N \sum_{l=0}^{[m/2]} b_m{m \choose 2l}\mathcal{A}_{l} \left(\int_0^R \cosh^{m-2l}\theta_y\sinh^{2l+\dm-1}\theta_y \tilde{\rho}(\theta_y) \d\theta_y\right)\cosh^{m-2l}\theta_x(\cosh^2\theta_x-1)^l.
\end{align*}

The calculation above implies that $\tilde{\rho}$ can be expressed as
\[
\tilde{\rho}(\theta)=\sum_{m=0}^N a_m\cosh^m\theta, \qquad \text{ for } 0 \leq \theta <R.
\]
Furthermore, this yields that
\begin{align}\label{recrel}
\sum_{m=0}^N a_m\cosh^m\theta_x=\sum_{k,m=0}^N \sum_{l=0}^{[m/2]} a_kb_m{m \choose 2l}\mathcal{A}_{l} \left(\int_0^R \cosh^{m-2l+k}\theta_y\sinh^{2l+\dm-1}\theta_y  \d\theta_y\right)\cosh^{m-2l}\theta_x(\cosh^2\theta_x-1)^l.
\end{align}

For the purpose of proving Proposition \ref{P5.1}, we only need the case $N=1$; nevertheless, we presented above the case of general $N$ for its own interest. Hence, from now on assume $b_m=0$, $m\geq 2$.

Consider three cases: 

\noindent {\em Case 1}: $b_0>0$ and $b_1=0$. In this case, \eqref{recrel} reduces to
\[
\sum_{m=0}^0 a_m\cosh^m\theta_x=\sum_{k=0}^0 a_kb_0\mathcal{A}_0\left(\int_0^R\cosh^k\theta_y\sinh^{\dm-1}\theta_y \d\theta_y\right),
\]
which implies $1=b_0\mathcal{A}_0\int_0^R\sinh^{\dm-1}\theta_y \d\theta_y$. Since $\mathcal{A}_0=|\bbs^{\dm-1}|$, we can determine $R$ uniquely. Also, we can determine $a_0=1$. We thus reached the equilibrium $\rho_c$ from \eqref{d-15}.\\

\noindent {\em Case 2}: $b_0=0$ and $b_1>0$. In this case, \eqref{recrel} reduces to
\[
\sum_{m=0}^1 a_m\cosh^m\theta_x=\sum_{k=0}^1 a_kb_1\mathcal{A}_0\left(\int_0^R\cosh^{k+1}\theta_y\sinh^{\dm-1}\theta_y \d\theta_y\right)\cosh\theta_x,
\]
which implies $1=\mathcal{A}_0b_1\int_0^R\cosh^2\theta_y\sinh^{\dm-1}\theta_y \d\theta_y$. One can then argue as in Case 1 to obtain the uniqueness of $R$, $a_0$ and $a_1$. This is the equilibrium solution from \eqref{d-8}. \\

\noindent {\em Case 3}: $b_0>0$ and $b_1>0$. In this case, \eqref{recrel} reads 
\begin{align*}
\sum_{m=0}^1 a_m\cosh^m\theta_x&=\sum_{k=0}^1 a_kb_0\mathcal{A}_0\left(\int_0^R\cosh^k\theta_y\sinh^{\dm-1}\theta_y\d\theta_y\right)\\
&\quad +\sum_{k=0}^1 a_kb_1\mathcal{A}_0\left(\int_0^R\cosh^{k+1}\theta_y\sinh^{\dm-1}\theta_y\d\theta_y\right)\cosh\theta_x.
\end{align*}
This yields
\begin{align*}
a_0&=\sum_{k=0}^1 a_kb_0\mathcal{A}_0\left(\int_0^R\cosh^k\theta_y\sinh^{\dm-1}\theta_y\d\theta_y\right)\\
&=b_0\mathcal{A}_0\left(
a_0\int_0^R \sinh^{\dm-1}\theta_y\d\theta_y+a_1\int_0^R\cosh\theta_y\sinh^{\dm-1}\theta_y\d\theta_y
\right)
\end{align*}
and
\begin{align*}
a_1&=\sum_{k=0}^1 a_kb_1\mathcal{A}_0\left(\int_0^R\cosh^{k+1}\theta_y\sinh^{\dm-1}\theta_y\d\theta_y\right)\\
&=b_1\mathcal{A}_0\left(
a_0\int_0^R\cosh\theta_y\sinh^{\dm-1}\theta_y\d\theta_y+a_1\int_0^R\cosh^2\theta_y\sinh^{\dm-1}\theta_y\d\theta_y
\right).
\end{align*}

From the two equations above we find
\begin{align}
\begin{aligned}\label{aratio}
a_0:a_1&=\left( b_0\mathcal{A}_0\int_0^R\cosh\theta_y\sinh^{\dm-1}\theta_y\d\theta_y \right): \left(1-b_0\mathcal{A}_0\int_0^R\sinh^{\dm-1}\theta_y\d\theta_y \right) \\[2pt]
&=\left( 1-b_1\mathcal{A}_0\int_0^R\cosh^2\theta_y\sinh^{\dm-1}\theta_y\d\theta_y \right): \left( b_1\mathcal{A}_0\int_0^R\cosh\theta_y\sinh^{\dm-1}\theta_y\d\theta_y \right).
\end{aligned}
\end{align}
Therefore, $R$ satisfies
\begin{align*}
&\left(b_0\mathcal{A}_0\int_0^R\cosh\theta_y\sinh^{\dm-1}\theta_y\d\theta_y\right)\left(b_1\mathcal{A}_0\int_0^R\cosh\theta_y\sinh^{\dm-1}\theta_y\d\theta_y\right)\\
&=\left(1-b_0\mathcal{A}_0\int_0^R\sinh^{\dm-1}\theta_y\d\theta_y\right)\left(1-b_1\mathcal{A}_0\int_0^R\cosh^2\theta_y\sinh^{\dm-1}\theta_y\d\theta_y\right),
\end{align*}
which yields
\begin{align*}
b_0b_1\mathcal{A}_0^2\left(\int_0^R \cosh\theta \sinh^{\dm-1}\theta \d\theta\right)^2&=1-b_0\mathcal{A}_0\int_0^R \sinh^{\dm-1}\theta \d\theta-b_1\mathcal{A}_0\int_0^R \cosh^2\theta \sinh^{\dm-1}\theta \d\theta\\
&\quad +b_0b_1\mathcal{A}_0^2\left(\int_0^R\sinh^{\dm-1}\theta \d\theta\right)\left(\int_0^R\cosh^2\theta_y\sinh^{\dm-1}\theta \d\theta\right).
\end{align*}

From now on, we only consider the case $\dm=2$. From the definition of $\mathcal{A}_l$, we get $\mathcal{A}_0=2\pi$. Then, from the equation above we get
\begin{align*}
4\pi^2 b_0b_1\left(\int_0^R \cosh\theta\sinh\theta \d\theta\right)^2&=1-2\pi b_0\int_0^R\sinh\theta \d\theta-2\pi b_1\int_0^R \cosh^2\theta\sinh\theta d\theta\\
&\quad +4\pi^2 b_0b_1\left(\int_0^R\sinh\theta d\theta\right)\left(\int_0^R\cosh^2\theta \sinh\theta \d\theta\right),
\end{align*}
which can be calculated as
\[
\pi^2 b_0b_1(\cosh^2R-1)^2=1-2\pi b_0(\cosh R-1)-\frac{2\pi b_1}{3}(\cosh^3R-1)+\frac{4\pi^2 b_0b_1}{3}(\cosh R-1)(\cosh^3 R-1).
\]
By some simplifications, we finally reach
\begin{align}\label{eqcoshR}
0=1-2\pi b_0(\cosh R-1)-\frac{2\pi b_1}{3}(\cosh^3R-1)+\frac{\pi^2 b_0b_1}{3}(\cosh R-1)^4.
\end{align}

We will prove that there exist unique $R$, $a_0$, and $a_1$, provided $b_0$ and $b_1$ are positive real numbers. Substitute $x=\cosh R-1$ into \eqref{eqcoshR} to get
\[
0=1-2\pi (b_0+b_1) x-2\pi b_1x^2-\frac{2\pi b_1}{3}x^3+\frac{\pi^2 b_0b_1}{3}x^4=:\fb(x).
\]
We want to find non-negative solutions of the above equation. We know that $\fb(0)=1$, $\fb'(0)=-2\pi(b_0+b_1)<0$, and
\[
\fb'(x)=-2\pi(b_0+b_1)-4\pi b_1 x-2\pi b_1x^2+\frac{4\pi^2 b_0b_1}{3}x^3.
\]

As $\fb'(0)<0$, and $\lim_{x\to\infty}\fb'(x)=\infty$, by the intermediate value theorem, there exists at least one positive solution $x_{+}$ such that $\fb'(x_{+})=0$. Now, we will show that the positive solution of $\fb'(x)=0$ is unique. Since the product of the three roots of $\fb'(x)=0$ is positive, if there exist more than two positive solutions then the other solution must also be positive. If $x_1, x_2, x_3$ are three positive solutions of $\fb'(x)=0$,  then $x_1x_2+x_2x_3+x_3x_1=-\frac{4\pi b_1}{4\pi^2 b_0b_1/3}=-\frac{3}{\pi b_0}<0$, which leads to a contradiction. We conclude that $\fb'(x)=0$ has a unique positive solution, which we denote by $x_+$. 

From the above, we infer that $\fb(x)$ is decreasing on $x\in[0, x_+]$ and increasing on $x\in [x_+, \infty)$. Finally, we use the intermediate value theorem to conclude that if $\fb(x_+)<0$ then there exist two distinct positive solutions of $\fb(x)=0$. We will now show that indeed, $\fb(x_+)<0$. From a simple calculation, we have
\begin{align*}
\fb\left(\frac{1}{2\pi(b_0+b_1)}\right)&=1-1-\frac{2\pi b_1}{(2\pi(b_0+b_1))^2}-\frac{2\pi b_1}{3(2\pi(b_0+b_1))^3}+\frac{\pi^2 b_0b_1}{3(2\pi(b_0+b_1))^4}\\
&=-\frac{b_1}{2\pi(b_0+b_1)^2}-\frac{b_1}{12\pi^2(b_0+b_1)^3}+\frac{b_0b_1}{48\pi^2(b_0+b_1)^4}\\
&<0-\frac{b_1}{12\pi^2(b_0+b_1)^3}+\frac{b_1}{48\pi^2(b_0+b_1)^3}<0.
\end{align*}
So we can conclude that $\displaystyle{\fb(x_+)\leq \fb\left(\frac{1}{2\pi(b_0+b_1)}\right)<0}$. 

Denote the two distinct roots of $\fb(x)=0$ by $\alpha_{b_0b_1}<\beta_{b_0b_1}$. There are two choices for the radius $R$: $\cosh R=\alpha_{b_0b_1}+1$ and $\cosh R=\beta_{b_0b_1}+1$. We will show that the latter cannot be possible.

As needed below, we calculate $\fb\left(\frac{1+\sqrt{1+4\pi b_0}}{2\pi b_0}\right)$. Set $\lambda:=\frac{1+\sqrt{1+4\pi b_0}}{2\pi b_0}$; then $\pi b_0\lambda^2=\lambda+1$. From this fact, we have
\begin{align*}
\fb\left(\lambda\right)&=1-2\pi(b_0+b_1)\lambda-2\pi b_1\lambda^2-\frac{2\pi b_1}{3}\lambda^3+\frac{\pi^2 b_0 b_1}{3}\lambda^4\\
&=1-2\pi(b_0+b_1)\lambda-2\pi b_1\lambda^2-\frac{2\pi b_1}{3}\lambda^3+\frac{\pi b_1 \lambda^2}{3}(\pi b_0\lambda^2)\\
&=1-2\pi(b_0+b_1)\lambda-2\pi b_1\lambda^2-\frac{2\pi b_1}{3}\lambda^3+\frac{\pi b_1 \lambda^2}{3}(\lambda+1)\\
&=1-2\pi(b_0+b_1)\lambda-\frac{5\pi}{3} b_1\lambda^2-\frac{\pi b_1}{3}\lambda^3\\
&<1-2\pi b_0 \lambda<0.
\end{align*}
Note that
\[
\alpha_{b_0b_1}<\frac{1}{2\pi(b_0+b_1)}<\lambda=\frac{1+\sqrt{1+4\pi b_0}}{2\pi b_0}<\beta_{b_0b_1}.
\]
By \eqref{aratio}, we have
\begin{equation}
\label{eqn:azao}
a_0:a_1=\left( \pi b_0(\cosh^2 R-1) \right): \left( 1-2\pi b_0(\cosh R-1) \right),
\end{equation}
which implies
\begin{equation}
\label{eqn:trho-1}
\tilde{\rho}(\theta)=C\left( \pi b_0(\cosh^2 R-1)+\left(1-2\pi b_0(\cosh R-1)\right)\cosh\theta \right), \qquad \text{ for } 0 \leq \theta <R.
\end{equation}
for some constant $C>0$.

Now, assume that $\cosh R=\beta_{b_0b_1}+1$. Then, we have
\[
\tilde{\rho}(0)=C\left(
\pi b_0\beta_{b_0b_1}(\beta_{b_0b_1}+2)+(1-2\pi b_0\beta_{b_0b_1})
\right)=C(\pi b_0\beta_{b_0b_1}^2+1),
\]
and
\begin{align*}
\tilde{\rho}(R)&=C\left(\pi  b_0(\cosh^2R-1)+(1-2\pi b_0(\cosh R-1))\cosh R\right)\\
&=C\left(\pi b_0\beta_{b_0b_1}(\beta_{b_0b_1}+2)+(1-2\pi b_0\beta_{b_0b_1})(\beta_{b_0b_1}+1)\right)\\
&=C(-\pi b_0 \beta_{b_0b_1}^2+\beta_{b_0b_1}+1).
\end{align*}

Since $\frac{1\pm\sqrt{1+4\pi b_0}}{2\pi b_0}$ are the two solutions of $-\pi b_0 x^2+x+1=0$ and $\beta_{b_0b_1}>\frac{1+\sqrt{1+4\pi b_0}}{2\pi b_0}$, then $-\pi b_0\beta_{b_0b_1}^2+\beta_{b_0b_1}+1<0$. However, this implies $\trho(R)<0$, which is not possible. We conclude that $\cosh R$ cannot be $\beta_{b_0b_1}+1$, but instead, $\cosh R=\alpha_{b_0b_1}+1$ ($R$ is unique).

In this case, by \eqref{eqn:azao}, we know that $a_0$ and $a_1$ have the same sign since $1-2\pi b_0(\cosh R-1)=1-2\pi b_0 \alpha_{b_0b_1}>1-\frac{b_0}{b_0+b_1}>0$. From \eqref{eqn:trho-1} we can express the equilibrium as
\[
\tilde{\rho}(\theta)=\begin{cases}
\displaystyle C\left(\pi b_0\alpha_{b_0b_1}(\alpha_{b_0b_1}+2)+(1-2\pi b_0\alpha_{b_0b_1})\cosh\theta\right),\qquad&\text{if }0\leq \theta\leq R,\\
0,\qquad&\text{if }\theta>R.
\end{cases}
\]
Here, $C$ should satisfy $1=2\pi \int_0^{\cosh^{-1}(\alpha_{b_0b_1}+1)}\tilde{\rho}(\theta)\sinh\theta d\theta$ (unit mass condition). So we have
\begin{align*}
\frac{1}{2\pi}&=\int_0^{\cosh^{-1}(\alpha_{b_0b_1}+1)}\tilde{\rho}(\theta)\sinh\theta \d\theta\\
&=C\int_0^{\cosh^{-1}(\alpha_{b_0b_1}+1)}\left(\pi b_0\alpha_{b_0b_1}(\alpha_{b_0b_1}+2)\sinh\theta+(1-2\pi b_0\alpha_{b_0b_1})\cosh\theta\sinh\theta\right) \d\theta\\
&=C\left(
\pi b_0 \alpha_{b_0b_1}^2(\alpha_{b_0b_1}+2)+(1-2\pi b_0\alpha_{b_0b_1})\frac{1}{2}\alpha_{b_0b_1}(\alpha_{b_0b_1}+2)
\right)\\
&=\frac{1}{2}C\alpha_{b_0b_1}(\alpha_{b_0b_1}+2).
\end{align*}
This implies that $C=\frac{1}{\pi\alpha_{b_0b_1}(\alpha_{b_0b_1}+2)}$.

Finally, derive
\[
\tilde{\rho}(\theta)=\begin{cases}
\displaystyle  b_0+\left(\frac{1-2\pi b_0\alpha_{b_0b_1}}{\pi\alpha_{b_0b_1}(\alpha_{b_0b_1}+2)}\right)\cosh\theta, \qquad&\text{if }0\leq \theta\leq \cosh^{-1}(\alpha_{b_0b_1}+1),\\
0,\qquad&\text{if }\theta>\cosh^{-1}(\alpha_{b_0b_1}+1).
\end{cases}
\]
\qed

\end{document}